\documentclass{amsart}
\usepackage[utf8]{inputenc}
\usepackage{color}

\usepackage[margin=1.0in]{geometry}

\usepackage{tikz}
\usetikzlibrary{patterns}
\usepackage{quiver}

\usepackage{amsmath}
\usepackage{amsthm}
\usepackage{amssymb}
\usepackage[matrix,arrow,curve]{xy}
\usepackage{enumitem} 
\usepackage{graphicx}
\usepackage{mdwlist}	
\usepackage{mathrsfs}
\usepackage{caption}
\usepackage{subcaption}
\usepackage{makecell}
\usepackage{ textcomp }
\usepackage{mathbbol}
\usepackage{longtable}

\newcommand\setItemnumber[1]{\setcounter{enum\romannumeral\@enumdepth}{\numexpr#1-1\relax}}

\newcommand{\p}{\mathbb{P}}
\DeclareMathOperator{\Bir}{Bir}
\DeclareMathOperator{\Aut}{Aut}
\DeclareMathOperator{\Bim}{Bim}
\DeclareMathOperator{\Hom}{Hom}
\DeclareMathOperator{\PGL}{PGL}
\DeclareMathOperator{\GL}{GL}
\DeclareMathOperator{\Lie}{Lie}
\DeclareMathOperator{\id}{Id}

\DeclareMathOperator{\SL}{SL}
\DeclareMathOperator{\PSL}{PSL}

\DeclareMathOperator{\Mat}{Mat}

\DeclareMathOperator{\Pic}{Pic}
\DeclareMathOperator{\NS}{NS}
\DeclareMathOperator{\Num}{N}
\DeclareMathOperator{\Sing}{Sing}
\DeclareMathOperator{\Spec}{Spec}

\DeclareMathOperator{\pr}{pr}
\DeclareMathOperator{\Exc}{Exc}
\renewcommand{\Spec}{\mathrm{Spec}}
\DeclareMathOperator{\Ind}{Ind}

\DeclareMathOperator{\End}{End}

\DeclareMathOperator{\bl}{bl}
\DeclareMathOperator{\diag}{diag}
\DeclareMathOperator{\sm}{sm}
\DeclareMathOperator{\Sym}{Sym}
\DeclareMathOperator{\ord}{ord}
\DeclareMathOperator{\alg}{alg}
\DeclareMathOperator{\tr}{tr}
\DeclareMathOperator{\trd}{trd}
\DeclareMathOperator{\Prd}{Prd}

\DeclareMathOperator{\codim}{codim}

\newcommand{\dto}{\dashrightarrow}

\newcommand{\X}{\mathcal{X}}
\newcommand{\A}{\mathcal{A}}
\newcommand{\cB}{\mathcal{B}}

\newcommand{\Y}{\mathcal{Y}}

\renewcommand{\L}{\mathcal{L}}

\newcommand{\E}{\mathcal{E}}

\newcommand{\cO}{\mathcal{O}}
\newcommand{\N}{\mathcal{N}}

\DeclareMathOperator{\Siegel}{\mathfrak{H}}
\DeclareMathOperator{\Nrd}{Nrd}

\newcommand{\be}{\mathbb{e}}
\newcommand{\Hm}{\mathbb{H}}
\newcommand{\ZZ}{\mathbb{Z}}
\newcommand{\RR}{\mathbb{R}}
\newcommand{\CC}{\mathbb{C}}

\newcommand{\DD}{\mathbb{D}}
\newcommand{\NN}{\mathbb{N}}
\newcommand{\PP}{\mathbb{P}}
\newcommand{\QQ}{\mathbb{Q}}
\newcommand{\HH}{\mathfrak{H}}

\newcommand{\GGmr}{\mathbb{G}_m^r}
\newcommand{\GGm}{\mathbb{G}_m}

\newcommand{\M}{\mathfrak{M}}

\newcommand{\me}{\mathrm{e}}

\newcommand{\Lau}{\mathbb{C}(\!(t)\!)}
\newcommand{\Pow}{\mathbb{C}[\![t]\!]}

\newcommand{\xdashrightarrow}[2][]{\ext@arrow 0359\rightarrowfill@@{#1}{#2}}

\newtheorem{Theorem}{Theorem}
\newtheorem{Corollary}[Theorem]{Corollary}
\newtheorem{conjecture}[Theorem]{Conjecture}

\newtheorem{theorem}[equation]{Theorem}
\newtheorem{lemma}[equation]{Lemma}
\newtheorem{proposition}[equation]{Proposition}
\newtheorem{corollary}[equation]{Corollary}

\newtheorem*{remark*}{Remark}

\theoremstyle{definition}

\newtheorem{example}[equation]{Example}
\newtheorem{remark}[equation]{Remark}

\makeatletter
\@addtoreset{equation}{section}
\makeatother

\begin{document}

\title{Families of automorphisms on abelian varieties}

\author{Charles Favre}
\address{CMLS, \'Ecole polytechnique, CNRS, Institut Polytechnique de Paris, 91128 Palaiseau Cedex, France}
\email{charles.favre@polytechnique.edu}
\author{Alexandra Kuznetsova}
\address{Steklov Mathematical Institute of Russian Academy of Sciences, Moscow, Russia}
\email{sasha.kuznetsova.57@gmail.com}

\thanks{The work of Alexandra Kuznetsova was performed at the Steklov International Mathematical Center and supported by the Ministry of Science and Higher Education of the Russian Federation (agreement no. 075-15-2022-265) as well as by the Russian Science Foundation, grant 21-11-00153. Both authors extend their thanks to Diego Izquierdo and Antoine Ducros for discussions on simple central algebras and analytic GAGA}

\date{\today}

\begin{abstract}
We consider some algebraic aspects of the dynamics of an automorphism on a family of polarized abelian varieties parameterized by the complex
unit disk. When the action on the cohomology of the generic fiber has no cyclotomic factor, we prove that 
such a map can be made regular only if the family of abelian varieties does not degenerate. 
As a contrast, we  show that families of translations are always regularizable. We further describe the closure of the orbits
of such maps, inspired by results of Cantat and Amerik-Verbitsky. 
\end{abstract}

\maketitle

\tableofcontents

\maketitle

\section{Introduction}

\noindent {\bf Regularizable mappings}.
Let $f\colon X\dto X$ be any birational self-map of a smooth complex projective variety of dimension $N\ge1$. Denote by
$\Ind(f)$ its indeterminacy locus, and by $\E(f)$  its exceptional locus consisting of 
the union of all the contracted hypersurfaces by $f$.

When $\Ind(f)= \Ind(f^{-1})= \emptyset$, 
then $f$ is an (biregular) automorphism, and its action by pull-back on the cohomology of $X$ commutes with iteration.
Powerful complex analytic techniques based on pluripotential theory have been developed and  under suitable
assumptions on the spectral properties of $f^*$, they
allow us to understand the ergodic properties of the dynamical system induced by $f$, see~\cite{MR2629598}.
When $\E(f)=\E(f^{-1})= \emptyset$, then $f$ is a \emph{pseudo-automorphism}. Since $f$ induces an isomorphism in codimension $1$, its action on $H^2(X,\QQ)$ is still functorial, 
but its action on the full cohomology ring $H^*(X,\QQ)$ remains a mystery. Furthermore pluripotential techniques can only work when $\Ind(f)$ and  $\Ind(f^{-1})$ are dynamically unrelated, see~\cite{MR2140266,MR2129771,MR2752759,MR3330918}. 
This motivates the question to measure how far a pseudo-automorphism is from being an automorphism. 

We say that a birational self-map $f\colon X\dto X$ is \emph{regularizable} if there exists a birational map $\phi \colon Y \dto X$
from a projective variety $Y$ such that $f_Y:= \phi^{-1} \circ f \circ \phi$ is an automorphism\footnote{we insist on $Y$ to be projective hence proper
contrary to some other authors, see, e.g., \cite{MR4340723}.}.
The existence of a functorial resolution of singularities (see, e.g.,~\cite{MR2289519}) implies that one can always assume $Y$ to be smooth. We shall also see that 
$f$ is regularizable if and only if $f^n$ is regularizable for some $n\in \NN^*$ (Proposition~\ref{prop:iterate-regular}). 
Similarly, we say that $f$ is \emph{pseudo-regularizable} when $f_Y$ is a pseudo-automorphism on some projective model $Y$. 
Our principal aim is to explore when a pseudo-regularizable birational map may be regularizable. 

\smallskip

Various geometric assumptions on $X$ force the regularizability of its whole group of birational transformations. This happens for instance
when $X$ does not carry any rational curve (since the total image of any point in $\Ind(f)$ is covered by rational curves); 
when $X$ is of general type (by Kobayashi-Ochiai theorem); or if $X$ is a degree $n\ge 4$ smooth hypersurface of $\PP^n_\CC$
(by birational super-rigidity, see, e.g.,~\cite{cheltsov}).
We shall focus our attention on dynamical properties of $f$ that ensure or forbid regularizability. A first set of constraints arises by looking at the growth of degrees of the iterates of $f$.
These degrees are defined as follows. For any ample line bundle $L\to X$, set $\deg_L(f) := f^* c_1(L) \cdot c_1(L)^{N-1}$. Then the sequence
$\{ \deg_L(f^n)\}$ is sub-multiplicative up to a bounded constant by~\cite{MR2119243,MR4048444,MR4133708}, and we can therefore define
the first dynamical degree of $f$ by setting $\lambda_1(f) := \lim_n \deg_L(f^n)^{1/n}$. 
By Weil's regularization lemma (see, e.g.,~\cite[\S 1.1.1]{MR4340488} and the references therein) it follows that if $\deg_L(f^n)$ is bounded, then $f$ is regularizable. 

One can similarly define higher degrees and dynamical degrees $\deg_{L,j}(f) := f^* c_1(L)^j \cdot c_1(L)^{N-j}$ and
$\lambda_j(f) := \lim_n \deg_{L,j}(f^n)^{1/n}$ for any $j\in \{0, \cdots, N\}$.
It is a fact that for any birational map $\phi \colon Y \dto X$ as above, and for any ample line bundle $L_Y\to Y$, 
we have $\deg_{L,j}(f^n) \asymp \deg_{L_Y,j}(f_Y^n)$.
This observation leads to the following series of results. 
\begin{enumerate}
\item
If $f$ is regularizable and $\lambda_1(f)=1$, then $\deg_L(f^n) \asymp n^{2(k-1)}$ with $k \in \{ 1, \cdots, N\}$, by~\cite{MR4431123} (the case $N=3$ was previously treated in~\cite{MR4030548}). 
\item 
When  $N=2$ and $\lambda_1(f)=1$, then $\deg_L(f^n) \asymp n^{k}$ with $k\in \{0,1,2\}$; and $f$ is regularizable 
if and only if  $k\in \{0,2\}$, by~\cite{MR563788,diller_favre}. 
\item
If $f$ is regularizable, then $\lambda_1(f)$ is a unit of the ring of integers of some number field (since it is the spectral radius of
$f^*\colon H^2(X,\QQ)\to  H^2(X,\QQ)$ which fixes the lattice $H^2(X,\ZZ)$). 
In particular, for any birational map $f$ of $\PP^N_\CC$ of degree $\ge2$ with $N\ge2$, 
and for a generic $A\in \PGL(N+1,\CC)$, then $A\circ f$ is not regularizable by a theorem by Vigny~\cite{MR3330918} (see also~\cite{MR4340488}). 
Also maps for which $\lambda_1(f)$ is transcendental cannot be regularizable (examples of such maps are given in ~\cite{bell-diller-jonsson}). 
\item 
When $N=2$ and $\lambda_1(f)>1$, then all Galois conjugates of $\lambda_1(f)$ have (complex) norm $\le1$ by~\cite{diller_favre}.
Moreover, if $f$ is regularizable, then either  $\lambda_1(f)$ is a quadratic unit or a Salem number. Conversely, 
if  $\lambda_1(f)$  is a Salem number, then $f$ is regularizable by~\cite{MR3454379}. 
\item
When $N\ge3$ and $\lambda_1(f)^2>\lambda_2(f)$, then again all Galois conjugates of $\lambda_1(f)$ have (complex) norm $\le1$ by~\cite{MR3255693}. 
LoBianco~\cite{MR4030548} proved
that the modulus of the Galois conjugates of $\lambda_1:=\lambda_1(f)$ of a regularizable map
belongs to the set $\{\lambda_1, \lambda_2^{-1}, \lambda_1^{-1} \lambda_2, \lambda_1^{-1/2}, \lambda_2^{1/2}, \lambda_1^{1/2} \lambda_2^{-1/2}\}$
 when $N=3$.  No general results are known in higher dimension. 
\end{enumerate}

Let us now discuss  how to produce (non-regular) examples of pseudo-automorphisms. 
First if $X$ has torsion canonical bundle, then $\E(f)$ and $\E(f^{-1})$ are automatically empty, hence
$f$ is a pseudo-automorphism. More generally if $\pi\colon X\to B$ is a family of polarized manifolds with trivial 
canonical bundle over a $1$-dimensional base and $\pi \circ f = \pi$, then one can find a birational model
of $X$ for which the relative canonical bundle $K_{X/B}$ is trivial, see~\cite{relative-minimal-model}, hence again $f$ is pseudo-regularizable. 

In dimension $2$ (where the notion of pseudo-automorphism and automorphism coincide), many constructions
of regularizable birational maps of $\PP^2_\CC$ with $\lambda_1(f)>1$ have been described, 
see~\cite{MR2354205,MR2677899,MR2858166,MR2825269,MR2905001,MR2904576,MR3498924} (the list is not exhaustive). 
Generalizations of these constructions to higher dimension $N\ge3$ lead to pseudo-regularizable birational maps, see~\cite{MR3179687,MR3161509,MR3587459}. 
To the knowledge of the authors, there is only one example by Oguiso and Truong~\cite{MR3329200} of a regularizable birational map $f$ of $\PP^N_\CC$
with $\lambda_1(f)>1$  which admits a Zariski dense orbit (see also~\cite{MR2904995} for product examples having Siegel disks). 
Even though pseudo-automorphisms are expected to be non regularizable in general, it remains a very challenging task to actually prove it. 
Examples of birational maps of $\PP^3_\CC$ have been treated by Bedford, Cantat and Kim~\cite{MR3277202}. 
The second author~\cite{Regularization} has also proved that a generic element of a family of pseudo-automorphisms introduced by Blanc 
 is not regularizable.

In this paper, we shall discuss the regularizability of a given bimeromorphic map $f$ on a $1$-parameter family of polarized abelian varieties $\X = (X_t)_{t\in \DD}$.
Note that by our previous discussion, $f$ is automatically pseudo-regularizable.  
Although we do not solve our problem in full generality, we give a necessary and sufficient condition
when the action of $f$ on $H^1(X_t,\ZZ)$ has no cyclotomic factors (Theorem~\ref{Theorem: criterion for irreducible family of AV}). 
We also exhibit necessary conditions in the case $\lambda_1(f)=1$ (Theorem~\ref{Theorem: decomposition}). 
We then turn to families of translations, and prove two theorems concerning this class of maps. First we show that families of translations are always regularizable (Theorem~\ref{Theorem: translations are regularizable}).  Second, we describe the closure of the orbits of a generic point
(Theorem~\ref{thm:orbit-closure}), inspired by a recent work of Amerik and Verbistky~\cite{arXiv:2112.01951}. 

\medskip
 %%%%%%%%%%%

\noindent {\bf Regularizability of bimeromorphic maps on families of abelian varieties}.
Before stating our main results, let us describe the set-up in more details. 
A  family of polarized abelian varieties of dimension $g$ over the complex unit disk is a proper holomorphic map 
 $\pi\colon \X \to \DD=\{|t|<1\}$  which is a submersion over $\DD^*=\{0<|t|<1\}$, such that $X_t= \pi^{-1}(t)$ is a smooth abelian variety of dimension $g$ for each $t\neq 0$. 
We shall also fix a relatively ample line bundle $\L \to \X$, and assume that $\X$ is smooth.

The family is \emph{smooth} when $\pi$ is a submersion over $\DD$. In that case,  $X_0$ is a smooth abelian variety.
A \emph{base change} of order $n$ for the family $\pi\colon \X \to \DD$ is a  family of polarized abelian varieties 
$\pi_n \colon \X_n \to \DD$ with a meromorphic map $\varphi\colon \X_n\dto \X$
making the following diagram commutative: 
\[\begin{tikzcd}[ampersand replacement=\&]
	\X_n \&  \X \\
	\DD \& \DD
	\arrow[from=1-1, to=1-2]
	\arrow[from=1-1, to=2-1]
	\arrow[from=1-2, to=2-2]
	\arrow[from=2-1, to=2-2]
	\arrow["{\varphi}", from=1-1, to=1-2]
	\arrow["{\pi}", from=1-2, to=2-2]
	\arrow["{\pi_n}"', from=1-1, to=2-1]
	\arrow["{t\to t^n}", from=2-1, to=2-2]
\end{tikzcd}\]
A  family of polarized abelian varieties $\pi'\colon \X' \to \DD$ is \emph{bimeromorphically equivalent} to  $\pi\colon \X \to \DD$
if there exists a bimeromorphic map $\varphi\colon \X' \dto \X$ which is a biregular isomorphism
over $\X^* = \pi^{-1}(\DD^*)$ and satisfies $\pi\circ \varphi = \pi'$. 

We shall say that a proper  family of polarized abelian varieties $\pi' \colon \X' \to \DD$ is \emph{not degenerating} if it admits a base change which is a smooth family. 
In other words,  the induced projective variety $X$ defined over the field $\Lau$ admits a smooth model over $\Spec\, \Pow$ (possibly after a finite extension of the base field). 
It is also equivalent to 
say that the canonical holomorphic map sending $t\in \DD^*$ to the 
isomorphism class\footnote{the space of such isomorphism classes is a quasi-projective variety, see~\cite[chapter 8]{Birkenhake-Lange_AV}.} of the polarized abelian variety $[X_t]$ extends holomorphically through $0$ after possibly precomposing with $t\mapsto t^n$.

Consider $f \colon \X \dto \X$ a bimeromorphic map satisfying $\pi \circ f = \pi$. Note that the indeterminacy locus of $f$ satisfies $\Ind(f) \subset X_0$, since an abelian variety
does not contain any rational curve. In this context, we say that $f$ is regularizable (resp. pseudo-regularizable) iff there exists a proper  family of polarized abelian varieties
 $\pi'\colon \X' \to \DD$ bimeromorphically equivalent to  $\pi\colon \X \to \DD$ through a
 bimeromorphic map $\varphi\colon \X' \dto \X$, such that $f_{\X'} := \varphi^{-1} \circ f \circ \varphi$ is a regular automorphism (resp. regular automorphism from 
 $X_1$ onto $X_2$ where $\X\setminus X_i$ are analytic subsets of codimension $\ge2$). 
 We insist here on the condition that $\X'$ should be proper. As we shall see below, any map is regularizable on its Néron model, but the central fiber 
 of the latter is in general non-compact. As above,  $f$ is always pseudo-regularizable. 
 
For each $t \neq 0$, the map $f$ induces a biregular automorphism $f_t \colon X_t \to X_t$, hence
a group isomorphism  $f_t^* \colon H^1(X_t, \ZZ) \to H^1(X_t, \ZZ)$. Since all fibers $X_t$ are smoothly diffeomorphic, any continuous path joining $t$ to $t'$
induces an isomorphism $H^1(X_t, \ZZ) \to H^1(X_{t'}, \ZZ)$ conjugating $f_t^*$ to  $f_{t'}^*$. 

\begin{Theorem}\label{Theorem: criterion for irreducible family of AV}
Let $\pi\colon \X \to \DD$ be a  family of polarized abelian varieties of dimension $g$, and let $f \colon \X \dto \X$ be any bimeromorphic map. 
\begin{enumerate}
\item
If the family is not degenerating, then $f$ is regularizable. 
\item  
Suppose $f$ is regularizable. If no root of unity is an eigenvalue of $f^*_t\colon H^1(X_t, \ZZ) \to  H^1(X_t, \ZZ)$, then 
the family is not degenerating.  
\end{enumerate}
\end{Theorem} 

The first statement is easy to prove. The main content lies in the second statement. 
Observe that if $X_t$ is a simple abelian variety for some $t\neq 0$, then the condition appearing in (2) is equivalent to $\lambda_1(f_t) >1$. 
Note also that our theorem is local over the base hence implies the next result.
\begin{Corollary}
Suppose $\X$ is a projective variety endowed with a fibration over a projective curve $\pi\colon \X\to B$ whose general fiber $X_b= \pi^{-1}(b)$ is an abelian variety. 
Suppose $f\colon \X\dto \X$ is a birational map preserving $\pi$. 
If $f^*_b\colon H^1(X_b, \ZZ) \to  H^1(X_b, \ZZ)$ has no eigenvalue equal to a root of unity for a general $b$, then $f$ is regularizable iff 
there exists a commutative diagram
\[\begin{tikzcd}[ampersand replacement=\&]
	X'\& X \\
	B'\& B
	\arrow[dashrightarrow, from=1-1, to=1-2]	
	\arrow["{\varphi}", from=2-1, to=2-2]
	\arrow["{\pi}", from=1-2, to=2-2]
        \arrow["{\pi'}", from=1-1, to=2-1]
\end{tikzcd}\]  
where $\varphi' \colon B' \to B$ is a finite ramified cover, and $\pi'\colon X'\to B'$ is a smooth family.
\end{Corollary}

Smooth non-isotrivial families of abelian varieties with automorphisms satisfying $\lambda_1(f) >1$ and parameterized by 
a projective curve exist for any $g\ge2$. This follows from the fact that Hilbert-Blumenthal varieties can be compactified by adding finitely 
many points, see \S\ref{sec:aut-no-cyclo} for details. 

\smallskip 

Let us explain our strategy for the proof of (2). We use the notion of Néron model, see, e.g.,~\cite{BLR_Neron_models}.
This is a (not necessarily proper) birational model $\N\to \DD$ that satisfies a suitable universal property so that 
$f$ induces a regular automorphism $f_\N\colon \N\to \N$, and the
 bimeromorphic map $\phi_\N \colon \X \dto \N$ is regular on the open subset $X^{\sm}\subset \X$ over 
 which $\pi$ is a local submersion. 
 By Grothendieck's semi-stable reduction theorem, after base change, the central fiber is a finite union
 of semi-abelian varieties, that is extension of an abelian variety by a multiplicative torus. 
 Further, the semi-stable reduction theorem (\cite{KKMS}), again after base change, 
 we may suppose that $\X\setminus X^{\sm}$ has codimension $\ge2$ in $\X$.
To simplify the argument assume that $g=2$  and $\lambda_1(f_t)>1$, and that $X_t$ is simple for some $t$. 
First we pick an irreducible component $E$ of $X_0$ that is $f$-invariant
 and satisfies $\lambda_1(f|_E) = \lambda_1(f_t)$ (Proposition~\ref{prop: component of X0 with maximal growth}). Then we look at its image $Z := \phi_\N(E)\subset \N$ inside the central fiber of the Néron model. 
When $\dim(Z) \le1$, we observe that $E$ is the exceptional divisor of the blow-up of a fixed point or a fixed curve in a neighborhood of which $f$ is a local isomorphism, 
 which implies  the dynamical degree of  $f_\N|_Z$ to be equal to $1< \lambda_1(f_t)$, a contradiction (see 
 Proposition~\ref{prop: dynamical degrees on the exceptional divisor}).  
It follows that $Z$ has dimension $2$, and is an extension of a multiplicative torus $\GGm^k$ by an abelian variety $A$
 of dimension $2-k$. When $k=0$, the family is not degenerating.  The proof is complete if 
 we prove that $\lambda_1(f_\N|_Z)<\lambda_1(f_t)$ when $k\in\{1,2\}$. 
 Roughly the reason is that given a matrix $M\in \SL(g,\ZZ)$ whose spectral radius is equal to $\rho>1$, then 
  the dynamical degree of the induced map on the product $E^g$ for any elliptic curve $E$
  is equal to $\rho^2$; whereas the dynamical degree of the map induced by $M$ on $\GGm^g$
  is equal to $\rho$ (Theorem~\ref{theorem: dynamical degrees on semi-abelian variety}).

The same line of arguments can be used in the case $\lambda_1(f_t)=1$ if one replaces the computation of dynamical degrees by 
the finer invariant given by the growth rate of the sequence of degrees. Recall an automorphism $f_t$ of an abelian variety $X_t$ with $\lambda_1(f_t) =1$ always satisfies
$\deg_1(f^n) \asymp n^{2k}$ for some $k\in \{0, \cdots, g-1\}$. 
When the degrees are bounded (i.e., $k=0$), then an iterate of $f_t$ is a translation. At the other side of the spectrum, 
when $\deg_1(f^n) \asymp n^{2(g-1)}$, then $X_t$ is isogenous to $E^g_t$ for some elliptic curve $E_t$. 
 
However, during the proof of Theorem~\ref{Theorem: criterion for irreducible family of AV} we used the product formula of Dinh-Nguyen~\cite[Theorem~1.1]{Dinh_Nguyen} that allows to compute exactly the dynamical 
degrees of a map preserving a fibration in terms of the dynamical degrees in the fiber and on the base. We were only able to extend this formula to 
 the degree growth in codimension $1$ with linear losses, so that we couldn't fully characterize regularizable mappings.

To state our next result, we define the following invariant. Let $r(\X)\in \{0, \cdots, g\}$ be the dimension of the maximal multiplicative torus of $\N_0$ when it is semi-abelian (this does not depend on the base change).
Observe that $r(\X)= 0$ if and only if $\X$ is not degenerating; and $r(\X) =g$ if and only if $\X$ has maximal degeneration (in the sense that all components of the central fiber
are multiplicative tori). 

\begin{Theorem}\label{Theorem: decomposition}
Let $\pi\colon \X \to \DD$ be a  family of polarized abelian varieties of dimension $g$, and let $f \colon \X \dto \X$ 
be any bimeromorphic map such that $\lambda_1(f_t)=1$. 
Write $\deg_1(f_t^n) \asymp n^{2k}$ with $k\in \{0, \cdots, g-1\}$. 

If $f$ is regularizable, then we have
\[2k \le \max \{ r(\X) , 2g - 2r(\X) -1\}~.\]
In particular, if $f$ is regularizable and $k=g-1\ge 2$, then $r(\X)=0$ and the family is not degenerating. 
\end{Theorem} 

For abelian surfaces ($g=2$), the following table summarizes our results:
\begin{table}[ht]
\centering 
\begin{tabular}{|c|c|c|}
  \hline
 $k$ & $r(\X)$ & Is $f$ regularizable?
 \\
 \hline
 $0$ &   & yes (Theorem~\ref{Theorem: translations are regularizable})
 \\
$1$ & $0$ & yes 
\\ 
$1$ & $1$ & no (Theorem~\ref{Theorem: decomposition})
\\
$1$ & $2$ & ???
 \\  \hline
 \end{tabular}
\end{table}

Observe also that in any dimension, if the degeneration is maximal  
and $f$ is regularizable, then  our theorem says  $k \le \lfloor g/2 \rfloor$. 
We suspect that if $\lambda_1(f_t)=1$, $f$ is regularizable and $r(\X) \le 2k -1$,
then there exists an $f$-invariant subfamily $\Y\subset \X$ of polarized abelian varieties
of dimension $\le g- r(\X)$ which is not degenerating and such that $\deg_1(f^n|_\Y) \asymp \deg_1(f^n)$.
We propose the following more ambitious conjecture.
\begin{conjecture}
Let $f\colon \X\dto \X$ be any bimeromorphic map of 
a  family of polarized abelian varieties that fixes the $0$-section. 

Then $f$ is regularizable if and only if the following holds. 
Possibly after base change, there is a $f$-invariant splitting such that $\X = \Y \times \Y'$ 
where $\Y$ is a non-degenerating family of polarized automorphisms of abelian varieties; 
$\Y'$ is a  family of polarized abelian varieties, and 
$f|_{\Y'}$ is a family of automorphisms such that 
$f^*_t|_{\Y'}\colon H^1(Y'_t, \ZZ)\to H^1(Y'_t, \ZZ)$ has finite order. 
\end{conjecture}

\noindent {\bf Dynamics of families of translations}.
A translation on an abelian variety $X$ is an automorphism $f\colon X \to X$ acting as the identity on $H^1(X,\ZZ)$. 
Note that this is the same as asking $f$ to belong to the connected component of the identity in $\Aut(X)$.
Let $f \colon \X \dto \X$ be a bimeromorphic map on a family of polarized abelian varieties. 
If $f_t$ is a translation for some $t$, then $f_t$ is a translation for all $t\neq0$. 
\begin{Theorem}\label{Theorem: translations are regularizable}
Let $\pi\colon \X \to \DD$ be a  family of polarized abelian varieties, and let $f \colon \X \dto \X$ be a bimeromorphic map. 
If $f^*_t\colon H^1(X_t, \ZZ) \to  H^1(X_t, \ZZ)$ has finite order, then $f$ is regularizable. In particular, 
any family of translations is regularizable. 
\end{Theorem} 
The proof we give relies on the description by Nakamura~\cite{Nakamura_Neron_models}
of a compactification of the Néron model in terms of toroidal data, following ideas developed by Mumford (see the appendix of~\cite{MR1083353}). 
In a nutshell, one represents $X_t$ as a quotient of $\GGm^g$ by a co-compact action of $ \ZZ^g$, and
one obtains the suitable compactification by building a suitable toric modification of $\GGm^g\times \DD$ 
along the central fiber so that the action of $\ZZ^g$ extends and remains co-compact (on each fiber including the central one). 
It is then not difficult to check that any family of translations extends to this model. 

\begin{remark}
We use in an essential way the fact that the base is one-dimensional. 
We suspect that there exist families of elliptic curves defined over a surface
that carry families of translations which are not regularizable.
\end{remark}

Our next result is inspired by a recent theorem of Amerik and Verbitsky~\cite[Theorem~3.12]{arXiv:2112.01951}, who generalized to hyperkähler manifolds a former result of 
Cantat~\cite[Proposition~2.2]{MR1854708} on the density of generic orbits of translations on families of elliptic curves covering a K$3$ surface. 

Let $f\colon \X\dto \X$ be a family of translations as in the previous theorem.  
Note that for any given point in $X_t$, the closure (for the euclidean topology) of the orbit of $f_t$ is a finite union of real tori
 of (real) dimension ranging from $0$ to $2g$, and that this dimension does not depend on the choice of point on $X_t$.
 When $t$ changes however, this dimension might jump.     
 
 \begin{Theorem}\label{thm:orbit-closure}
Let  $f \colon \X \dto \X$ be any family of translations on a polarized family of abelian varieties $\pi\colon \X \to \DD$. 
Then there exist families of polarized abelian varieties  $\Y \to \DD$, $\Y' \to \DD$; 
families of translations $f_\Y \colon \Y\to \Y$, and $f_{\Y'}\colon \Y'\to \Y'$; 
and a meromorphic map $\phi \colon \Y\times \Y' \dto \X$ such that 
\begin{enumerate}
\item
$\phi (f_{\Y,t}, f_{\Y',t}) = f_t \circ \phi$ for all $t\in\DD$;
\item 
$\phi_t\colon Y_t\times Y'_t\to X_t$ is an isogeny for all $t\neq0$;
\item 
for all $t$ in a dense set of full measure, the closure of the orbit of any point $p\in Y_t$
under $f_{\Y,t}$ is dense in $Y_t$ for the euclidean topology;
\item
there exists a sequence of families of translations of finite order  $g_n$ on $\Y'$ such that 
$g_n \to f_\Y$ locally uniformly on $\Y'$. 
\end{enumerate}
\end{Theorem} 
In rough terms, up to isogeny, we can split $\X$ into a product family $\Y\times \Y'$ so that 
orbits are dense in $Y_t$ for $t$ generic, and the family of translations on $\Y'$ is uniformly approximated
by finite order automorphisms $g_n$. Note that for each $n$ and for a generic $t\in\DD^*$, the closure of the orbit of a point in $X_t$ under $(f_\Y,g_n)$ is
a finite union of translates of $Y_t$.

Let us state the following direct consequence of the previous result.
\begin{Corollary}
Suppose $f\colon \X\dto \X$ is a family of translations on a family of polarized abelian varieties $\pi\colon \X\to B$
where $B$ is a projective curve. If $X_t$ is simple for some $t$ and $\deg(f^n) \to \infty$, then
the orbit of any point in a generic fiber $X_t$ is dense for the euclidean topology. 
\end{Corollary}
Any automorphism $f\colon \X\to \X$ on a projective surface with $\deg(f^n) \asymp n^2$ is a family of 
translations on a family of elliptic curves, by Gizatullin's theorem, see~\cite{MR563788,diller_favre,MR3161509,MR3480704}. 
The previous result says that the orbit of $f$ on a generic fiber is dense, extending
~\cite[Proposition~2.2]{MR1854708} to any projective surface.

\smallskip

Our strategy for the proof of Theorem~\ref{thm:orbit-closure} can be described as follows. 
The euclidean closure of the orbit of a point in the fiber $X_t$ under a translation is a real torus that contains a translate
of a subabelian variety of maximal dimension $A_t$. For $t$ generic, the dimension of $A_t$ remains the same
and $A_t$ forms a family of subabelian varieties. By Poincaré's irreducibility theorem, up to isogeny, we can split $\X$ as a product of
two families of abelian varieties. On one factor, the orbits are dense for the euclidean topology.  On the other factor, 
the closure of the orbits are totally real. Exploiting the fact that holomorphic maps taking values in totally real submanifolds
are necessarily constant, we show that in this situation translations can be approximated uniformly
over the base by translations of finite order. 

\medskip

 \noindent {\bf Further remarks}. 
 Various generalizations of Theorem~\ref{Theorem: criterion for irreducible family of AV} can be investigated. 
 First we may consider bimeromorphic maps on $\X$ such that the action on the base 
 has infinite order. Note that in that case, the family of abelian varieties needs to be isotrivial. We may also replace $\DD$ by a higher dimensional base $\DD^k$, $k\ge2$ and
 ask whether regularizability along any germ of analytic disk implies regularization over $\DD^k$. 
 
Finally we may consider algebraic versions of our results. Let $R$ be a discrete valuation ring, and denote by $K$ its fraction field. 
Let $A$ be an abelian variety over $K$, and  $f \colon A \to A$ be any automorphism. 
Observe that Néron models exist in this degree of generality by~\cite{BLR_Neron_models}. We may thus ask the following question. 
 Suppose that the action of $f$ on the first étale cohomology group of $A$ has no eigenvalue equal to a root of unity. 
Is it true that the Néron model of $A$ is proper if and only if there exists a proper model of $A$ over $\Spec R$ to which $f$ extends
as an automorphism?

\medskip

\noindent {\bf Plan of the paper}.
In \S\ref{sec:growth}, we briefly review the basic properties of dynamical degrees, and prove 
Propositions~\ref{prop: component of X0 with maximal growth},~\ref{prop: dynamical degrees on the exceptional divisor} and~\ref{prop: dynamical degree on invariant subvariety} which are the keys to the proof of Theorem~\ref{Theorem: criterion for irreducible family of AV}. We conclude this part by discussing general properties
of regularizable maps. 

We discuss in \S\ref{sec:semi-abelian} the geometry of semi-abelian varieties and generalize the computation of dynamical degrees of morphisms 
on toric varieties done in~\cite{MR3043585,MR3059849} to semi-abelian varieties. 

We briefly recapitulate in \S\ref{sec:neron} the notion of Néron model, 
and then give the proof of our characterizations of regularizable maps on families of abelian varieties
(Theorems~\ref{Theorem: criterion for irreducible family of AV} and~\ref{Theorem: decomposition}). 

Section~\ref{sec:examples-family} can be read independently from the rest of the paper. It aims at describing interesting examples of families
of automorphisms on polarized abelian varieties, and to classify (up to isogeny and base change) such families
in dimension $\le 5$. The material is standard here, and goes back to the work of Shimura~\cite{MR156001}. 

Finally, we devote Section~\ref{sec:translations} to families of translations of abelian varieties, proving Theorems~\ref{Theorem: translations are regularizable} and~\ref{thm:orbit-closure}.
The crucial ingredient is the construction (due to Mumford) of an adequate compactification of the Néron model, 
that we recall in \S\ref{sec:relative} following closely~\cite{Nakamura_Neron_models}.

%%%%%%%%%%%%%%%%%%%%%%%%%%%%%%%%%%

\section{Growth of degrees of algebraic maps on quasi-projective varieties}\label{sec:growth}

In this section, we extend the notion of dynamical degrees to any rational self-map of a quasi-projective variety, and collect
a couple of results on the behaviour of dynamical degrees under restriction. 

Given any two sequences of positive real numbers $u_n, v_n$, we write $u_n \asymp v_n$ (resp. $u_n\lesssim v_n$) if and only if
there exists a constant $C>1$ such that $C^{-1} v_n \le u_n \le C v_n$ (resp.  $v_n \le C u_n$).

\subsection{Dynamical degrees} 
Let $X$ be a (possibly singular) quasi-projective complex variety of dimension $N$, and $f\colon X \dto X$ be
any dominant rational self-map. 

Let~$\overline{X}$ be any projective variety containing $X$ as a Zariski dense subset. 
Fix any ample line bundle $L\to \overline{X}$. Denote by $\bar{f}$ the unique dominant rational self-map of $\bar{X}$
whose restriction to $X$ equals $f$. This map is determined by its graph $\Gamma \subset \bar{X} \times \bar{X}$
which is an irreducible variety such that the first projection $\pi_1\colon \Gamma \to \bar{X}$ is a birational morphism, 
and the second projection $\pi_2\colon  \Gamma \to \bar{X}$ is surjective, and $f(x) = \pi_2(\pi_1^{-1}(x))$ for any $x$ off the 
image under $\pi_1$ of its exceptional locus.
For any $i\in\{0, \cdots, N\}$, we set 
\[ 
\deg_{L, i}(\bar{f}) =  \pi_2^*(c_1(L))^i\wedge\pi_1^*(c_1(L))^{N-i}\in \mathbb{N}^*.
\]
It was proven by Dinh and Sibony~\cite{Dinh_Sibony} (see also~\cite{MR4048444} for an argument working in arbitrary characteristic, and~\cite{MR4133708} for an estimate on the optimal constants), that given any birational map $\varphi\colon X_1 \dto \bar{X}$ between projective varieties, 
and for any ample line bundle $L_1 \to \bar{X}_1$
there exists a constant $C>1$ (depending only on $\varphi, L$ and $L_1$) such that
 \begin{equation}\label{eq:birin}
  C^{-1} \deg_{L_1, i}(f_1) 
\le \deg_{L, i}(\bar{f}) \le C \deg_{L_1, i}(f_1) 
\end{equation}
where $f_1 := \varphi^{-1}\circ \bar{f} \circ \varphi$. 
In particular, the growth rate of the sequence $\{\deg_{L, i}(\bar{f}^n)\}_{n\in \NN}$ depends
neither on the compactification of $X$, nor on the choice of $L$. 
To simplify notation, we shall thus write $\deg_i(f^n)$ instead of $\deg_{L, i}(\bar{f}^n)$. Beware that this sequence
is only defined up to bounded uniform constants. 

Dinh and Sibony further proved the existence of a positive constant $C>0$ such that
$C\deg_i(f^n)$  is submultiplicative.
We may thus define the $i$-th dynamical degree by setting: 
\[\lambda_i(f) := \lim_{n\to \infty} \deg_i(f^n)^{1/n} \in [1, +\infty[~.\] 
By~\eqref{eq:birin}, $\lambda_i(f)$ is invariant under birational conjugacy.
Note that when $f$ is a birational automorphism, then $\lambda_0(f)= \lambda_N(f) = 1$. 
In general, Khovanskii-Teissier's inequalities imply $\lambda_{i-1}(f) \lambda_{i+1}(f) \ge \lambda_i(f)^2$
so that the dynamical degrees are log-concave. It follows that  $\lambda_{i}(f)=1$
for all $i$ if and only if $\lambda_1(f) =1$ (in which case $f$ is necessarily birational). 

Suppose now that $X$ is smooth, and take $\bar{X}$  a smooth projective variety as above, together with an ample line bundle
$L\to \bar{X}$. We may endow $L$ with a positive hermitian form so that its curvature form $\omega$
is a Kähler form on $\bar{X}$. If $f$ is moreover regular, then we may compute the degrees as an integral: 
\[ 
\deg_{i}(f) = \int_{X} f^*(\omega^i)\wedge \omega^{N-i}~.
\]
One can also
compute the degrees directly by looking at the induced linear action on Dolbeault cohomology
$f^*_i \colon H^{i,i}(\bar{X}) \to H^{i,i}(\bar{X})$. For any given norm on $H^{i,i}(\bar{X})$, there exists a positive constant 
$C>1$ such that 
\begin{equation}\label{eq:spec}
C^{-1}\,  \| f^*_i\|\le \deg_{i}(f) \le C\,  \| f^*_i\|~.
\end{equation}
One can alternatively work on the real subspace $\Num^i_{\RR}(\bar{X})$
of $H^{i,i}(\bar{X})$ spanned by fundamental classes of subvarieties of codimension $i$.
Again, we have
$C^{-1}\,  \| f^*_i|\Num^i_{\RR}(\bar{X})\|\le \deg_{i}(f) \le C\,  \| f^*_i|\Num^i_{\RR}(\bar{X})\|$
for uniform bounded constant. 
When $f$ is a regular map, 
this implies $\lambda_i(f) = \rho(f_i^*) = \rho(f^*_i|\Num^i_{\RR}(\bar{X}))$
where $\rho(u)$ denotes the spectral radius of a linear map $u$. 

We refer to~\cite{MR4276288} for an interpretation of dynamical degrees in terms of spectral radii of bounded operators on suitable
Banach spaces.

\subsection{Degree growth and restrictions}
In this section, we prove three results that allow one to compare the dynamical degrees of an automorphism
with  the dynamical degrees of its restriction to a subvariety. These results play a key role in the proof of Theorem~\ref{Theorem: criterion for irreducible family of AV}.

\begin{proposition}\label{prop: component of X0 with maximal growth}
 Let $\pi\colon \X\to \DD$ be a flat projective morphism of relative dimension $N$, such that $X_t = \pi^{-1}(t)$ is connected for $t\neq0$. 
 Suppose that $f\colon \X \to \X$ is a biholomorphism, for which $\pi\circ f = \pi$ and
 $f(E)=E$ for any irreducible components $E$ of~$\X_0$. 
 
 Then for any $0\leqslant i \leqslant N$,
\[ \deg_i(f^n|_{X_t}) \asymp \max_E \deg_i(f^n|_E)~,
 \]
 where $E$ ranges over all irreducible components $E$ of $\X_0$.
 \end{proposition}
 \begin{remark}
 It is crucial for $f$ to be holomorphic on $\X$. Examples of a bimeromorphic map $f\colon \X \dto \X$
 such that $\pi\colon \X\to \DD$ is a smooth family of K$3$ surfaces and $\Ind(f) \subset X_0$
 such that  $\lambda_1(f_0)< \lambda_1(f_t)$ for all $t\neq0$ is given in~\cite{MR3748233}.
\end{remark}
\begin{proof}
Let $\L \to \X$ be any relatively ample line bundle. We have $\deg_i(f^n|_{X_t})= (f^{n*}c_1(\L))^i \cdot c_1(\L)^{N-i} \cdot [X_t]$. 
As $[X_t]$ is cohomologous to the divisor $\pi^{-1}(0) = \sum b_E [E]$, we get
\begin{equation}\label{eq:123}
 \deg_i(f^n|_{X_t}) = 
 \sum_E b_E \deg_i(f^n|_E)~.
 \end{equation}
Since $f\colon X_t \to X_t$ is regular, by~\eqref{eq:spec} there exists an integer $\delta$ such that 
$ \deg_i(f^n|_{X_t}) \asymp n^\delta \lambda_i(f|_{X_t})^n$. Similarly, for each irreducible component $E$ of $\X_0$, we may find an integer 
$\delta_E$ such that 
$\deg_i(f^n|_E)\asymp n^{\delta_E} \lambda_i(f|_E)^n$. The result follows from~\eqref{eq:123}.
\end{proof}

Suppose $f\colon X \dto Y$ is a birational map between smooth projective varieties.
Recall that $\Ind(f)$ denotes the indeterminacy locus of $f$, and is a subvariety of codimension at least $2$. 
The set of points $p\notin \Ind(f)$ at which $f$ is not a local biholomorphism is a divisor by the jacobian criterion
which coincides with the exceptional locus defined in the introduction and denoted by $\Exc(f)$.

 We define the proper transform of an irreducible subvariety $Z\subset X$ such that $Z\not\subset \Ind(f)$
 by setting $f^{\vee}(Z):= \overline{f(Z\setminus \Ind(f))}$. It is also an irreducible subvariety, 
 of the same dimension as $Z$ when $Z\not\subset \Exc(f)$.  

The total transform of a closed analytic subset $Z\subset X$ is defined as
$f(Z) := \pi_2 ( \pi_1^{-1}(Z))$ where $\pi_1\colon \Gamma \to X$ and 
$\pi_2\colon \Gamma \to Y$ are the canonical projections of the graph of $f$ onto $X$ and $Y$
respectively. In general $f(Z)$ strictly contains $Z$, and $p\in \Ind(f)$ if and only if $f(p)$
has positive dimension.

\begin{proposition}\label{prop: dynamical degrees on the exceptional divisor}
 Let $X, Y$ be smooth projective varieties and $\alpha\colon Y\dashrightarrow X$ be a birational map.
 Let $f\colon X\dto X$ be a birational self-map and $E\subset Y$ be an irreducible hypersurface such that:
 \begin{enumerate}
  \item[\textup{(1)}] $f_Y^\vee(E)  = E$, where $f_Y = \alpha^{-1}\circ f\circ\alpha\in \Bir(Y)$;
  \item[\textup{(2)}] and $Z = \alpha^\vee(E)$ is not included in $\Ind(f^{-1})$.
 \end{enumerate}
Then $Z$ is not included in $\Ind(f)$, $f^\vee(Z) = Z$, and   
  we have:
 \begin{equation}\label{eq:comput}
  \lambda_i(f_Y|_E) = \max_{\max\{0, i-c+1\}\leqslant j \leqslant \min\{i, N-c\}} \lambda_j(f|_Z)~,
 \end{equation}
 where $N= \dim(X)$ and $c= \codim(Z)$.
 
Moreover there exists a positive constant $C>1$ such that  
\begin{equation}\label{eq:compare1}
C^{-1} \deg_1(f^n|_Z)
\le
\deg_1(f_Y^n|_E)
\le
C n\times \deg_1(f^n|_Z)
\end{equation}
\end{proposition}
Observe that (1) requires in particular that $E$ is not a component of $\Exc(f_Y)$. Note also that 
the result follows from the invariance of dynamical degrees under birational conjugacy when $E$ is not included in $\Exc(\alpha)$.

\begin{example}
Consider the map  $f(x,y,z) = ( x,zy, z)$ in $\CC^3$. It defines a birational self-map of $Y:= (\PP^1_{\CC})^3$
which fixes pointwise  the rational curve $Z=\{x=y=0\}$ so that $\deg_1(f^n|_Z)=1$ for all $n$.

Let $\mu \colon X \to Y$ be the blow-up along $Z$. 
Then $E= \mu^{-1}(Z)$ is isomorphic to $ (\PP^1_{\CC})^2$, and in the coordinates $x,z$
the lift of $f$ to $E$ can be written as follows $f(x,z) = (x/z,z)$ so that $\deg_1(f^n|E)\asymp n$.
\end{example}

\begin{proof}
We prove first that $Z\not\subset \Ind(f)$. Since $f_Y^\vee(E)  = E$ is a divisor, we may find a point $p\in E\setminus \Ind(\alpha) \cup \Ind(f_Y)\cup \Exc(f_Y)$
such that $f_Y(p) \notin  \Ind(\alpha) \cup \Ind(f_Y)\cup \Exc(f_Y)$. The total image of $\alpha(p)$ under $f$ is included
in $ \alpha ( f_Y(p))$ which has dimension $0$. Hence $\alpha(p)\notin \Ind(f)$, and  $f(\alpha(p))=\alpha ( f_Y(p))$. 
It follows that $f^\vee(Z)\subset Z$. Since $E$ is irreducible, $Z$ is irreducible too, and $f_Y^\vee(E)  = E$ implies $f^\vee(Z)= Z$.

Next we prove~\eqref{eq:comput}. Recall that dynamical degrees are invariant under birational conjugacy. By Hironaka resolution of singularities, we may thus suppose that $Z$ is smooth, $\alpha$ is a regular map for which the sheaf of ideals defining $Z$ lifts to a locally principal ideal sheaf. 
By the universal property of blow-ups, we may thus write $\alpha= \bl_Z\circ \beta$, where $\bl_Z\colon Y' \to X$ is the blow-up
of $Z$ with exceptional locus $E_Z$, and $\beta\colon Y \to Y'$ is a birational map such that $\beta(\Exc_\beta)$ does not contain $E_Z$.
Note that $\lambda_i(f|E)= \lambda_i(f_{Y'}|E_Z)$ where $f_{Y'}\colon Y'\to Y'$ is the lift of $f$ to $Y'$.

We abuse notation and write $\bl_Z \colon E_Z  \to Z$  for the restriction of $\bl_Z$ to $E_Z$.
The morphism $\bl_Z \colon E_Z  \to Z$ is then the projectivization of the normal bundle $N_{Z/Y}\to Z$. 
The relative dimension of $\bl_Z$ is equal to 
$c-1$, and the dynamical degrees can be computed
using the product formula of Dinh-Nguyen~\cite[Theorem~1.1]{Dinh_Nguyen}: 
\[
\lambda_i(f_{Y'}|_{E_Z}) = \max_{\max\{0, i- c+1\}\leqslant j \leqslant \min\{i, N-c\}} \lambda_j(f|_{Z}) \lambda_{i-j}(f_{Y'}|\bl_Z)~,
\]
where 
\[\lambda_{j}(f_{Y'}|\bl_Z)= \lim_{n\to \infty} \left(f_{Y'}^{n*}c_1(L)^j \cdot c_1(L)^{N-c+1-j}\cdot [\bl_Z^{-1}(t)]\right)^{1/n}~,\]
for a generic $t\in Z$ (here we abuse notation and choose an arbitrary ample line bundle $L\to Y$). The existence of the limit is justified in op. cit. (see also~\cite{MR4133708}).

We claim that $\lambda_{j}(f_{Y'}|\bl_Z)=1$ for all $j$. Grant this claim. Then, we have 
\[
\lambda_i(f_{Y'}|_{E_Z}) = \max_{\max\{0,  i- c+1\}\leqslant j \leqslant \min\{i, N-c\}} \lambda_j(f|_{Z})~,
\]
and the formula~\eqref{eq:comput} follows.

It remains to prove the claim. We argue using complex analytic arguments. Note that since $f_Y^\vee(E)=E$, $f_Y$ is a local biholomorphism
at a generic point in $E$. It follows that $f$ is also a local biholomorphism at a generic point
$p\in Z$. The map $f_{Y'} \colon E_Z \dto E_Z$ hence
maps the fiber $\bl_Z^{-1}(p)\simeq\PP^{c-1}_{\CC}$ to $\bl_Z^{-1}(f_Y(p))$
and has degree $1$. This implies
\[
f_{Y'}^{n*}c_1(L)^j \cdot c_1(L)^{N-c+1-j}\cdot [\bl_Z^{-1}(t)] = c_1(L)^{N-c+1}\cdot [\bl_Z^{-1}(t)]
\]
concluding the proof that $\lambda_{j}(f_{Y'}|\beta)=1$.
 
 \smallskip
 
 We finally prove~\eqref{eq:compare1}. To that end, it suffices to show   
 \[C^{-1} \deg_1(f_Y^n|_{Z})
\le
\deg_1(f_{Y'}^n|_{E_Z})
\le
C n\times\deg_1(f_Y^n|_{Z})\]
for some constant $C>1$.
 Observe that the Néron-Severi space $\Num^1_{\RR}(E_Z)$ equals the direct
 sum $\bl_Z^* \Num^1_{\RR}(Z) \oplus \RR \xi$ where $\xi$ is the first Chern class of the tautological line bundle (whose restriction to each fiber
 of $\bl_Z$ is $O(1)$).  
 Note that we have the following commutative diagram: 
 \[\begin{tikzcd}[ampersand replacement=\&]
	E_Z\& E_Z \\
	Z\& Z
	\arrow["f_{Y'}", dashrightarrow, from=1-1, to=1-2]
	\arrow["{\bl_Z}", from=1-2, to=2-2]
	\arrow["{\bl_Z}"', from=1-1, to=2-1]
	\arrow["f_Y", dashrightarrow, from=2-1, to=2-2]
\end{tikzcd}\]  
No hypersurface in $E_Z$ is contracted by $\bl_Z$ into
the indeterminacy locus of  $f_Y$ since $\bl_Z$ is a fibration and 
$\codim (\Ind(f_Y))\ge 2$. It follows that for any class $\omega \in \Num^1_{\RR}(Z)$
we have $( f_Y \circ \bl_Z)^* \omega = \bl_Z^*(f_Y^*(\omega))$, see, e.g.,~\cite[Lemma~2.3]{Regularization}.
In a similar way, $\bl_Z$ is regular hence  $(\bl_Z\circ f_{Y'})^* \omega = f_{Y'}^*(\bl_Z^*(\omega))$, and 
we conclude that $ f_{Y'}^*(\bl_Z^*(\omega)) = \bl_Z^*(f_Y^*(\omega))$ for any $\omega \in \Num^1_{\RR}(Z)$. 
Note that this already implies: 
\[\deg_1(f_{Y'}^n) \asymp \rho\left(f_{Y'}^{n*}\colon \Num^1_{\RR}(E_Z)\to \Num^1_{\RR}(E_Z)\right)
\ge 
\rho\left(f_Y^{n*}\colon \Num^1_{\RR}(Z)\to \Num^1_{\RR}(Z)\right) \asymp
\deg_1(f^n_l).\] 
Suppose first $f_Y|_Z$ is algebraically stable (in the sense that no hypersurface is contracted by $f_Y|_Z$ to a subvariety that is eventually mapped to $\Ind(f_Y|_Z)$). 
Write $f_{Y'}^{*} \xi = \xi + \bl_Z^*(\omega_\star)$ for some $\omega_\star\in  \NS_{\RR}(Z)$.
A repeated use of~\cite[Lemma~2.3]{Regularization} implies that for each integer $n$, we have
$f_Y^{n*}\omega = (f_Y^{*})^n\omega$. Note that $f_{Y'}|_{E_Z}$ is then also algebraically stable, hence for any class $\omega \in \Num^1_{\RR}(Z)$, and for any $t\in \RR$, we have
\[
f_{Y'}^{n*}(\bl_Z^*\omega+ t\xi) = 
(f_{Y'}^{*})^n(\bl_Z^*\omega+ t\xi)
= 
\bl_Z^* (f_Y^{*})^n\omega
+ 
t\xi 
+\sum_{j=0}^{n-1} \bl_Z^* (f_Y^{*})^j \omega_\star
\]
and we conclude that $\|f_{Y'}^{n*}(\bl_Z^*\omega+ t\xi)\| \le C n\times \rho\left(f_Y^{n*}\colon \Num^1_{\RR}(Z)\to \Num^1_{\RR}(Z)\right)$.

In the general case, we proceed as follows. We first choose a sequence of proper birational morphisms
$\mu_n\colon Z^{(n)} \to Z$ such that the birational maps
\[\mu_n^{-1} \circ \mu_{n+1}\colon Z^{(n+1)} \to Z^{(n)} \text{ and } F^{(n+1)}_Y:= \mu_n^{-1} \circ f_Y \circ \mu_{n+1}\colon Z^{(n+1)} \to Z^{(n)}\]
are both regular. To simplify notation, we write $f_Y^n = F^{(1)}_Y \circ \cdots \circ F^{(n-1)}_Y$, and $M_n := \mu_{1} \circ \cdots \circ \mu_{n-1}$.
Pick any ample class $\omega \in \NS_{\RR}(Z)$. Then
$\deg_1(f_Y^n) =F^{n*}_Y\omega \cdot M_n^* (\omega^{N-c-1})$.
Since $\bl_Z$ is the projectivization of a vector bundle $V$, we have a commutative square: 
 \[\begin{tikzcd}[ampersand replacement=\&]
	E_Z^{(n)}\& E_Z^{(n-1)} \\
	Z^{(n)}\& Z^{(n-1)}
	\arrow["\hat{\mu}_n", from=1-1, to=1-2]
	\arrow["{\bl^{(n-1)}}", from=1-2, to=2-2]
	\arrow["{\bl^{(n)}}"', from=1-1, to=2-1]
	\arrow["\mu_n", from=2-1, to=2-2]
\end{tikzcd}\]  
where $\hat{\mu}_n$ is also birational and $\bl^{(n)}$ is the projectivization of $M_n^*V$. 
Write $\hat{M}_n := \hat{\mu}_{n-1} \circ \cdots \circ \hat{\mu}_1$.
Observe that the following diagram is commutative:
 \[\begin{tikzcd}[ampersand replacement=\&, column sep=huge]
  \& E_Z^{(n)}\& \\
 Z^{(n)}\& E_Z^{(n-1)}\& E_Z^{(n-1)}\\
	Z^{(n-1)}\& Z^{(n-1)}\&
	\arrow["\mu_n"', from=2-1, to=3-1]
	\arrow[dashrightarrow, from=3-1, to=3-2]
	\arrow[dashrightarrow, from=2-2, to=2-3]
	\arrow["{\bl^{(n)}}"',  from=1-2, to=2-1]
	\arrow["{\bl^{(n-1)}}", from=2-3, to=3-2]
	\arrow["{\bl^{(n-1)}}", near start, from=2-2, to=3-1]
	\arrow["\hat{\mu}_n", from=1-2, to=2-2]
	\arrow["F^{(n)}_{Y'}", from=1-2, to=2-3]
	\arrow["F^{(n)}_Y", near start, crossing over, from=2-1, to=3-2]
\end{tikzcd}\]  
Pick any small  $t$ such that the class $\omega_+= \bl_Z^*(\omega) + t \xi$ is ample.  We have
\begin{align*}
\deg_1(f_{Y'}^n) 
&=F^{n*}_{Y'}\omega_+ \cdot \hat{M}_n^* (\omega_+^{N-1})
\\
&=
F^{n*}_{Y'} \bl_Z^*(\omega) \cdot \hat{M}_n^* (\omega_+^{N-1})
+ 
t F^{n*}_{Y'}\xi \cdot \hat{M}_n^* (\omega_+^{N-1})
\\
&=
 (\bl^{(n)})^* F^{n*}_{Y}(\omega) \cdot \hat{M}_n^* (\omega_+^{N-1})
+ 
t F^{n*}_{Y'}\xi \cdot \hat{M}_n^* (\omega_+^{N-1})
\le \deg_1(f^n_l) + t F^{n*}_{Y'}\xi \cdot \hat{M}_n^* (\omega_+^{N-1})
\end{align*}
The fact that the restriction of $f$ to the generic fiber of $\bl_Z$ has degree $1$ implies that 
we may write $F^{*}_{Y'} \xi = \hat{\mu}^*_1\xi + \bl^{(1)*}(\beta_\star)$ 
for some $\beta_\star \in \Num^1_{\RR}(Z^{(1)})$. And we obtain: 
\[
F^{n*}_{Y'}\xi= \hat{M}^*_n\xi + 
\sum_{j=0}^{n-1}\bl^{(n-j)*}(\mu_n^*\cdots \mu_{j+1}^*) F^{(j)*}_Y\beta_\star 
\]
 Since $|F^{(j)*}_Y\beta_\star \cdot M_j^*\omega^{N-1}|$ is bounded by $\deg_1(f^n)$ up to a uniform constant, we obtain
$F^{n*}_{Y'}\xi\cdot \hat{M}_n^* (\omega_+^{N-1}) = O(n\deg_1(f^n))$ as required.
 \end{proof}

\begin{proposition}\label{prop: dynamical degree on invariant subvariety}
 Let $X$ be a smooth projective variety and let $f\colon X\dashrightarrow X$ be a birational map.
 If~$Z\subset X$ is a codimension $c$ subvariety such that $Z\not\subset \Ind(f)$ and $f^\vee(Z) = Z$ then for all $0\leqslant k\leqslant \dim(Z)$, 
 there exists a constant $C>1$ such that 
 $\deg_k(f^n|_Z) \leqslant C\, \min\{\deg_k (f^n), \deg_{k+c} (f^n)\}$.  In particular, we have
  \[
  \lambda_k(f|_Z)\leqslant \min\{\lambda_k(f), \lambda_{k+c}(f)\}.
 \]
\end{proposition}

\begin{proof}
By Hironaka's resolution of singularities and birational invariance of degrees up to uniform constants,  we may suppose that $Z$ is smooth. 
Fix an ample line bundle $L\to X$.

Recall the definition of basepoint free numerical classes from Fulger and Lehmann~\cite[\S5]{MR3592463}. 
Suffice it to say that any basepoint free class in $\Num^i_{\RR}(X)$ is both pseudo-effective and
nef, and that the cone of basepoint free classes has non-empty interior. 
Moreover, for any basepoint free class $\Omega \in \Num^i_{\RR}(X)$, there exists a constant $C>0$
such that $\Omega \le C c_1(L)^i$.

It follows that the fundamental class $[Z]$ can be written as a difference of two basepoint free classes 
$[Z] = \Omega_1 - \Omega_2$ with $\Omega_1, \Omega_2\in \Num_{\RR}^c(X)$, and we have
\begin{align*}
 \deg_k(f^n|_Z) &= f^{n*}c_1(L)^k \cdot c_1(L)^{N-c} \cdot [Z]
 \\
 &=
 f^{n*}c_1(L)^k \cdot c_1(L)^{N-c} \cdot \Omega_1- f^{n*}c_1(L)^k \cdot c_1(L)^{N-c} \cdot \Omega_2
\end{align*}
Beware that the class $f^{n*}c_1(L)^k$ is only pseudoeffective and not nef in general. 
To get around that difficulty, we consider for each $n$ a resolution of the graph  $\Gamma_n$ of $f^n$ with projection maps
$\pi_{1,n}, \pi_{2,n}\colon \Gamma \to X$ so that $f^n= \pi_{2,n}\circ \pi_{1,n}^{-1}$.
Then for $\epsilon\in\{1,2\}$, we get
\[
f^{n*}c_1(L)^k \cdot c_1(L)^{N-c-k} \cdot \Omega_\epsilon
=  
\pi_{2,n}^{*}c_1(L)^k \cdot \pi_1^*c_1(L)^{N-c-k} \cdot \pi_1^* \Omega_\epsilon
\le 
C\, \left(\pi_{2,n}^{*}c_1(L)^k \cdot \pi_1^*c_1(L)^{N-k}\right)
\]
and 
$ \deg_k(f^n|_Z)\le C \, \deg_k(f^n)$.

Denote by $\Gamma_{n,Z}$ the closure of the graph of $f^n|_{Z\setminus \Ind(f)}$ in $\Gamma_n$.
Let $W$ be any complete intersection subvariety of codimension $k$ in $Z$ whose
fundamental class is equal to $ c_1(L|_Z)^k$. Since  $f^\vee(Z) = Z$, and $f$ is birational, we may choose $W$ so that 
it intersects properly each component of the locus $V_l :=\{ p\in X, \dim \pi_{2,n}^{-1}(p) \ge l\}$
for all $l\ge1$ and all $n\ge0$. Note that $W$ defines two fundamental classes: $[W]_Z=  c_1(L|_Z)^k \in \Num^k(Z)$, and 
$[W]_X=c_1(L)^k \cdot [Z] \in \Num^{k+c}(X)$.

By~\cite[Lemma~4.2]{MR4048444}, the proper transform
$(\pi_{2,n}^{-1})^\vee(W)$ (resp. $(\pi_{2,n}^{-1})^\vee(W)\cap \Gamma_{n,Z}$) represents $\pi_{2,n}^*[W]_X$ (resp. 
$\pi_{2,n}^*[W]_Z$). In particular, the class
$\pi_{2,n}^*[W]_X - \pi_{2,n}^*[W]_Z$ is represented by an effective cycle.
We now compute
\begin{align*}
\deg_k(f^n|Z) 
&=
[(\pi_{2,n}^{-1})^\vee(W)\cap \Gamma_{n,Z}]\cdot \pi_{1,n}^* c_1(L)^{N-c-k}
\\
&\le [(\pi_{2,n}^{-1})^\vee(W)]\cdot \pi_{1,n}^* c_1(L)^{N-c-k}
\\
&= f^{n*}[W]\cdot c_1(L)^{N-c-k}= \deg_{k+c}(f^n)
\end{align*}
This concludes the proof.
\end{proof}

\subsection{Regularization of bimeromorphic mappings}
A bimeromorphic self-map $f\colon X \dto X$ of a normal complex manifold is said to be regularizable
iff there exists a proper bimeromophic map $\varphi \colon X \dto Y$ such that
$\pi \circ f \circ\pi^{-1}\colon Y \to Y$ is a biholomorphism. 

We collect here two observations on this notion.

\begin{proposition}\label{prop:smooth-regular}
Suppose  the bimeromorphic self-map $f\colon X \dto X$ is regularizable. Then 
there exist a smooth manifold $Y$ and
a proper bimeromophic map $\varphi \colon X \dto Y$ 
such that
$\pi \circ f \circ\pi^{-1}\colon Y \to Y$ is a biholomorphism. 
\end{proposition}

\begin{proof}
This is a consequence of the existence of a functorial resolution of singularities in the category of analytic spaces, see~\cite[Theorem~2.0.1]{MR2500573}.
\end{proof}

\begin{proposition}\label{prop:iterate-regular}
Let  $f\colon X \dto X$ be any  bimeromorphic self-map of a normal complex variety. 
If $f^n$ is regularizable for some $n\in \NN^*$, then $f$ is also regularizable. 
\end{proposition}

\begin{proof}
Let $U$ be an open (for the analytic Zariski topology) dense subset in $X$ such that the restrictions of
$f, \cdots, f^{n-1}$ are all regular on $U$. 
Define $\Gamma$ to be the closure (for the euclidean topology) of the set of 
points $\{(x, f(x), \cdots, f^{n-1}(x)),\, x\in U\}\subset X^n$. 
The first projection $\pi_1\colon \Gamma \to X$ is a proper bimeromorphic map. 
On $X^n$, consider the biholomorphism $F(x_0, \cdots x_{n-2} , x_{n-1}) = (x_1, \cdots , x_{n-1}, f^n(x_0))$. 
This map preserves $\Gamma$ and $\pi_1 \circ F = f \circ \pi_1$. 
This completes the proof. 
\end{proof}

%%%%%%%%%%%%%%%%%%%%%%%%%%%%%%%%%%%%%%%%%%%

\section{Semi-abelian varieties}\label{sec:semi-abelian}
In this section we recall some basic facts on the geometry of semi-abelian varieties and compute the degree growth of their 
automorphisms.

\subsection{Geometry of semi-abelian varieties} 
Our reference for this section is~\cite[Chapter~5]{MR3156076}.

 A {\it semi-abelian variety} $G$ is a connected commutative complex algebraic group 
fitting into an exact sequence of algebraic groups:
 \begin{equation}\label{equation: semi-abelian variety}
  1 \to T \to G \to A \to 1,
 \end{equation}
 where $T\cong \GGmr$ is a split algebraic torus and $A$ is an abelian variety.
 Observe that $G$ is projective if and only if it is abelian, and $G$ is affine if and only if it
 is an affine torus.
 
By Chevalley's structure theorem, the torus $T$ is uniquely determined so that the sequence~\eqref{equation: semi-abelian variety}
is unique (there might be other exact sequences  $1 \to T \to G \to A' \to 1$ with $A'$ abelian
in category of complex Lie groups, see~\cite[\S 5.1.5]{MR3156076}).  

Write $r=\dim(T)$ and $g=\dim(A)$. Since $T$, $G$ and $A$ are abelian, the exponential maps
$\exp\colon\Lie(T)\to T$,  $\exp\colon \Lie(G) \to G$, and $\exp\colon \Lie(A) \to A$
have discrete kernels, and they all fit into the following commutative diagram in which both 
lines are exact: 
\[\begin{tikzcd}[ampersand replacement=\&]
	0 \& {\Lie(T)} \& {\Lie(G)} \& {\Lie(A)} \& 0 \\
	1 \& T \& G \& A \& 1
	\arrow[from=1-1, to=1-2]
	\arrow[from=1-2, to=1-3]
	\arrow[from=1-3, to=1-4]
	\arrow[from=1-4, to=1-5]
	\arrow[from=2-1, to=2-2]
	\arrow[from=2-2, to=2-3]
	\arrow[from=2-3, to=2-4]
	\arrow[from=2-4, to=2-5]
	\arrow["{\exp_T}"', from=1-2, to=2-2]
	\arrow["{\exp_G}"', from=1-3, to=2-3]
	\arrow["{\exp_A}", from=1-4, to=2-4]
\end{tikzcd}\]
We write $V:= \Lie(G) \simeq \CC^{r+g}$, and $\Lambda := \ker(\exp_G)$.
The latter is a discrete subgroup of rank $r+ 2g$.
Observe that  $W:= \Lie(T) \simeq \CC^r$ is a subspace of $V$ such that $\Lambda_W := \Lambda \cap W$
is a discrete subgroup of rank $r$ that generates $W$ as a complex vector space. 
Write $\Lambda_{\RR}:=\Lambda\otimes_{\ZZ}\RR$, and consider the complex vector space
\begin{equation*}
 U := \Lambda_{\RR}\cap i\Lambda_{\RR} \subset V.
\end{equation*}
Then $U$ is a complex vector subspace of $V$, and a dimension argument implies $V = W\oplus U$. 

We have thus obtained:
\begin{proposition}\label{proposition: construction of semiabelian varieties}
 To any semi-abelian variety $G$ is associated a complex vector space $V$ of dimension $r+g$, a discrete subgroup
 $\Lambda$ of rank $r+2g$, and a splitting $ V= W\oplus U$ of complex vector spaces
 such that 
 \begin{itemize}
 \item
 $\dim(W)=r$, $\dim(U)=g$;
\item 
 $\Lambda_W:= \Lambda\cap W$ has rank $r$ and generates $W$ as a complex vector space;
 \item
 $U = \Lambda_{\RR}\cap i\Lambda_{\RR}$;
 \item
 the image of $\Lambda$ in $V/W$ is a cocompact lattice $\bar{\Lambda}$ and $(V/W)/\bar{\Lambda}$ is an abelian variety of dimension $g$;
 \item 
 the exact sequence 
$1\to W/\Lambda_W\to G\to (V/W)/\bar{\Lambda}$ is the defining sequence~\eqref{equation: semi-abelian variety} canonically attached to $G$.
\end{itemize}
The quadruple $(V,W,U,\Lambda)$ is unique up to isomorphism.
\end{proposition}
A quadruple $(V,W,U,\Lambda)$ as above is called a presentation of the semi-abelian variety $G$.
\begin{remark}
The intersection $\Lambda \cap U$ is a discrete subgroup of $U$ whose rank 
can be any integer from $0$ and $2g$. Observe that $G$ is a product (as an algebraic group) if and only if
this rank is equal to $2g$.
 \end{remark}

\subsection{Algebraic compactification of semi-abelian varieties}\label{sec:compactification}
Let $G$ be a semi-abelian variety. 
The exact sequence~\eqref{equation: semi-abelian variety} exhibits $G$ as the total space of
a principal $T$-bundle over $A$ (in algebraic terms, $G$ is a $T$-torsor over $A$).
In terms of a presentation $(V,W,U,\Lambda)$ of $G$, this can be understood as follows. 

The family of affine planes $z+U$ projects to a smooth foliation on $G$ that is transverse to the fibers
of the canonical projection $\pi \colon G \to A$, and its holonomy gives rise to a 
monodromy representation given by a morphism $\bar{\Lambda}$ to the group of biholomorphisms of $T$. 

In concrete terms, pick any $\bar{\lambda} \in \bar{\Lambda}$ and lift it to $\lambda \in \Lambda$. 
Write $\lambda = \lambda_W + \lambda_U$ with  $\lambda_W\in W$ and  $\lambda_U\in U$. 
Observe that in general $\lambda_W\notin \Lambda_W$. However since $\Lambda_{\RR} =  \Lambda_{W,\RR}\oplus U$, 
it follows that $\lambda_W\in \Lambda_{W,\RR}$. 
The monodromy associated $\bar{\lambda}$ is thus the translation by  $-\lambda_W \in \Lambda_{W,\RR}$. 
In particular,  the monodromy is given by a morphism  $\rho \colon  \bar{\Lambda} \to K$
where $K= \Lambda_{W,\RR}/\Lambda_W$ is the maximal compact (real) subgroup of $T= W/\Lambda_W$.

Set $X:= V/\Lambda_W$: this is a principal $T$-bundle over $V/W$. Observe that the canonical splitting of $V$ descends to a biholomorphism
$\varphi\colon W/\Lambda_W \oplus U \to X$.
One can then recover $G$ as the quotient of  $X$ by the action of $\bar{\Lambda}$ given by 
$\bar{\lambda}\cdot \varphi (w,u) :=  \varphi(\rho(\bar{\lambda})\cdot w,u+\lambda_U)$.

\smallskip

Choose any smooth projective toric variety $M$ of dimension $r$. 
Then $M$ is equipped with a $T$-action which has a dense orbit, and
this orbit is canonically identified with $T$. 
Then we define $\bar{G}_M$ as the quotient
of $M \times U$ by the action of $\bar{\Lambda}$
given by $\bar{\lambda}\cdot (w,u) = (\rho(\bar{\lambda})\cdot w,u+\lambda_U)$
for $w\in M$ and $u\in U$. 
In this way we obtain all smooth equivariant algebraic compactifications of $G$.

\smallskip

We shall only use the case $M= (\p^1)^r$ in the sequel, and thus write $\bar{G} := \bar{G}_{(\p^1)^r}$
in order to simplify notations. 

The choice of a basis for $\Lambda_W$ gives canonical coordinates $w_1,\dots, w_r$ on $T=W/\Lambda_W$, hence on $(\p^1)^r$.
In the same way,  we choose a basis of $V/W\cong U$ which gives linear 
coordinates $z_1,\dots, z_g$ on $V/W$. 
The positive smooth $(1,1)$ form 
\begin{equation}\label{def:kahler}
 {\omega} =  \frac{i}{2} \sum_{k = 1}^r\frac{dw_k\wedge d\overline{w_k}}{(1+w_k\overline{w_k})^2} + \frac{i}{2} \sum_{j=1}^g dz_j\wedge d\overline{z_j}.
\end{equation}
on $(\p^1)^r\times U$ then descends to a Kähler form on $\bar{G}$ since the monodromy representation takes its values in $K$.

%%%%%%%%%%%%%%%%%

\subsection{Automorphisms of semi-abelian varieties} 

Let $G$ be a semi-abelian variety. We let $\Aut(G)$ be the group of algebraic biholomorphisms\footnote{Observe that the group of all biholomorphisms of $G$ might be larger than $\Aut(G)$.}
of $G$. This group contains all translations $z\mapsto z+x$ for any $x\in G$. 

Pick any element $g\in \Aut(G)$. Since any algebraic map from $T$ to $A$ is constant, any algebraic self-map of $G$
 descends to $A$. In particular, for any $g\in  \Aut(G)$ fixing the neutral point, 
the map $\varphi_y(x):= g(x+y) - g(x) -g(y)$ descends to the zero map on $A$, and we have
$\varphi_y(T) \subset T$. Since $\varphi_y$ is algebraic 
its restriction to $T$ is a morphism, hence $\varphi_y \equiv 0$. This proves any $g\in G$
is the composition of a group morphism and a translation, so that we have the following exact sequence
\[
1 \to G \to \Aut(G) \to \Aut_\bullet(G)\to 1~,\]
where $\Aut_\bullet(G)$ is the subgroup of algebraic group isomorphisms. 

Since any algebraic biholomorphism descends to $A$, the canonical exact sequence~\eqref{equation: semi-abelian variety} also yields two exact sequences
\[  1 \to H_T \to \Aut(G) \to H_A \to 1,  \text{ and } 
1 \to H_{T,\bullet}  \to \Aut(G_\bullet) \to H_{A,\bullet} \to 1
 \]
 where $H_T$ (resp. $H_A$) is a subgroup of
  $\Aut(T)$ (resp. $\Aut(A)$); and 
   $H_{T,\bullet}$ (resp. $H_{A,\bullet}$) is a subgroup of
  $\Aut_\bullet(T)\simeq \SL(r,\ZZ)$ (resp. $\Aut_\bullet(A)$ which is a discrete subgroup of $\GL(g,\CC)$).

\begin{lemma}\label{lemma: differential of automorphism of semiabelian variety}
Let  $(V,W,U,\Lambda)$ be a presentation of $G$, and $\pr \colon V\to G$ be the canonical morphism.
Then for any $f \in \Aut(G)$ there exist a point $x\in G$ and linear endomorphisms $u_T(f) \colon W \to W$, $u_A(f) \colon U\to U$
such that
$f(\pr(v)) = \pr(u(f)(v))+ x$ for all $v\in V$, 
where $u(f)$ is the endomorphism of  $V= W\oplus U$ defined by $u(f) = u_T(f) \oplus u_A(f)$.
\end{lemma}
\begin{proof}
Replacing $f$ by $f - x$ with $x=f(0)$ we may suppose $f(0) = 0$.
By our previous arguments, $f$ is then a group morphism. It thus lifts to the universal cover $V$
of $G$ and defines an endomorphism $u(f) \colon V\to V$ such that $u(f)(\Lambda) = \Lambda$. 
Since $u(f)$ is $\CC$-linear, it preserves $U = \Lambda_\RR\cap i \Lambda_\RR$. 
On the other hand, we proved that $g$ must preserve $T$, hence $u(f)$ fixes $W$ too.
We have thus proved that $u(f)$ preserves each factor of the splitting $V= W\oplus U$ as required.
\end{proof}

Recall that an \emph{isogeny} $f\colon G \to G'$ between two semi-abelian varieties is a surjective group morphism
with finite kernel. This is equivalent to impose $f$ to be finite and surjective.

We shall below  consider the ring of group endomorphisms $\End(G)$ which is isomorphic to the set of complex linear
maps $u \colon V \to V$ such that $u(\Lambda)\subset \Lambda$.

\subsection{Degree growth of automorphisms of semi-abelian varieties}

\begin{theorem}\label{theorem: dynamical degrees on semi-abelian variety}
Let $G$ be any semi-abelian variety, and $(V, W, U, \Lambda)$ be a presentation of $G$.
Let $\pi\colon G \to A$ be the canonical projection to an abelian variety such that $\ker(\pi)$
is a split torus $T$.

Pick any automorphism $f\in \Aut(G)$, and let $x\in G$, $u_A(f) \in \End(U)$ and $u_T(f) \in \End(W)$
such that $f(\pi(v)) = \pi(u(f)(v)) + x$ for all $v\in V$ with $u(f) = u_T(f) \oplus u_A(f)$.
Then, we have
\[
\deg_k(f^n)
\asymp
\max_{ 0\le j \le k} 
\left\{
\| \Lambda^{j,j} u_A(f)^n \| \times \| \Lambda^{k-j} u_T(f)^n\|
\right\}
\]
\end{theorem}

Let $u\colon V\to V$ be any endomorphism. 
Then for all $j$, we denote by $\Lambda^j u$ the induced endomorphism on 
the vector space $\Lambda^j V^*$; and by $\Lambda^{j,j} u$ the endomorphism
on the  space $\Lambda^{j,j} V^*$ of $2j$-forms to type $(j,j)$.

\begin{proof}
We compute the degrees in the projective compactification $\bar{G}$ as in \S\ref{sec:compactification}
with the Kähler form~\eqref{def:kahler}.

Observe that any translation on $G$ extends to an automorphism on $\bar{G}$, 
hence acts as the identity on the cohomology of $\bar{G}$. It follows that we may suppose that $x=0$. 
It follows that $f$ restricts to a morphism $f_T$ on $T$, and also descends to a morphism $f_A\colon A\to A$. 
Write
\[
 {\omega}_T :=  \frac{i}{2} \sum_{k = 1}^r\frac{dw_k\wedge d\overline{w_k}}{(1+w_k\overline{w_k})^2}
 \text{ and } \omega_A:=
  \frac{i}{2} \sum_{j=1}^g dz_j\wedge d\overline{z_j}.
\]
To simplify notation we identify $\omega_A$ with its image in the quotient space $A$.
 Since $\lambda_k(f_A)$ equals the spectral norm of the pull-back action of $f$ on $H^{k,k}(A)$, we have
\[
\int_{A} f_A^{n*}\omega_A^k \wedge \omega_A^{g-k}
\asymp \| \Lambda^{k,k} u_A(f)^n\|
~;\]
and~\cite{MR3043585} or \cite{MR3059849} imply similarly: 
\[
\int_{\PP^r_{\CC}} f_T^{n*}\omega_T^k \wedge \omega_T^{r-k}
\asymp \| \Lambda^k u_T(f)^n\|~.
\]
Choose a fundamental domain $\Pi(G)$ for the action of $\bar{\Lambda}$ in $X = V/\Lambda_W$. 
Recall that there is a canonical projection $X\to U$, and 
observe that we may choose $\Pi(G)$ as the product of $T$ with a fundamental domain $\Pi(A)$ in $U$
for the action of $\bar{\Lambda}$. Denote by $f_X\colon X\to X$ the lift of $f$
(or equivalently the map induced by $u(f)$ from $V$ onto $X$).
Note that by Lemma~\ref{lemma: differential of automorphism of semiabelian variety}
$f_X$ is equal the product map $(f_T, u_A(f))$ in the coordinates $(w_j), (z_i)$ on $T\times U$. 

We then have the following series of equalities: 
\begin{align*}
\deg_k(f^n)
&=
\int_{\bar{G}} f^{n*}\omega^k \wedge \omega^{r+g -k}
=\int_G f^{n*}\omega^k \wedge \omega^{r+g -k}
\\
&=
\int_{\Pi(G)} u(f)^{n*}(\omega_A+\omega_T)^k \wedge (\omega_A+\omega_T)^{r+g -k}
\\
&=
\sum_{j,l} {k \choose j} {r+g-k \choose l} \int_{\Pi(G)} u_A(f)^{n*}\omega_A^j \wedge  f_T^{n*}\omega_T^{k-j} \wedge \omega_A^l \wedge \omega_T^{r+g-l}
\\
&=
\sum_{j=0}^k {k \choose j} {r+g-k \choose g-j} \int_{\Pi(A)} u_A(f)^{n*}\omega_A^j  \wedge \omega_A^{g-j} \, \int _T  f_T^{n*}\omega_T^{k-j} \wedge \omega_T^{r-j}
\\
&\asymp
\max_{0\le j\le k} \deg_j(f_A^n) \, \deg_{k-j}(f_T^n)
\end{align*}
which concludes the proof.
\end{proof}

Denote by $\tau_1, \dots, \tau_r$ the eigenvalues counted with multiplicities of $u_T(f)$ and by $\alpha_1,\dots,\alpha_g$ the eigenvalues of $u_A(f)$. Reorder them in the following way:
\begin{align*}
 &|\tau_1|\geqslant |\tau_2|\geqslant \dots \geqslant |\tau_r|;\\
 &|\alpha_1|\geqslant |\alpha_2|\geqslant \dots\geqslant |\alpha_g|.
\end{align*}

The previous theorem can be used to describe the degree growth in codimension $1$. 
\begin{corollary} \label{corollary: growth of first degree on semi-abelian variety}
 Let $f\colon G\to G$ be an automorphism of a semi-abelian variety $G$. Then in the above notation the first dynamical degree of $f$ is equal to:
 \begin{equation*}
  \lambda_1(f) = \max\{|\tau_1|, |\alpha_1|^2\}.
 \end{equation*}
 Moreover, if $ \lambda_1(f) = 1$, then $ \deg_1(f^n) \asymp n^d$ with $d= {\max\{{j_T - 1}, {2(j_A - 1)}\}}$
 where $j_T$ and $j_A$ denote the unipotent index of $u_T(f)$ and $u_A(f)$ respectively.
\end{corollary}

\begin{corollary}\label{corollary: dynamical degrees on semi-abelian variety}
 Let $f\colon G\to G$ be an automorphism of a semi-abelian variety $G$. In the above notation~$f$ has the following dynamical degrees:
 \[
  \lambda_j(f)
  = \max_{k+l =j, 0\le k \le r, 0\le  l\le g}
  \prod_{m=1}^k |\tau_m| \times
  \prod_{n = 1}^l |\alpha_n|^2 
  \]
\end{corollary}
In the previous formula, we adopt the convention that when $k=0$ or $l=0$ then
the product is over the emptyset, hence equals $1$. 
\begin{remark}
The computation of dynamical degrees of endomorphisms of semi-abelian surfaces was also explored in Abboud's thesis, see~\cite[\S 4.3.3]{abboud}.
\end{remark}

\subsection{Decomposition of automorphisms of abelian varieties}
In this section, we suppose that $G = V/\Lambda$ is an abelian variety so that $\Lambda$
generates the complex vector space $V$ as a real vector space. 

\smallskip

Recall that an abelian variety is simple when it does not contain any non-trivial abelian subvarieties. 
Let $L\to G$ be any ample line bundle on $G$. The map sending $a\in G$ to $t_a^*L \otimes L^{-1}$ (where $t_a$ denotes
the translation by $a$)  defines an isogeny $\phi_L\colon G \to G^{\vee} := \Pic^0(G)$ (see~\cite[II.8 Theorem 1]{MR2514037}). 
If $H$ is an abelian subvariety of $G$, and $\imath \colon H \to G$ is the canonical embedding map, we may 
consider the fiber  $H'$ containing $0$ of the map $G \to H^\vee$ given by $\imath^{\vee} \circ \phi_L$. 
Then $H\cap H'$ is finite, and the addition map $H\times H' \to G$ is an isogeny, see~\cite[p.127]{Birkenhake-Lange_AV}.

 \begin{theorem}\label{thm:splitting}
Pick  any algebraic group isomorphism $f\in \Aut_\bullet(G)$. Then there exist two canonically defined 
$f$-invariant abelian subvarieties $G_0, G_1$
satisfying the following conditions:
\begin{enumerate}
\item
the addition map $\mu\colon G':= G_0\times G_1 \to G$ is an isogeny;
\item 
all eigenvalues of $u(f|_{G_0})$ are roots of unity;
\item 
no eigenvalue of $u(f|_{G_1})$ is a root of unity.
\end{enumerate}
\end{theorem}

\begin{remark}
Note that $\lambda_1(f|_{G_0})=1$, and for any $f$-invariant abelian subvariety $H$ of $G_1$
we have  $\lambda_1(f|_{H})>1$. Also when $G$ is simple, then either $G_0$ or $G_1$ is reduced to a point. 
\end{remark}

\begin{proof}
Consider the complex linear map $u(f) \colon V \to V$ as above. Since $u(f)$ preserves the lattice $\Lambda$, it also induces
a $\QQ$-linear map $u_r \colon \Lambda_\QQ \to \Lambda_\QQ$. Since $\Lambda$ is co-compact in $V$, the embedding $\Lambda \subset V$ induces an isomorphism of
real vector spaces $\phi\colon \Lambda_\RR \to V$. 

Let $J(\lambda) :=  \phi^{-1}( i \phi(\lambda))$. It defines an endomorphism $J\colon \Lambda_\RR\to \Lambda_\RR$ such that $J^2= - \id$,
and since $u(f)$ is $\CC$-linear, we have $J\circ u_r= u_r\circ J$. 

Let $\chi_r\in \ZZ[T]$ be the characteristic polynomial of $u_r$. 
Then $\chi_r$ is a polynomial of degree $2g$, which can be decomposed
as a product $\chi_r = P Q$ where $P\in \ZZ[T]$ is a product of cyclotomic factors, and $Q\in \ZZ[T]$ does not vanish at any root of unity.
We obtain a splitting $\Lambda_\QQ = \ker P(u_r) \oplus \ker Q(u_r)$, and since $u_r$ and $J$ commute,
the real spans of both $\ker P(u_r)$ and $\ker Q(u_r)$ in $\Lambda_\RR$ are $J$-invariant. 
Write $V_0 = \phi(\ker P(u_r))$ and $V_1 = \phi(\ker Q(u_r))$. 
It follows that $V_0$ and $V_1$ are $u(f)$-invariant complex vectors spaces such that $V= V_0 \oplus V_1$
such that $\Lambda \cap V_i$ is a lattice in $V_i$ for each $i=0,1$. 
The theorem follows with $G_i = V_i/(\Lambda_i\cap V_i)$. 
\end{proof}

%%%%%%%%%%%%%%%%%%%%%%%%%%%%%%%%%%%%%%%%%%%%%%%%%%

\section{Families of abelian varieties and Néron models} \label{sec:neron}
In this section, we discuss Néron models in the context of degenerating families of 
complex abelian varieties, and give a proof of Theorem \ref{Theorem: criterion for irreducible family of AV}.

\subsection{N\'eron models}
In order to talk about Néron models, we have to make a slight twist in the terminology used in the introduction.

A  family (of polarized manifolds) over the complex unit disk is 
 a smooth complex manifold $\X$, with a surjective holomorphic map 
 $\pi\colon \X \to \DD$  which is a proper submersion over $\DD^*$, and 
carrying a relatively ample line bundle $\L \to \X$. 

When $X_t= \pi^{-1}(t)$ is an abelian variety for each (or some) $t\neq 0$, then we say that it is 
a  family of polarized abelian varieties. The family of neutral elements of $X_t$
gives a canonical holomorphic section over $\DD^*$, that we refer to the zero section. 

We allow $X_0$ to be non-compact. Observe that $X_0$ is compact if and only if $\pi$ is proper, in which case
we say the family is proper. 
We say that a polarized family is smooth when the map $\pi$ is a submersion over $\DD$. 
A proper smooth  family of polarized abelian varieties is simply a deformation of abelian varieties
defined over $\DD$.
We write $\X^* = \X\setminus X_0 = \pi^*(\DD^*)$.

A model $\Y$ of $\X$ is a  family of polarized manifolds $\varpi \colon \Y \to \DD$, 
and a meromorphic map $\phi \colon \Y \dto \X$ such that 
$\pi' = \pi \circ \phi$ and $\phi$ is a biholomorphism from $\Y^*$ onto $\X^*$. 

\medskip

Let $R$ be the ring of germs of holomorphic functions at the origin $0\in \DD$. 
With the $t$-adic norm, $|\sum a_n t^n|_t = \exp(- \min\{n, a_n\neq 0\})$, it becomes
a discrete valued ring, whose completion is the ring of formal power series 
$(\Pow,|\cdot|_t)$. The fraction field $K$ of $R$ is a valued field whose completion is $(\Lau, |\cdot|_t)$.

If  $\pi\colon \X \to \DD$ is a  family of polarized abelian varieties, there exists $r>0$ and an embdedding $\imath \colon \X \hookrightarrow \PP^N_\CC \times \DD_r$
such that $\pr_2 \circ \imath = \pi$, where $\DD_r :=\{ |t| <r\}$ and $ \pr_2$ is the projection onto the second factor. It follows that $\X$ is defined by a finite family of homogeneous polynomials
with coefficients in $R$. This family defines a projective scheme flat over $\Spec R$ that we denote by 
$\X_R$. Its generic fiber $X_K$ is a smooth abelian $K$-variety.
 
\begin{theorem}\label{thm:neron-model}
 Let $X_K$ be any abelian variety over $K$. 
 There exists an $R$-scheme $\N_R$, flat over $\Spec R$
 and an isomorphism $\phi \colon N_K \to X_K$ such that
 $\X$ satisfies the following property:
 \begin{itemize}
 \item  for each smooth $R$-scheme $\Y$ and for any $K$-morphism 
$ f\colon Y_K \to X_K$, 
  there exists a unique $R$-morphism $\varphi\colon \Y\to  \N_R$ which extends $f$, i.e., such that $f=\phi_K\circ \varphi_K$.
\end{itemize}
\end{theorem}
The property characterizing $\N_R$ is called the N\'eron mapping property, and $\N_R$
is then called a Néron model.
It implies the uniqueness (up to isomorphism) of the Néron model. 
We refer to~\cite{BLR_Neron_models} for a proof of this result. 
It is crucial to keep in mind that in general the special fiber $\N_s$ is not projective. 

\begin{corollary}\label{cor:analyticneron}
Let  $\pi\colon \X \to \DD$ be a  family of polarized abelian varieties. 
Then there exists a  family of polarized abelian varieties $\pi_\N\colon \N\to \DD$, 
and an isomorphism $\phi_N\colon\N^*\to  \X^*$ satisfying $\pi = \pi_\N\circ \phi_N$, 
such that the following holds. 

For any smooth polarized family $\varpi \colon \Y \to \DD_r$ with $0<r<1$, and for any holomorphic map 
$f\colon \Y^*\to \X^*$ such that $\varpi = f\circ \pi$, there exists a unique
holomorphic map $\varphi \colon \Y\to \N$ such that 
 $f=\phi_N\circ \varphi $.
 \end{corollary}

 \begin{remark}
 We shall call $\N$ as in the previous corollary the analytic Néron model  (or Néron model for short). 
 Observe that the extension property applied to the addition law on $\X^*\times \X^*$ implies that 
 $\N$ is also a commutative complex algebraic group. 
 
 The extension property also implies that any holomorphic section $\sigma\colon \DD \to \X$
 of $\pi$ (i.e., satisfying $\pi \circ \sigma(t) =t$) induces a holomorphic map
 $\sigma_\N \colon \DD \to \N$ such that 
 $\phi_\N \circ \sigma_\N = \sigma$.
  \end{remark}
  
  When the Néron model $\N$ is relatively projective (i.e., $\N_0$ is projective), then 
  we shall say that $\X$ is a \emph{non-degenerating family}. 
  \begin{remark}
When $t$ varies in $\DD^*$, then $(X_t, \L_t)$ forms a holomorphic family of polarized abelian varieties of a fixed type $D$ so that 
  we have a holomorphic map from $\DD^*$ to the moduli space of polarized abelian varieties of type $D$ (see, e.g.,~\cite[p.219]{Birkenhake-Lange_AV}).
  The family is non-degenerating if and only if this map extends holomorphically through $0$.
\end{remark}

 \begin{remark}
 In the case $g=1$ and $\X$ is a family of elliptic curves, then the Néron model can be obtained as follows, see~\cite{MR1312368}.
 First, after applying the minimal model program, we may suppose that $\X$ is relatively minimal, i.e., $X_0$ does not contain any smooth rational curve
of self-intersection $-1$. 
The divisor $\pi^*(0)$ is in general not reduced, and can be written as a sum
$\sum E_i + \sum b_j F_j$ where $E_i, F_j$ are the irreducible components of $X_0$
and $b_j\ge 2$ for all $j$. 
Then $\N$ can be taken as the smooth part of $\X \setminus \bigcup F_j$. 
We refer to~\cite[\S V.7]{MR2030225} for the description by Kodaira of all possible divisors $\pi^*(0)$ as above. 
 \end{remark}

\begin{proof}[Proof of Corollary~\ref{cor:analyticneron}]
Theorem~\ref{thm:neron-model} yields a polarized family of abelian varieties $\N\to\DD_r$
for some $0<r<1$. The isomorphism $\phi\colon N_K\to X_K$ defines a biholomorphism 
$\varphi\colon \N^*\to \pi^{-1}\{0 < |t| < r\}$ (possibly after reducing $r$), such that $\varphi^*\L$
is the polarization on $\N^*$. 

Glue $\N$ and $\X^*$ along $\pi^{-1}\{0 < |t| < r\}$, where we identify a point $z\in \N^*$ with $\varphi(z)\in  \pi^{-1}\{0 < |t| < r\}$.
The resulting space is a  family of polarized abelian varieties defined over $\DD$. 
To simplify notation we shall keep the same notation $\N$ for this family.

Suppose now that we are given a smooth polarized family $\varpi \colon \Y \to \DD_r$ with $0<r<1$, and a holomorphic map 
$f\colon \Y^*\to \X^*$ such that $\varpi = f\circ \pi$.  Since $\varpi$ is smooth in the analytic category, the associated $R$-scheme
is smooth (in the algebraic sense, see~\cite[Satz 3.1]{MR0404678}). 
By Néron extension property,  we can find $\rho<r$ and an analytic map $\varphi \colon \varpi^{-1}\{0<|t|<\rho\}\to \N$ such that 
 $f=\phi_N\circ \varphi$. By analytic continuation, $\varphi$ extends to $\Y$. 
 \end{proof}

Finally we shall need the following result. 
Suppose $\pi\colon\X\to \DD$ is any proper  family of  polarized abelian varieties, and pick any integer $n\in \NN^*$. 
A base change of order $n$ is a proper  family of polarized abelian varieties $\pi_n\colon \X_n \to \DD$, with a meromorphic map
$\phi\colon \X_n \dto \X$
such that the following square is commutative: 
\[\begin{tikzcd}[ampersand replacement=\&]
	{\X_n} \& \X \\
	{\DD} \& {\DD}
	\arrow[dashrightarrow, "\phi", from=1-1, to=1-2]
	\arrow["\pi", from=1-2, to=2-2]
	\arrow["\pi_n", from=1-1, to=2-1]
	\arrow["{t\mapsto t^n}", from=2-1, to=2-2]
\end{tikzcd}
\]  

\begin{theorem}\label{theorem: compactification of Neron model}
  Let $\pi\colon \X \to \DD$ be any proper polarized family of abelian varieties of dimension $g$.
  Then there exists a finite base change $\pi'\colon \X' \to \DD$ of order $n\in\NN^*$ such that:
  \begin{enumerate}
   \item[$(1)$] The fiber $\X'_0$ is reduced;
   \item[$(2)$] 
the irreducible component $G$ of the central fiber  $\N'_0$ of the Néron model of $\X'$ containing the neutral element is a semi-abelian variety, 
and the quotient space $\N'_0/G$ is a finite group.
  \end{enumerate}
Moreover, the following property holds: 
  \begin{itemize}
\item[$(3)$]
the Néron model $\N'$ is proper 
if and only if
$G$ is an abelian variety.
\end{itemize}
\end{theorem}

\begin{proof}
The first two statements follows from the semi-stable reduction theorem of Grothendieck applied to the $R$-scheme $\X_R$, see~\cite[\S 7.4, Theorem~1]{BLR_Neron_models}.
The third statement is a consequence of~\cite[\S 7.4, Theorem~5]{BLR_Neron_models}. 
\end{proof}

Recall that it may be necessary to do a non-trivial base change for the central fiber of the Néron model to be semi-abelian.
\begin{example}
Let $A$ be an abelian variety of dimension $g$, $f, \sigma\in \Aut(A)$ two automorphisms such that 
$f(0) =\sigma(0) =0$, $\sigma^N = \id$ for some $N$. Then $f$ and $\sigma$ lifts to linear maps on the Lie algebra of $A$
hence commute, and we can form the quotient space 
$\X = (A\times \DD)/G$ where we identify
$(z,t) \sim (\sigma(z), \zeta t)$ with $\zeta$ a primitive $N$-th root of unity. 
Since $f$ and $\sigma$ commute, $f$ descends to $\X$ and defines an automorphism 
$\tilde{f} \colon \X \to \X$.

In general $\X_0$ is singular but is not birational to an abelian variety. Take for instance
$A = E^g$ with $E$ an elliptic curve, $f$ any element in $\SL(g,\ZZ)$ and 
$\sigma = \zeta z$ where $\zeta = -1$, or $\zeta$ is a primitive $3$rd (resp. $4$th)root of unity
when $E= \CC/\ZZ[j]$ (resp. $E= \CC/\ZZ[i]$).
When $g=2$, the minimal resolution of singularities of the quotient $A/G$ is a K3 surface when $\zeta =-1$. One can even obtain 
examples for which $A/G$ is a (singular) rational surface. The easiest example is to take
the simple abelian variety obtained by quotienting by 
the ring of integers of $\QQ(\zeta_5)$
where $\zeta_5$ is a primitive $5$-th root of unity admits two non conjugate complex embedding, and 
the variety $A:=\CC^2/\ZZ[\zeta_5]$ is a simple abelian variety. The multiplication by  $\zeta_5$ on $\ZZ[\zeta_5]$
induces an automorphism $g$ of order $5$, and $A/\langle g \rangle$ is rational. 
 This variety also admits an automorphism with
dynamical degree $>1$ since the group of units of $\QQ[\zeta_5]$ has rank $1$
by  Dirichlet's unit theorem. See~\cite{MR623443,MR1317527} for details.
\end{example}

\subsection{Automorphisms of families of polarized abelian varieties}
Let $\pi\colon \X\to \DD$ be a  family of polarized abelian varieties, and $f\colon \X \dto \X$
be any bimeromorphic map. By the properness of $\pi$, we necessarily have 
$\pi = \pi \circ f$. We denote by  $\Bim(\X)$ the group of bimeromorphic self-maps of $\X$, 
and by $\Aut(\X)$ its group of biholomorphisms.

Since abelian varieties do not contain rational curves, it follows that 
$\Exc(f)$ and $\Ind(f)$ are both contained in $\X_0$, and 
$f_t\colon X_t \to X_t$ is an algebraic biholomorphism for all $t\neq0$.

Recall that for each $t\neq0$, $f_t$ lifts to the universal cover $\pr_t\colon V_t \to X_t$
as an affine map $z \mapsto u(f_t)(z)+x_t$ with $x_t\in V_t$, and  $u(f_t) \in \End(V_t)$ 
so that $u(f_t)$ preserves the lattice $\Lambda_t = \pr_t^{-1}(0)$. 

Since $\Lambda_t$ and $u(f_t)$ are varying continuously (in fact holomorphically) in $t$, 
choosing a path from $t_0$ to $t_1$ determines a canonical isomorphism
$V_{t_0} \to V_{t_1}$ sending $\Lambda_{t_0}$ to $\Lambda_{t_1}$ and conjugating $u(f_{t_0})$ to $u(f_{t_1})$.
 We collect here the following two observations.

\begin{proposition}\label{prop:conjugate}
Let $f\in \Bim(\X)$, and $E$ be any irreducible component of the central fiber of the Néron model of $\X$.
Suppose that the component of $\N_0$ containing $0$ is semi-abelian, and that $f_\N(E)=E$. 

Then the characteristic polynomial of $u(f_t)$ for $t\neq0$ is identical to the one of  $u(f|E)$.
\end{proposition}

\begin{proof}
We may suppose $\X=\N$ and write $f= f_\N$. 
Recall that $\N$ is a family of abelian complex algebraic groups. 
Denote by $E_0$ the connected component of $\N_0$ containing $0$, and pick
any point $x\in E$. Then $E$ is isomorphic to $E-x\cong E_0$, hence is a translate of a semi-abelian
variety. By Lemma~\ref{lemma: differential of automorphism of semiabelian variety} any algebraic automorphism $h\colon E\to E$ thus lifts to an affine map on its
universal cover. We shall denote by $u(h)$ its linear part. 

Choose any local holomorphic section $x \colon \DD_r \to \X$
such that $x(0) = x$. Note that a priori $x$ is only defined in a neighborhood of $0\in \DD$ (i.e., $r<1$) in which case
we replace $\X$ by $\pi^{-1}(\DD_r)$. To simplify notation we shall assume $r=1$ in the remaining of this proof. 

We introduce the following automorphism $\tilde{f} \in \Aut(\N)$ by letting $\tilde{f}_t := f_t - f_t(x(t)) + x(t)$
so that $f_t(x(t)) = x(t)$ for all $t$. Since $f$ and $\tilde{f}$ are isotopic 
we have $\tilde{f}(E)=E$, $u(\tilde{f}|E)= u(f|E)$, and $u(\tilde{f}_t)= u(f_t)$ for all $t\neq0$.

We now look at the differential $d\tilde{f}_t(x(t)) \in \End(T_{x(t)}\N_t)$. We may trivialize $T_{x(t)}\N_t$ in a continuous
way and interpret $d\tilde{f}_t(x(t))$ as a continuous family of endomorphisms of the same complex vector space. 
In particular, the characteristic polynomial of $d\tilde{f}_0(x(0))$  is equal to the one of 
 $d\tilde{f}_t(x(t))$ for $t\neq0$ (as the latter is constant). 
 
We conclude by observing that $d\tilde{f}_t(x(t))$ (resp. $d\tilde{f}_0(x(0))$) is conjugated to $u(f_t)$ (resp. to $u(f|E)$)
since $f_t$ (resp. $f|E$) lifts to a constant linear map on its universal cover by Proposition~\ref{lemma: differential of automorphism of semiabelian variety}.
\end{proof}

\begin{proposition}\label{prop: if for thm for families over disc}
Let $\pi\colon \X\to \DD$ be a polarized family of abelian varieties, and let $\phi \colon \X' \dto \X$
be any base change.  Then for any $f\in \Bim(\X)$, there exists $f'\in \Bim(\X')$ such that $\phi \circ f' = f\circ \phi$, and 
$f$ is regularizable if and only if $f'$ is. 
\end{proposition}
\begin{proof}
We may suppose that $\X'$ is the fibered product over $t\mapsto t^n$
 for some $n\in \NN^*$, i.e., 
 $\X' = \{(x,t) \in \X \times \DD, \,   \pi(x) = t^n\}$ (note that $\X'$ may acquire singularities along the central fiber). 
The automorphism $f'(x,t) = (f(x), t)$ satisfies $\pi' \circ f' (x,t)= \pi'(x,t) := \pi(x)$ hence $f'\in \Bim(\X')$. And since 
$\phi(x,t) = x$,  we have  $\phi \circ f' = f\circ \phi$ by construction.

Observe also that if $f$ is regularizable, then we may suppose that $f\colon \X\to \X$ is a regular map, and 
$f'$ is also regular on the fibered product. 

Suppose conversely that $f'$ is regularizable. We may thus find a model $\Y \dto \X'$ such that the induced map $f_{\Y}\colon \Y\to \Y$ 
is regular. Observe that $\phi \colon \X' \to \X$ is a cyclic cover: the 
 $\ZZ_n$-action given by $\zeta \cdot (x,t) = (x, \zeta \cdot t)$ induces a biholomorphism 
 $\X'/\ZZ_n \mathop{\to}\limits^\sim \X$. This action commutes with $f'$, and by equivariant resolution of singularities
 we may suppose that it lifts to a regular action on $\Y$. 
 Let $\tilde{\Y} := \Y/\ZZ_n$. Then $\tilde{\Y}$ is a model of $\X$, and
  $f_\Y$ descends to a regular automorphism $\tilde{\Y}$. 
  This implies $f$ to be regularizable.   
\end{proof}

\begin{proposition}\label{prop:regul-isogeny}
Let $\pi'\colon \X'\to \DD$ and $\pi\colon \X\to \DD$ be two  families of polarized abelian varieties, and let $\phi \colon \X' \dto \X$
be a meromorphic map such that $\pi' = \pi \circ \phi$ and $\phi_t \colon X'_t \to X_t$ is an isogeny for all $t\in\DD^*$. 
Suppose $f\in \Bim(\X)$ and $f'\in \Bim(\X')$ satisfy $\phi \circ f' = f\circ \phi$. 

Then $f$ is regularizable if and only if $f'$ is. 
\end{proposition}
\begin{proof}
By the resolution of singularities, we may suppose that $\phi \colon \X' \to \X$
is regular.  Suppose that $f\colon \X \to \X$ is an automorphism. 
Let $R$ be the ring of germs of holomorphic functions at $0$ in the unit disk, and $K$ be its fraction field.

Consider the Stein factorization $\Y$ of $\X'\to \X$, 
so that $\X' \mathop{\to}\limits^\varphi \Y \mathop{\to}\limits^\nu \X$ with $\nu$ finite, and $\phi = \nu \circ \varphi$. 
Let $\Y'$ be the normalization of $\Y$. 
Denote by $f'_Y \colon \Y' \dto \Y'$ the map induced by $f'$. 
Since $f \circ \nu =  \nu\circ f'_Y$, $f$ is an automorphism, $\nu$ is finite,
it follows that $f'_Y$ cannot have indeterminacy points, proving that $f'$ is regularizable since $\Y'$ is normal.

To prove the converse, recall that
the dual variety $X^\vee_t$ of $X_t$ is the quotient of the space of anti-linear $1$-forms
$\bar{\Omega}$ by the dual lattice $\Lambda_t^\vee :=\{\ell\in \Omega, \Im \ell(\Lambda_t)\subset \ZZ\}$, see~\cite[Chapter 2, \S 4]{Birkenhake-Lange_AV}. 
Since $X_t$ is polarized, there exists a continuous family of positive definite hermitian
forms $H_t \colon V\times V \to \CC$ such that $\Im H_t(\Lambda_t, \Lambda_t) \subset \ZZ$. These polarizations
induces a canonical isogeny map $\phi_t \colon X_t \to X^\vee_t$ sending $z$ to $H_t(z,\cdot)$. Its degree is 
determined by the type of the polarization by~\cite[Chapter~2, Proposition~4.9]{Birkenhake-Lange_AV}. It follows that
$X^\vee_t$ is also polarized, and 
the family of abelian varieties $X^\vee_t$ varies holomorphically\footnote{the polarization is principal iff  $X^\vee_t$
is isomorphic to $X_t$.}.
We briefly sketch how to see that  $\X^\vee$ extends as a family over $\DD$. 

After a base change, there exists 
a holomorphic map $s\mapsto Z(s)$ from the upper half-plane to the Siegel domain ($Z(s)$ is a symmetric $g\times g$ complex matrix such that $\Im(Z(s)>0$), and such that
$Z(s+1) = Z(s) + D B $ for some matrix $B$ with integral coefficients (see \S\ref{sec:naka-vs-mum} below for more details). And $X_t$ is isomomorphic 
to $\CC^g$ quotiented out by the lattice generated by the vectors $d_1 e_1, \cdots, d_g e_g$ and the columns of $Z(e^{2i \pi s})$. 
It follows from~\cite[Chapter~8.1, (1)]{Birkenhake-Lange_AV} that $X^\vee_t$ is then isomorphic to 
$\CC^g$ quotiented out by the lattice generated by the vectors $e_1, \cdots, e_g$ and the columns of $Z'(s):= D^{-1}Z(e^{2i \pi s})$. 
Observe that $Z'(s +1) = Z'(s)+  B$, and  Mumford's construction (see again \S\ref{sec:naka-vs-mum}) based on toroidal geometry
constructs a degeneration of the family $X^\vee_t$ over $\DD$. 
We also have a canonical meromorphic map 
$\phi\colon \X \dto \X^\vee$  (over $\DD$) restricting to $\phi_t$ for all $t\neq0$.   

We now observe that a bimeromorphism $f \colon \X \dto \X$
 canonically induces a bimeromorphism on the dual family $f^\vee\colon \X^\vee \dto \X^\vee$
and that $f^\vee$ is regularizable iff $f$ is. This follows from the previous argument and the fact that 
$(\X^\vee)^\vee$ is canonically isomorphic to $\X$. 

If $f'$ is regularizable, then $(f')^\vee$ is also regularizable, hence $f^\vee$ 
since we have a map $\X^\vee \dto (\X')^\vee$, and we conclude that $f$ is regularizable as required.
\end{proof}  

\subsection{Proof of Theorem \ref{Theorem: criterion for irreducible family of AV}}

Recall the set-up: $\pi\colon \X\to \DD$ is a  proper family of polarized abelian varieties of dimension $g$, and $f\in \Bim(\X)$. 

\medskip

(1) We assume that $\X$ is a non-degenerating family, and want to show that $f$ is regularizable. 
By definition, this means that there exists a base change $\pi'\colon \X' \to \DD$
which defines a smooth family of abelian varieties. The lift of $f$ to $\X'$ is then automatically 
regular (since $\X'_0$ does not contain any rational curve), and we conclude that $f$ is regularizable
by Proposition~\ref{prop: if for thm for families over disc}.

\medskip

(2) 
Suppose that $f$ is regularizable, and no root of unity is an eigenvalue of $u(f_t)$ for some (hence all) $t\neq0$.
Note that this implies $\lambda_1(f_t)>1$.
We may assume that $f\colon \X\to \X$ is a regular map. 
Replacing $\X$ by a suitable base change, we may also suppose that the central fiber of 
$\X$ is reduced and the central fiber of its Néron model $\N$
is an extension of  a finite group  by a semi-abelian variety (see Theorem~\ref{theorem: compactification of Neron model}).

Let $\X^{\sm}$ be the set of points in $\X$ at which $\pi$ is a local submersion, i.e., $\X^{\sm} = \X \setminus \Sing(\X_0) \supset \X^*$. The Néron mapping property implies
the existence of a canonical (holomorphic) map $\psi \colon \X^{\sm} \to \N$  that extends the isomorphism $\X^*\to \N^*$. 

Replacing $f$ by an iterate, we may suppose that 
each irreducible component of $\N_0$ is fixed by $f_\N$. 
Furthermore, each component $E'$ in $\N_0$ is a translate of the component containing the neutral point of $\N_0$
hence is a translate of a semi-abelian variety $G$. The restriction map $f_\N|E'$ 
can be lifted to an affine map on the universal cover of $E'$, the linear part of which we denote by $u(f|E')$. 
Let $r'$ be the dimension of the maximal split subtorus $T\subset G$, so that $A:= G/T$
is an abelian variety of dimension $g' = g-r'$. By Lemma~\ref{lemma: differential of automorphism of semiabelian variety}, $u(f|E')$ splits 
as a sum of two endomorphisms $u(f|E')= u_T \oplus u_A$.

Let $\tau_i$ and $\alpha_j$ be the eigenvalues counted with multiplicity of $u_T$ and $u_A$ respectively so that 
\begin{align*}
 &|\tau_1|\geqslant |\tau_2|\geqslant \dots \geqslant |\tau_{r'}|;\\
 &|\alpha_1|\geqslant |\alpha_2|\geqslant \dots\geqslant |\alpha_{g'}|.
\end{align*}
On the other hand, let $\mu_j$ be the eigenvalues of  $u(f_t)$ ordered as follows
 $|\mu_1|\geqslant \dots \geqslant |\mu_{g}|$.
By assumption, we have $|\mu_1|>1$, and no $\mu_j$ is a root of unity. 
By Proposition~\ref{prop:conjugate}, the collection $\{\mu_j\}$ is the union of $\tau_j$ and $\alpha_k$. 

We claim that if $r'>0$ then $|\tau_1|>1$. Indeed we have $|\tau_1| = \lambda_1(f|T)$, so that by the log-concavity of dynamical degrees
$|\tau_1|=1$ implies $\lambda_j(f|T)=1$ for all $j$, hence
$|\tau_j| =1$ for all $j$. Since $\tau_1$ is the root of the characteristic polynomial of $u_T$, it
is an algebraic integer, and by Kronecker's theorem we conclude that
$\tau_1$ is a root of unity, which yields a contradiction.

Let us proceed now by contradiction and suppose the family is not degenerating. This is equivalent to assume that $r'\ge1$. 
Let $k$ be the least integer so that $|\tau_1|>|\alpha_k|$. If $|\tau_1|\le |\alpha_{g'}|$, then 
we set $k = g'+1$.

As before, let $E$ be an irreducible component of $\X_0$ such that  $\lambda_k(f|E) = \lambda_k(f_t)$. 
By Propositions~\ref{prop: dynamical degrees on the exceptional divisor} and~\ref{prop: dynamical degree on invariant subvariety}
we infer the existence of an irreducible component $E'$ of $\N_0$ such that  $\lambda_k(f_\N|E') = \lambda_k(f_t)$.
But Corollary~\ref{corollary: dynamical degrees on semi-abelian variety} implies
\[
 \lambda_k(f_t) = \prod_{j=1}^k |\mu_j|^2 = \prod_{j=1}^{k-1} |\alpha_j|^2 \times |\tau_1|^2
\]
whereas
\[
 \lambda_k(f_\N|E') = \max\left\{ \prod_{j=1}^{k-1} |\alpha_j|^2 \times |\tau_1| , \prod_{j=1}^{k} |\alpha_j|^2 \right\} 
\]
proving 
$\lambda_k(f_\N|E')< \lambda_k(f_t)$ since $|\tau_1|>1$.

\subsection{Proof of Theorem~\ref{Theorem: decomposition}} 
We suppose that $f$ is an automorphism of $\X$, and $\deg(f_t^n) \asymp
n^{2k}$. Recall that replacing $f$ by an iterate, and after a suitable base change, we may assume the following. 
First $u(f_t)\colon V\to V$ is unipotent. 
Second the central fiber $\X_0$ is reduced. Let $\N$ be the Néron model of $\X$. 
Third, the component $E_0$ of the central fiber of $\N$
containing $0$ is a semi-abelian variety. 

We let $T$ be the maximal multiplicative torus in $E_0$, and denote by 
$r$ its dimension, so that we have an exact sequence $1 \to T \to E_0 \to A \to 1$ 
where $A$ is an abelian variety of dimension $g' := g- r$.

\smallskip

As in the previous section, we set $\X^{\sm} = \X \setminus \Sing(\X)$, and 
denote by $\psi \colon \X^{\sm} \to \N$
the canonical bimeromorphic morphism. 

By Proposition~\ref{prop: component of X0 with maximal growth}, there exists an irreducible component $E$ of $\X_0$ such that 
$\deg(f^n|E) \asymp \deg(f^n_t)$, and $Z:=\psi(E\cap \X^{\sm})$ is an irreducible subvariety of $\N_0$
which is $f_\N$-invariant. Since $\N_0$ is quasi-projective, it admits a projective compactification by Nagata's theorem,
so that Proposition~\ref{prop: dynamical degrees on the exceptional divisor} applies to the morphism $\psi \colon \X^{\sm} \to \N$, and yields:
\[
n^{2k-1}
\lesssim
\deg_1(f_\N^n|Z) \lesssim
n^{2k}~.
\]
We may assume that the component of $\N_0$ containing $Z$ is $E_0$. 
Proposition~\ref{prop: dynamical degree on invariant subvariety} implies
 $n^{2k-1} \lesssim \deg_1(f_\N^n|E_0)$.
By Corollary~\ref{corollary: growth of first degree on semi-abelian variety}, 
we have $\deg_1(f_\N^n|E_0)\asymp \max \{\deg_1(f_\N^n|T), \deg_1(f_\N^n|A)\}$, therefore
\[n^{2k-1} \lesssim  \max \{n^{r-1}, n^{2(g'-1)}\}
~,\]
which implies the required inequality
$2k \le \max \{ r , 2g - 2r -1\}$.

%%%%%%%%%%%%%%%%%%%%%%%%%%%%%%%

\section{Examples of automorphisms on families of abelian varieties} \label{sec:examples-family}
In this section, we produce examples of automorphisms on families of abelian varieties
for which Theorem~\ref{Theorem: criterion for irreducible family of AV} applies. 
Using the classification of the endomorphism rings of simple abelian varieties, we give a
classification up to isogenies of such families in dimension $\le 5$.

\subsection{Decomposition of families of abelian varieties} 

\begin{proposition}\label{prop:splitting}
Let $\pi\colon \X \to \DD$ be a polarized family of abelian varieties of dimension $g$, 
and $f \colon \X \dto \X$ be a meromorphic map satisfying $\pi \circ f = \pi$, and $f_t$ is an algebraic
group morphism for all $t\in\DD^*$. 

After a suitable finite base change, there exist two families $\pi_j\colon \X_j\to \DD$, $j=0,1$, of polarized abelian varieties, 
bimeromorphic maps $f_j \in \Bim(\X_j)$, 
and a meromorphic map $\psi \colon \X_0\times \X_1 \dto \X$ over $\DD$ such that $\psi (f_0, f_1) = f \circ \psi$ and
the following hold: 
\begin{enumerate}
\item
the induced map $\X_{0,t}\times \X_{1,t} \to \X_t$ is an isogeny for all $t\in \DD^*$;
\item 
$\lambda_1(f_{0,t}) = 1$, and no eigenvalue of $u(f_{1,t})$ is a root of unity. 
\end{enumerate}
\end{proposition}
In particular, we can apply Theorem~\ref{Theorem: criterion for irreducible family of AV} to the family $f_1\colon \X_1\dto \X_1$. 

\begin{proof}
For each $t\neq0$, we apply Theorem~\ref{thm:splitting}. We get two abelian subvarieties $X_{0,t}$ and $X_{1,t}$
that are both $f_t$-invariant. Denote by $f_i$ the restriction of $f$ to $X_i$, $i=0,1$. 
Then $\lambda_1(f_{0,t})= 1$ and  none of the eigenvalues of $u(f_{1,t})$ is a root of unity. 

Observe that $X_{0,t}$ and $X_{1,t}$ form holomorphic families  of abelian varieties. 
Indeed, we may view $X_t$ as a quotient of a fixed complex vector space $V$ of dimension $g$
by a holomorphically varying cocompact lattice $\Lambda_t$. Also, $f_t$ is induced by a holomorphic
family of endomorphisms $u_t\colon V\to V$ such that $u_t(\Lambda_t) = \Lambda_t$. 
We factor the characteristic polynomial of $u_t$ into a product of two polynomials $P$ and $Q$
where $P^{-1}(0)$ consist of roots of unity, and $Q^{-1}(0)$ contains no root of unity. 
By the proof of Theorem~\ref{thm:splitting}, we have $X_{0,t} = \ker P(u_t) / (\ker P(u_t)  \cap \Lambda_t)$ and
$X_{1,t} = \ker Q(u_t) / (\ker Q(u_t)  \cap \Lambda_t)$ so that both families are holomorphic over $\DD^*$. 

Note that both abelian varieties $X_{0,t}$ and $X_{1,t}$ are endowed with a polarization induced from the one
on $X_t$. We now explain briefly why these families extend over $\DD$. 
Fix any $t_0\neq0$, write $X_{t_0} =V/\Lambda_{t_0}$. The polarization induces a hermitian form $H_{t_0}$ on $V$
such that $\Im(H_{t_0})$ is a symplectic form which takes integral values on $\Lambda_{t_0}$. 
Choose symplectic basis of $P(u_{t_0})\cap \Lambda_{t_0}$ and $Q(u_{t_0})\cap \Lambda_{t_0}$ respectively, and
extend them as a symplectic basis of a finite index subgroup of $\Lambda_{t_0}$. Let $g_0$ (resp. $g_1$) be the rank
of $P(u_{t_0})\cap \Lambda_{t_0}$ (resp. of $Q(u_{t_0})\cap \Lambda_{t_0}$).

The family $X'_t = X_{0,t} \times X_{1,t}$ over $\DD^*$ is a quotient of $X_t$ by a finite group of translations
which acts meromorphically on $\X$ because any torsion section extends meromorphically over $0$ (the group law inducing a meromorphic map $\X \times \X \dto \X$). 
It follows that $\X'$ forms a family of polarized abelian varieties over $\DD$.

As we have a symplectic basis of the lattice defining $X'_t$, 
after a finite base change, 
 there exists a 
holomorphic map $s\mapsto Z'(s)$, from the upper half-plane to the Siegel domain, 
an integral matrix $B'$ and a diagonal integral matrix $D'$ such that: 
$Z'(s+1) = Z'(s) + D' B' $;  and 
$X'_{e^{2i \pi s}}$ is isomomorphic  to $\CC^g$ quotiented out by the lattice  $D \ZZ^{g} + Z'(s) \ZZ^{g}$. 
As our reference symplectic basis was a product of two symplectic basis of $P(u_{t_0})\cap \Lambda_{t_0}$ and $Q(u_{t_0})\cap \Lambda_{t_0}$, it follows that
$Z'$ can be written in block-diagonal form 
$Z'(s) =  Z_0(s) \oplus Z_1(s)$
where  $s\mapsto Z_i(s)$, $i=1,2$ are holomorphic maps from the upper half-plane to the Siegel domains. 
The monodromy information on $Z'$ implies the existence of integral matrices $B_i$ and diagonal integral matrices $D_i$ such that: 
$Z_i(s+1) = Z_i(s) + D_i B $ for each $i$ which implies that both families $X_{i,t}$ extends over $\DD$ (see the proof of Proposition~\ref{prop:regul-isogeny} and \S\ref{sec:naka-vs-mum}).

Let $\psi \colon \X_0^*\times \X_1^* \to \X^*$ be the addition map, and observe that by construction $\psi$ induces an isogeny $\X_{0,t}\times \X_{1,t} \to \X_t$  for all $t\neq0$. 
Since $\X_0$ is a subfamily of $\X$, and $f_0$ is the restriction of $f$ to $\X_0$, 
the map $f_0$ extends meromorphically to $X_0$. The same argument proves $f_1$ is meromorphic. 
\end{proof}

\subsection{Division algebras}
We recall some general facts on division algebras that will be used in the next section. 

 Let $B$ be a division algebra which is finite dimensional over $\QQ$. 
Denote by $K$ its center (this is a number field), and write $e=[K:\QQ]$. Then by Artin-Wedderburn theorem, $B\otimes_{K} K^{\alg}$ is isomorphic to $\Mat_d(K^{\alg})$
hence $\dim_K (B) = d^2$ for some integer $d\ge1$ called the index of $B$. 

Pick any $x\in B$. The sub-algebra $\ZZ[x]$ is finitely generated hence $x$ admits a minimal (possibly reducible) polynomial
$F_x\in \QQ[x]$. The $K$-linear map $b \mapsto bx$ is an endomorphism of $B$ whose characteristic polynomial
is the $d$-th power of a monic polynomial $\Prd_x\in K[T]$ of degree $ed$ called the reduced characteristic polynomial.
If $K'|K$ is a finite extension such that $B\otimes_K K' \simeq \Mat_d(K')$ then we have
$\Prd_x (T)= \det (T\id - \phi(x\otimes1))$.  The reduced trace $\trd_{B/K} \colon B\to K$ is defined as the opposite of the coefficient
of $\Prd_x$ in $T^{ed-1}$, and the reduced norm $\Nrd_{B/K}\colon B\to K$ is $\Nrd_{B/K}(x)= (-1)^{ed}\Prd_x (0)$. 
We define $\trd_{B/\QQ}(x) = \tr_{K/\QQ} (\trd_{B/K}(x)) \in \QQ$, and $\Nrd_{B/\QQ}(x) = \mathrm{Nm}_{K/\QQ} (\Nrd_{B/K}(x)) \in \QQ$
where $\tr$ and $\mathrm{Nm}$ are the standard trace and norm attached to the field extension $K/\QQ$. 
The reduced norm $\Nrd_{B/\QQ} \colon B\to \QQ$ is the unique (up to a scalar factor) 
polynomial function $N\colon B\to \QQ$ of degree $ed$ on $B$ such that $N(ab)=N(a)\cdot N(b)$. 

An order $\cO$ in $B$ is a finitely generated $\ZZ$-submodule which is stable by multiplication (hence is a ring), and such that $\cO\otimes_\ZZ \QQ = B$, see~\cite[Chapter 10]{MR4279905}.  Any element in an order is integral, i.e., is annilihated by a monic polynomial with integral coefficients. Conversely, any integral element is contained in some order
that one can choose to be maximal under inclusion (see~\cite[Exercice~10.5]{MR1972204}).

\subsection{Automorphisms of simple abelian varieties}
We refer to~\cite[IV]{MR2514037}, or~\cite[\S 5.5]{Birkenhake-Lange_AV} for details. 

Let $X=V/\Lambda$ be any abelian variety, $\End(X)$ be the ring (under addition and composition) of algebraic group morphisms of $X$,
and define $\End_\QQ(X):=  \End(X)\otimes_\ZZ \QQ$. Since $\End(X)$ can be identified to the set of complex linear maps 
$u\colon V\to V$ such that $u(\Lambda) \subset \Lambda$, it is a ring (under the addition and composition laws), and $\End_\QQ(X)$ is a finitely dimensional 
$\QQ$-algebra.

Suppose $X$ is simple, i.e., contains no non-trivial abelian subvarieties. Then the kernel of any algebraic group morphism $u\colon X \to X$
is finite, hence $\End_\QQ(X)$ has no divisor of zero and is a division algebra. Since $X$ is polarized, $\End_\QQ(X)$ admits a
involution $f\mapsto f'$ (called the Rosati involution) which is an anti-automorphism such that the quadratic form $q(f) := \tr(f f' \colon H^1(X,\QQ) \to  H^1(X,\QQ) )$ is positive definite. 
It follows that $B:=\End_\QQ(X)$ is a division ring which is finite dimensional over $\QQ$, admits an involutive anti-automorphism $b \mapsto b'$
such that $\trd_{B/\QQ} (b b')>0$.  
The classification of such algebras was obtained by Albert.

Denote by $K$ its center as above, so that $B$ has dimension $d^2$ over $K$. Consider also the field
$K_0\subset K$ of elements fixed by the Rosati involution, and  
write $[K:\QQ]=e$, and  $[K_0:\QQ]=e_0$. 
The strong approximation theorem implies that $K_0$ is a totally real number field, and $K$ 
is a totally imaginary quadratic extension of $K_0$ or $K=K_0$, see~\cite[V.5 (5.2) and (5.4)]{Birkenhake-Lange_AV}. 
In particular, it is a CM field. 

Further informations on the structure of $B$ are obtained by considering its class in the Brauer group of $K$. 
We refer to~\cite[IV. Application I]{MR2514037} for details. 

The following table summarizes the informations we know about $B, K$ and $K_0$. 

\begin{table}[ht]
\caption{Endomorphism algebras of simple abelian varieties}
\centering 
\begin{tabular}{|c|c|clc|c|c|}
  \hline
  Albert type & Center (number field $K$) & $B$& Index $d$
  & 
  $B\otimes_\QQ \RR$ &restriction  \\
  \hline
 I &
totally real 
& $B=K$
  & $1$ & $\RR^e$ & $e|g$
\\
II &
totally real
& $D_{a,b}$
 & $2$  & $\Mat_2(\RR)^e$ & $2e|g$
\\
III & 
totally  real 
& $D_{a,b}$ $a$ and $b\ll 0$ 
  & $2$  & $\Hm^e$ & $2e|g$
\\
IV & 
CM  &
second kind 
& $d\ge 1$  & $\Mat_d(\CC)^{e/2}$ & $\frac12 e d^2|g$
 \\
   \hline
 \end{tabular}
 \label{table1}
\end{table}
To avoid confusion with the upper half plane $\HH =\{z\in \CC, \Im(z)>0\}$, we let $\Hm =\{\alpha + \beta i + \gamma j +\delta (ij) , \alpha, \beta, \gamma, \delta \in \RR\}$ be the Hamilton division algebra over $\RR$
with $i^2 = j^2 = -1$ and $ij =-ji$. More generally over a field $K$, we let 
$D= D_{a,b}:=\{\alpha + \beta i + \gamma j +\delta (ij) , \alpha, \beta, \gamma, \delta \in K\}$
where $a,b$ are given elements in $K$ and $i^2 = a$, $j^2 =b$ and $ij =-ji$.
Note that $\trd_{B/K}(\mu) = 2\alpha \in \cO_K$, and $\Nrd_{B/K}(\mu) = \alpha^2 - a \beta^2 - b \gamma^2  + (ab)\delta^2 $, 
and $D_{-1,-1} = \Hm$.

Types I, II and III are referred to as being of the first kind:  the restriction of the involution is then trivial on $K$, i.e., $K=K_0$. 
Division algebras arising in Type IV are called of the second kind. 

\smallskip

In Type I, $B=K$ is any totally real field.

\smallskip

In Type II, $B$ is any quaternion algebra over a totally real field $K$ such that for each embedding $\sigma_i \colon K \to \RR$, we have
$B\otimes_{\sigma_i(K)} \RR \simeq \Mat_2(\RR)$, so that 
$B\otimes_\QQ \RR \simeq   \Mat_2(\RR)^e$, and the reduced trace is given by $\sum_{i=1}^e \tr(M_i)$. 
One says that $B$ is totally indefinite. 
One can show that there exists an isomorphism $B\otimes_\QQ \RR \simeq   \Mat_2(\RR)^e$ such that 
the anti-involution on $B$ corresponds to  $(M_1, \cdots, M_e) \to ({}^t\!M_1,\cdots, {}^t\!M_e)$.

\smallskip

In Type III, $B$ is any quaternion algebra over a totally real field $K$ such that for each embedding $\sigma_i \colon K \to \RR$, we have
$B\otimes_{\sigma_i(K)} \RR \simeq \Hm$, so that 
$B\otimes_\QQ \RR \simeq   \Hm^e$. 
One says that $B$ is totally definite. 
Through this isomorphism, the anti-involution on $B$ corresponds to 
$(b_1, \cdots, b_e) \to (\sigma_\Hm(b_1), \cdots ,\sigma_\Hm(b_e))$,  and the reduced trace is given by $\sum_{i=1}^e \trd_{\Hm}(b_i)$.

\smallskip

Type IV is harder to describe. Then $B$ is any division algebra over a CM field $K$, such that the restriction of the involution to $K$
is the conjugation $\sigma \colon K \to K$ whose fixed point set is $K_0$, and satisfying the following conditions. 
For any finite place $v$ fixed by $\sigma$, then $[B]_v =0$ in the Brauer group of $K_v$; and 
for any finite place $v$ fixed by $\sigma$, then $[B]_v + [B]_{\sigma(v)}=0$.
In that case there exists an isomorphism $B\otimes_\QQ \RR \simeq   \Mat_d(\CC)^e$ such that 
the involution on $B$ corresponds to  $(M_1, \cdots, M_e) \to ({}^t\!\bar{M}_1,\cdots, {}^t\!\bar{M}_e)$. 

\medskip

Suppose now that we are given $f \in \Aut_\bullet(X)$, which we identify with the morphism
$f^* \colon H^1(X,\ZZ) \to  H^1(X,\ZZ)$. 
Then there exists a polynomial with integer coefficients annilihating $f$
and $\det(f) =1$. In other words, $f$ is integral viewed as an element in $B$ and satisfies $\Nrd_{B/\QQ}(f)=\pm 1$. 
In fact, by the uniqueness of the existence of norms on division algebra, we have
\[
\chi_{f} (n) 
= 
\Nrd_{B/\QQ} (n - f)^{\frac{2g}{de}}
\]
for all $n\in\NN$, where $\chi_{f} (T) = \det (T \id - f)$ is the characteristic polynomial of 
$f^* \colon H^1(X,\QQ) \to  H^1(X,\QQ)$,
see~\cite[\S 13.1]{Birkenhake-Lange_AV}. 
In particular, $\lambda_1(f)$ is a  
root of the reduced characteristic polynomial of $f$ in $B$. 
This relation is exploited in~\cite{MR4363581,sugimoto} to construct many examples of automorphisms a simple abelian varieties
whose first dynamical degrees are Salem numbers. 

The condition appearing in Theorem~\ref{Theorem: criterion for irreducible family of AV} that $f^* \colon H^1(X,\QQ) \to  H^1(X,\QQ)$
has no eigenvalue equal to a root of unity is thus equivalent to impose that 
the polynomial $\Nrd_{B/\QQ} (T - f)$ does not vanish at any root of unity.

\subsection{Automorphisms of families of abelian varieties having no cyclotomic factors} \label{sec:aut-no-cyclo}
Let us explain how to construct non-isotrivial families of polarized abelian varieties $\X$, 
and families of algebraic  group isomorphisms  $f_t\in \Aut_\bullet(X_t)$
such that the characteristic polynomial of $u(f_t)$ has no cyclotomic factor, and $X_t$ is simple
for some (hence for a generic) $t\in \DD^*$. 

We consider the moduli spaces of abelian varieties whose ring of endomorphisms contain a given division algebra.
The principle is to start with a division algebra $B$ which is finite dimensional over $\QQ$, together with a positive involutive anti-automorphism, and
a representation $\rho\colon B\to \Mat_g(\CC)$ 
such that there exists $\theta\colon B\to \Mat_{2g}(\QQ)$ 
for which $\theta\otimes 1_\CC = \rho \oplus \bar{\rho}$
(see~\cite[Lemma~9.1.1]{Birkenhake-Lange_AV} for an intrinsic characterization of such representations). 
Note that all examples are described by Albert as above. Then 
we choose an order $\cO \subset B$ and an element $f\in \cO$ such that $\Nrd_{B/\QQ}(f) = \pm 1$.

As explained in~\cite[Chapter 9]{Birkenhake-Lange_AV}, each data $(B,\cO,\rho)$ gives rise to a (several for type IV algebras)
complex variety $\HH(B,\cO,\rho)$ parameterizing triples $(X,H,\imath)$ where $X =\CC^g/\Lambda$ is an abelian variety, $H$ is a positive definite hermitian form on $\CC^g$ corresponding to a polarization of $X$, 
and $\imath \colon \cO \to \End(X) \subset \Mat_g(\CC)$ is an embedding sending the Rosati involution to the involution in $B$ such that
$\imath$ and $\rho$ are equivalent representations. It turns out that  $\HH(B,\cO,\rho)$ is always a hermitian symmetric domain, see Table~\ref{table2} below. 
The moduli space $\M(B,\cO,\rho)$ of such triples up to natural isomorphism is a quotient of $\HH(B,\cO,\rho)$ by a discrete (arithmetic) subgroup, which
admits a canonical structure of quasi-projective variety. Its Satake-Baily-Borel compactification is a projective variety $\overline{\M(B,\cO,\rho)}$.

Any holomorphic map $\DD \to \overline{\M(B,\cO,\rho)}$ and any choice of $\mu\in \cO$ with $\Nrd_{B/\QQ}(\mu) = \pm 1$ gives rise to a family of polarized abelian varieties $X_t$ with automorphisms
$f_t\colon X_t\to X_t$ over $\DD$. 

In the next table, we set 
\begin{align*}
\Siegel_g
&:=\{Z \in \Mat_{g\times g}(\CC)|\, {}^t\! Z =  Z \text { and } \Im (Z) >0\}
\\
\mathcal{H}_g
&:=\{Z \in \Mat_{g\times g}(\CC)|\, {}^t\! Z = - Z \text { and } \id - {}^t\! \bar{Z}Z >0\}
\\
\mathcal{H}_{r,s}
&:=\{Z \in \Mat_{r\times s}(\CC)|\,  \id_s - {}^t\! \bar{Z}Z >0\}
\end{align*}
\begin{table}[ht]
\caption{Moduli space of abelian varieties with endomorphism structure}
\centering 
\begin{tabular}{|c|c|c|c|}
\hline
  Albert type 
  & 
  dimension of the abelian variety
  & 
 symmetric domain 
 &
 dimension of the moduli space  
 \\
  \hline
 I 
&
$l  e$  
& 
$(\Siegel_l)^e$ 
& 
$\frac{e}2 l(l+1)$
 \\
 II
&
$2l  e$  
& 
$(\Siegel_l)^e$ 
& 
$\frac{e}2 l(l+1)$
\\
III
&
$2l  e$  
& 
$(\mathcal{H}_l)^e$ 
& 
$\frac{e}2 l(l-1)$
\\
IV
&
$d^2 e l = 2 d^2 e_0 l$
&
$\prod \mathcal{H}_{r_j,s_j}$
&
$\max_{r_j + s_j = dl} \sum_{j=1}^{e_0} r_j s_j$
\\
   \hline
 \end{tabular}
\label{table2}
\end{table}

Explaining in details all possible outcomes for all types of endomorphism algebras of simple abelian varieties would be too lengthy. 
We discuss only those examples that arise in dimension $\le 5$ (up to isogenies). 

\smallskip

\subsubsection*{Type I}~\cite[\S 9.2]{Birkenhake-Lange_AV}.
Let $K$ be any totally real field of degree $e\ge 2$ over $\QQ$ with ring of integers $\cO_K$. 
Write $g = l \times e$ for some $l\in \NN^*$. 
Fix non-equivalent embeddings $\{\imath_j \colon K \to \RR\}_{1\le j \le e}$. 
Pick $\mu\in \cO_K^*$ any unit which is not a root of unity (it always exists by Dirichlet's unit theorem if $e\ge 2$).  

For any $Z\in (\Siegel_l)^e$, introduce the lattice $\Lambda_z$
given as the image of 
\[\lambda_z\colon \cO_K^l\oplus \cO_K^l \to \CC^g, 
\lambda_z(\alpha, \beta)= (\alpha_1 Z_1 +\beta_1, \cdots,  \alpha_e Z_e +\beta_e)~.\]
Here we have $\alpha = (\alpha^1 , \cdots, \alpha^l) \in \cO_K^l$, $\beta = (\beta^1 , \cdots, \beta^l) \in \cO_K^l$, and we write
$\alpha_j = (\imath_j(\alpha^1) , \cdots, \imath_j(\alpha^l))$, and $\beta_j =(\imath_j(\beta^1) , \cdots, \imath_j(\beta^l))$. 
Then $X_z := \CC^g /\Lambda_z$ is an abelian variety with polarization 
\[
E(v,w)=
\sum_{j=1}^e \Im(v_j  (\Im Z_j)^{-1} \bar{w}_j)~, 
\]
which is integral on $\Lambda_z$. The diagonal map $(\mu_1 \id_l, \cdots, \mu_e \id_l)$ on $\CC^g$ is an automorphism of $\Lambda_z$
hence we get an automorphism $f_z\colon X_z\to X_z$ whose characteristic polynomial is irreducible and has no
cyclotomic factor. The first dynamical degree of $f_z$ equals $\max\{|\mu_j|^2\}$ by Corollary~\ref{corollary: growth of first degree on semi-abelian variety}.

\begin{remark}
Observe that the subspace $\RR^g\subset \CC^g$ intersects $\Lambda_z$ in a lattice, and 
is invariant under the diagonal map. It follows that $f_z$ admits an invariant totally real
 compact torus of dimension $g$. 
\end{remark}

The manifold $\HH(K,\cO_K,\rho)$ is isomorphic to $(\Siegel_l)^e$, has dimension $\frac{e}2 l(l+1)$, and 
 $\M(K,\cO_K,\rho)$ is obtained by quotienting out of  by an arithmetic subgroup of $e$ copies of the symplectic group in dimension $2l$. 
 When $l=1$,  we have $e=g$ and $\M(K,\cO_K,\rho)= \HH^g/\PSL(2,\cO_K)$
 is a Hilbert-Blumenthal modular variety. We refer to~\cite[Chapter 1]{MR930101} for the proof that it is a quasi-projective variety of dimension $g$ 
 that can be projectively compactified by adding finitely many points. 

\begin{remark}
There exist smooth families of polarized abelian varieties defined over a \emph{compact} Riemann surface $X\to B$
that admit a family of automorphisms with $\lambda_1 >1$. 
Indeed since the Hilbert-Blumenthal variety is the complement of a finite set in a projective varieties, it contains
many compact curves (of high genus, see~\cite{MR3862056}).
\end{remark}

\subsubsection*{Type II  \textup{\&} III}
The precise construction is given in \cite[\S 9.3--5]{Birkenhake-Lange_AV} or~\cite{MR156001}. 

\smallskip

\subsubsection*{Type IV}
We refer to~\cite[\S 9.6]{Birkenhake-Lange_AV}
for details of the construction in this case. Observe that there is no non-trivial family of such type in dimension $\le 5$.

\subsection{Classification in low dimension}
We list all up to isogeny by Proposition~\ref{prop:splitting} positive dimensional families of polarized abelian varieties which admit a family
of automorphisms $f_t$ such that the characteristic polynomial of $f_t^* \colon H^1(X_t, \QQ) \to H^1(X_t, \QQ) $
has no cyclotomic factors (in particular $\lambda_1(f_t)>1$). 

Any abelian variety $X$ is isogenous to a direct product  $X_1^{k_1} \times \cdots \times X_m^{k_m}$
where $X_1, \cdots, X_m$ are  simple non isogenous abelian varieties. 
The ring $\End_\QQ (X_1\times \cdots \times X_m)$ is then isomorphic to $\Mat_\QQ(k_1, \End(X_1))\times \cdots \times \Mat_\QQ(k_m, \End(X_m))$ where $\End_\QQ(X_i)$ 
is a division algebra of type I, II, III or IV. Thus, the classification of families of automorphisms as in Theorem~\ref{Theorem: criterion for irreducible family of AV}  is reduced to a combinatorial question.

To exclude certain families, we observe that an abelian variety 
$X$ with $\End_\QQ(X) = \QQ$ does not carry any automorphism with  $\lambda_1>1$. 
In particular, in the type I case, we can always suppose $e>1$.

Also we shall use the work of Shimura, namely~\cite[Theorem~5]{MR156001}, to rule
out some cases: 
\begin{itemize}
\item
Type III and $g/2e=1$ then $X$ is not simple;
\item
Type IV and $\sum r_\nu s_\nu = 0$ then $X$ is not simple;
\item
Type IV $l=2$, $d=1$, $r_\nu=s_\nu=1$ for all $\nu$ then the endomorphism algebra is included in 
a type III example;
\item
Type IV and $l=1$, $d=2$, $r_\nu=s_\nu=1$ for all $\nu$, then $X$ is not simple.
\end{itemize}

In each case, we denote by $m$ the maximal dimension of families of abelian varieties
that one can obtain. In the table below, by convention  by a unit we mean a unit
not lying in the group of roots of unity.

\bigskip

\noindent $\boxed{g=2}$

\begin{itemize}
\item[2.1]
generically non-simple: $X_t= E_t^2$, $\dim(E_t)=1$;  $f$ is determined 
by a unit in a totally real quadratic field  (i.e., a matrix in $\SL(2,\ZZ)$), $m=1$;
\item[2.2]
generically simple with endomorphism algebra of type I, $l=1$ , $e=2$; $f$ is determined 
by a unit in a totally real quadratic field  not in $\mathbb{U}_\infty$, $m=2$.
\end{itemize}

\begin{remark}
Abelian surfaces whose endomorphism algebra are of type II, III or IV
are rigid (see Table~\ref{table2}). A direct proof of these facts is given in~\cite{MR3619745}.
\end{remark}

\noindent $\boxed{g=3}$

\begin{itemize}
\item[3.1]
generically non-simple:  $X_t=E_t^3$, $\dim(E_t)=1$; $f$ is determined 
by a unit in a cubic field (i.e., a matrix in $\SL(3,\ZZ)$, $m=1$;
\item[3.2]
generically simple with endomorphism algebra of type I, $l=1$ , $e=3$; $f$ is determined 
by a unit in a totally real cubic field, $m=3$.
\end{itemize}

\begin{remark}
Simple abelian $3$-folds whose endomorphism algebra is of type I 
and has invariants $l=3$ and $e=1$ do not carry automorphisms with $\lambda_1>1$. 
Types II, III, IV are not possible since $g$ is odd. Also any automorphism 
of the family $X_t = E_t \times Y_t$ where $Y_t$ is as in 2.2 preserves $E_t$, 
thus, there is an eigenvalue of $f_*$ which is a root of unity for any $f \in\End(X_t)$.
\end{remark}

\noindent $\boxed{g=4}$

\begin{itemize}
\item[4.1]
generically non-simple: $X_t=E_t^4$, $\dim(E_t)=1$;  $f$ is determined 
by a unit in a quartic field  (i.e., a matrix in $\SL(4,\ZZ)$), $m=1$; 
\item[4.2]
generically non-simple:  $X_t=E_t^2\times A_t$,  $\dim(E_t)=1$, $\dim(A_t)=2$ as in 2.2; $f$ is determined 
by two units in two real quadratic fields, $m=2$;
\item[4.3]
generically non-simple:  $X_t=A^1_t\times A^2_t$,  $\dim(A^i_t)=2$ as in 2.2; $f$ is determined 
by two units in two real quadratic fields, $m=4$;
\item[4.4]
generically non-simple:  $X_t=(A_t)^2$,  $\dim(A_t)=2$ as in 2.2; $f$ is determined 
by a matrix in $\GL(2,\cO_K)$ for some real quadratic field $K$, $m=2$;
\item[4.5]
generically simple with endomorphism algebra of 
type I, $l=1$, $e=4$;
$f$ is determined by a unit in a totally real quartic field, $m=4$;
\item[4.6]
generically simple with endomorphism algebra of 
type I, $l=2$, $e=2$, $f$ is determined by a unit in a real quadratic field, $m=6$;
\item[4.7]
generically simple with endomorphism algebra of 
type II, $e=2$, $l=1$, $f$ is determined by $\mu\in B$ a totally indefinite quaternion algebra
over a real quadratic field with $\Nrd_{B/\QQ}(\mu) = \pm1$, $m=2$;
\item[4.8]
generically simple with endomorphism algebra of 
type II, $e=1$, $l=2$, $f$ is determined by $\mu\in B$ a totally indefinite quaternion algebra
over $\QQ$ with $\Nrd_{B/\QQ}(\mu) = \pm1$, $m=3$;
\end{itemize}
\begin{remark}
Observe that in the last two cases 4.7 and 4.8,  the existence of $\mu$ with reduced characteristic polynomial 
having no cyclotomic factor puts some restriction on the algebra. In case 4.6, if $B$ is given by 
$i^2 = a$ and $j^2 = b$, we need the existence of a solution of the quadratic equation
$x^2 - a y^2 - b z^2 + (ab) t^2 = \epsilon$, with $(x,y,z,t) \in \QQ$, $\epsilon = \pm1$,  and 
$x \notin\{0, \pm 1\}$ if $\epsilon =1$ and $x\neq0 $ if $\epsilon = -1$.
\end{remark}

\begin{remark}
Simple abelian $4$-folds whose endomorphism algebra is of type I (resp. type IV)
and has invariants $l=4$ and $e=1$ (resp. $d=1$, $(e_0,l)=(1,2)$) do not carry automorphisms with $\lambda_1>1$. 
Simple abelian $4$-folds whose endomorphism algebra is of
type III with $e=2$, $l=1$ (resp. type IV with $d=1$, $l=1$) are rigid. 
Finally for any totally definite quaternion algebra of type III with $e=1$ defined over $\QQ$ 
the elements with $\Nrd_{B/\QQ}(\mu) = \pm1$ form a finite group.  
\end{remark}

\noindent $\boxed{g=5}$
\begin{itemize}
\item[5.1]
generically non-simple: $X_t=E_t^5$, $\dim(E_t)=1$; $f$ is determined 
by a unit in a quintic field (i.e., a matrix in $\SL(5,\ZZ)$), $m=1$; 

 \item[5.2]
generically non-simple:  $X_t=E_t^2 \times Y_t$, $\dim(E_t)=1$ as in 2.1, $\dim(Y_t)=3$ as in 3.2;
$f$ is determined  by two units in the ring of integers of a totally real quadratic and cubic fields respectively, $m=4$;

 \item[5.3]
generically non-simple:  $X_t=A_t \times E_t^3$, $\dim(A_t)=2$ as in 2.2, $\dim(E_t)=1$ as in 3.1;
$f$ is determined  by two units in the ring of integers of  a real quadratic and of a cubic field, $m=3$;

 \item[5.4]
generically non-simple:  $X_t=A_t \times Y_t$, $\dim(A_t)=2$ as in 2.2, $\dim(Y_t)=3$ as in 3.2;
$f$ is determined 
by two units in the ring of integers of  a totally real quadratic and cubic fields respectively, $m=5$;
\item[5.5]
generically simple with endomorphism algebra of 
type I, $l=1$, $e=5$, $f$ is determined by a unit in a totally real quartic field, $m=5$.
\end{itemize}

\begin{remark}
Simple abelian $5$-folds whose endomorphism algebra is of type I 
and has invariants $l=5$ and $e=1$ do not carry automorphisms with $\lambda_1>1$. 
\end{remark}

%%%%%%%%%%%%%%%%%%%%%%%%%%%%%%%%%%%%%%%%%%%%%%%

\section{Translations on families of abelian varieties}\label{sec:translations}
In this section, we focus our attention to families of translations, and prove Theorems~\ref{Theorem: translations are regularizable} and~\ref{thm:orbit-closure} from the introduction.

\subsection{Relative toric manifolds over $\DD$}\label{sec:relative}
Following Mumford, we build a suitable compactification of denegeration of abelian varieties, 
by using special families of toric varieties. 

Let  $R$ be the  discrete valued ring of germs of holomorphic functions on $\DD$ at $0$, and let $K$ be the fraction field of $R$.
We shall consider toric $R$-schemes (as defined in~\cite[IV.3]{KKMS} or~\cite[Chapter 3.5]{MR3222615})
which are smooth over $R$ and whose generic fiber
is equal to $\GGm^g$ over $K$. Any such scheme gives rise to a 
smooth complex manifold $X$ containing $\DD^* \times \GGm^g$ as an open dense subset such that:
\begin{itemize}
\item
the projection map onto the first factor extends to a surjective holomorphic map $\pi \colon X \to \DD$ and $\pi^{-1}(\DD^*) = \DD^* \times \GGm^g$ ;
\item
the action of $\GGm^g$ by multiplication on the second factor of $\DD^* \times \GGm^g$ extends holomorphically  to 
$X$.
\end{itemize}
For convenience, we shall refer any such manifold to a relative toric manifold of dimension $g$ over $\DD$. 
Let us now describe the combinatorial data that encodes toric $R$-schemes satisfying the conditions above, 
and explain how to build the relative toric manifold from these data.

\smallskip

Let $M$  be the lattice of characters
of $\GGm^g$, so that an element $m\in M$ is a group morphism $\me(m)\colon \GGm^g \to \GGm^1$. Define 
$N= \Hom ( M, \ZZ)$ its lattice of co-characters (a point $n\in N$ corresponds to a one-parameter subgroup of $\GGm^g$). 
Write  $\tilde{M}_\RR:= M_\RR \times \RR_+$ and  $\tilde{N}_\RR:= N_\RR \times \RR_+$. 

An admissible\footnote{this is a purely local terminology} fan $\Delta$ is a (possibly infinite) collection of rational polyhedral cones
in $\tilde{N}_\RR$ that is closed under taking intersections and faces. 
We impose moreover the following conditions:
\begin{itemize}
\item
no cone of $\Delta$ contain a linear subspace of $N_\RR$ (i.e., $\sigma$ is strongly convex);
\item
each cone $\sigma\in \Delta$ is simplicial (but not necessarily regular\footnote{
 a simplicial cone $\sigma$ of dimension $k$ is regular if
the semi-group $\sigma\cap \tilde{N}$ is generated by $k$ elements});
\item
no cone in contained in $N_\RR\times \{0\}$.
\end{itemize} 
The last two conditions respectively correspond to the fact that we want the resulting toric variety $\X$ to be smooth, and the generic fiber over $\DD$ to be the split torus. 

A ray of $\Delta$ is a one-dimensional cone.

\smallskip

For each cone $\sigma$, we define the semi-group $S_\sigma= \{\tilde{m }\in M\times \ZZ, \, \langle \tilde{m}, \tilde{n} \rangle \ge 0 \text{ for all } \tilde{n} \in \sigma\}$.
Following, e.g.,~\cite[\S 1.2]{MR922894}, we associate to each cone $\sigma$  the relative (affine) toric manifold of dimension $g$ over $\DD$
by setting 
\[
X_\sigma = 
\{ x \colon S_\sigma \to \CC, \, x(\tilde{m}_1+\tilde{m}_2)= x(\tilde{m}_1) 
\cdot x(\tilde{m}_2), \, x(0)=1\ \text{ and } |x(0,1)|< 1\} 
\]
which we endowed with the topology of the pointwise convergence, and the unique structure of complex manifold for which 
any $\tilde{m} \in S_\sigma$ induces a holomorphic function. 

The manifold $X_\sigma$ is equipped with a canonical holomorphic map $\pi_\sigma(x):= x(0,1)\in\DD$, and 
$\CC[\tilde{M}]$ forms a subalgebra of the field of meromorphic functions on $X_\sigma$ that are regular on $X_\sigma^*=\pi_\sigma^{-1}(\DD^*)$
(the point $(0,1)$ corresponds to $\pi_\Delta$). 

Note that by construction an element in $\CC[\tilde{M}]$ is holomorphic in $X_\sigma$ if and only if
it belongs to $\CC[S_\sigma]$. And this algebra generates the structure sheaf $\cO_{X_\sigma}$ as a $\cO_{X_\sigma}$-module.
It is customary to write $\be(\tilde{m})$ for the holomorphic function induced by $\tilde{m} \in S_\sigma$ on $X_\sigma$.

The torus $\GGm^g$ acts on itself by multiplication, hence on the $\CC$-algebra $\CC[M]$ and the trivial extension
 to $\CC[\tilde{M}]$ (leaving the vector $(0,1)$ fixed) turns $X_\sigma$ into a relative toric manifold over $\DD$.
\begin{remark}
 As $\sigma$ is regular, we may suppose that $\tilde{e}_0 =(e_0,1), \cdots , \tilde{e}_k =(e_k,1) \in N\times \NN$ such that $ \sigma \cap \tilde{N}$ is generated by $(\tilde{e}_0, \cdots, \tilde{e}_k)$. We may suppose $e_0=0$ and complete the family of vectors $e_i$ such that $e_1, \cdots, e_g$ form
 a basis of $N$. In the canonical basis $(0,1), (e_i^*,0)$ of $M\times \ZZ$, 
$S_\sigma$ is equal to  $\{( a_1, \cdots, a_g,b), \in \ZZ^{g+1}, b\ge 0, b+a_j \ge 0 \text{ for all  } j=1,\cdots, k\}$
which is generated by $ (e_i^*,0),  (-e_i^*,1)$ for $j=1,\cdots, k$, and $\pm(e_i^*,0)$ for $j= k+1, \cdots , g$.
Hence $X_\sigma$ is biholomorphic to 
\[\{ (z_1, w_1, \cdots, z_k, w_{k}, y_{k+1}, \cdots, y_g, t ) \in  (\CC^*)^{2k}\times \CC^{g-k}  \times \DD,\, 
z_1w_1 = t, \cdots, z_kw_k = t
\}~.\]
 \end{remark}

Any inclusion of cones $\sigma \subset \sigma'$ yields an inclusion $S_\sigma \supset S_{\sigma'}$ hence induces by restriction an embedding $X_{\sigma} \subset X_{\sigma'} $
preserving the action by $\GGm^g$ and the projection to $\DD$. 
We define $X(\Delta)$ as the disjoint union of all $X_\sigma$ where $\sigma$ ranges over all cones of the fan $\Delta$, 
and $X_\sigma$ and $X_{\sigma'}$ are patched along $X_{\sigma\cap \sigma'}$ (interpreted as an open subset of both).

Any relative toric manifold over $\DD$ is isomorphic to $X(\Delta)$ for some fan satisfying the conditions above. 
We shall denote by $\pi_\Delta \colon X(\Delta) \to \DD$ the canonical projection map. As before,  $\CC[\tilde{M}]$
forms a subalgebra of the field of meromorphic functions on $X(\Delta)$ that are regular on $X(\Delta)^*$, and the action of 
$\GGm^g$ on $X(\Delta)$ is the dual to its natural action on $\CC[\tilde{M}]$.
\begin{lemma}\label{lem:one-parameter}
Let $\phi\colon \DD^*\to \GGm^g$ be any holomorphic map that is meromorphic at $0$, and let $n_\phi \in N$
be the unique co-character satisfying
\[
\langle n_\phi,m \rangle:= \ord_t \me(m)(\phi(t)) \text{ for all } m\in M
~.\]
Then for any fan $\Delta$,  the map $t \mapsto (\phi(t),t)$ extends as a holomorphic function $\DD\to X(\Delta)$ if and only if $\RR_+\cdot (n_\phi,1)$
belongs to a cone in $\Delta$.
\end{lemma}

\begin{proof}
Let $\Delta$ be any admissible fan. Then $\tilde{\phi}(t):=  (\phi(t),t)$ extends as a holomorphic function $\DD\to X(\Delta)$
if and only if there exists a cone $\sigma\in \Delta$ such that $\tilde{\phi} \colon \DD\to X_\sigma$ is holomorphic.

Since $\CC[S_\sigma]$ generates the structure sheaf $\cO_{X_\sigma}$ as a $\cO_{X_\sigma}$-module, 
$\tilde{\phi} \colon \DD\to X_\sigma$ is holomorphic if and only if $\me(\tilde{m})(\tilde{\phi})$ is holomorphic at $0$
for any $\tilde{m}\in S_\sigma$. The latter is equivalent to the condition $\ord_t(\me(\tilde{m})(\tilde{\phi}(t)))\ge0$ for all $\tilde{m}\in S_\sigma$.

If $(n_\phi,1)$ does not belong to $\sigma$, then one can find $\tilde{m}=(m,k) \in S_\sigma\subset M\times \ZZ$ such that
$\langle (n_\phi,1), \tilde{m}\rangle <0$, and 
\begin{align*}
\ord_t(\me(\tilde{m})(\tilde{\phi}(t)))
&=
\ord_t(\me(m,k)(\tilde{\phi}(t)))
\\
&= \ord_t(\me(m)(\phi(t))) +k 
= \langle (n_\phi,1), \tilde{m}\rangle <0
\end{align*}
so that $\tilde{\phi}$ does not extend holomorphically. 

If $(n_\phi,1)$ belongs to $\sigma$, then the same computation shows that $\tilde{\phi}$ extends holomorphically. 
\end{proof}

\subsection{Analytic Mumford's construction following Nakamura} \label{sec:naka-vs-mum}
We describe in this section the construction of a proper model of a family of polarized abelian varieties over an open disc as done in~\cite{Nakamura_Neron_models}.
Recall our notation for the upper half-plane $\HH = \{ l\in\CC|\ \Im(l) > 0\}$ and for the Siegel domain
$\Siegel_g = \{M \in \Mat_{g\times g}(\CC)|\ {}^t\! M= M \text { and } \Im(M)>0\}$ for any $g\ge1$.

Let $\pi\colon \X\to \DD$, $\L\to \X$ be any  family of polarized abelian varieties. 
Fix any $t_\star \in \DD^*$, write $X_\star=X_{t_\star}$
as the quotient of some complex vector space $V$ of dimension $g$  by a co-compact lattice $\Lambda_\star$. 
The polarization induces a positive-definite hermitian form $H$ on $V$ whose imaginary part $E= \Im H$ takes integral values on $\Lambda_\star$ 
(note that $H (v,w) = E(iv,w)+ i E (v,w)$). 
Since $E$ is a skew-symmetric form on $V$, we may
choose a symplectic basis $f_1, \cdots , f_g, e_1, \cdots , e_g$ of the lattice $\Lambda_\star$ so that 
\[
E= 
 \begin{pmatrix}
  0& D \\
  -D & 0 \end{pmatrix}
\]
for some diagonal integral matrix $ D= \diag (d_1, \cdots, d_g)$ with $d_1|\cdots |d_g$.
In the basis $\tilde{e}_i = \frac1{d_i}e_i$ of $V$, the lattice is generated by 
the vectors given by the columns of the following period matrix
\[ \Pi_\star = \begin{pmatrix}
 a_{11} & \dots & a_{1g}& d_1& 0 & \dots & 0 &\\
a_{21} & \dots & a_{2g}& 0 & d_2& \dots & 0 &\\
  &\ddots&  &&&\ddots&\\
 a_{g1} & \dots & a_{gg}& 0& 0 & \dots &  d_g
 \end{pmatrix}
 = (Z_\star~D)
\]
where $Z_\star$ belongs to $\Siegel_g$. 
Now consider the universal covering map $\be\colon \HH\ \to \DD^*$ with $\be(s) = \exp(2\pi is)$. 
Identifying the universal covers of $X_t$ to $V$ so that the polarization $\L_t\to X_t$
defines the same hermitian form $H$ for all $t$, we may locally follow holomorphically 
the lattices $\Lambda_t = \ker (V \to X_t)$ such that $e_1, \cdots, e_g$ belongs to $\Lambda_t$, and 
since $\HH$ is simply connected, we get a holomorphic function $s\mapsto Z(s) = [a_{ij}(s)] \in \Siegel_g$ 
so that $\Lambda_{\be(s)}$ is generated by the column of $\Pi(s)= (Z(s)~ D)$. 
Here is a diagram summarizing the construction to follow:\[
\begin{tikzcd}[ampersand replacement=\&]
	\& {\mathbb{G}_m^g\times \mathbb{D}\circlearrowleft \Gamma} \\
	{Z_{\star}\in\mathfrak{H}_g} \& {\mathcal{X}\supset X_{\star}} \\
	{s\in\mathfrak{H}} \& {\mathbb{D} \ni t}
	\arrow["\pi", shift right=3, from=2-2, to=3-2]
	\arrow["\exp"', from=3-1, to=3-2]
	\arrow[shift left=4, from=2-1, to=3-1]
	\arrow["Z", shift right=5, curve={height=-12pt}, from=3-1, to=2-1]
	\arrow[from=2-1, to=2-2]
	\arrow["{Z_0}"', from=3-2, to=2-1]
	\arrow[shift right=3, from=1-2, to=2-2]
\end{tikzcd}
\]
In the rest of the discussion, we shall replace our original family by the 
family $X'_ t:= V/\Lambda'_t$ where $\Lambda'_t := \Lambda_t + \ZZ \tilde{e}_1 + \cdots +\ZZ \tilde{e}_g$
is generated by the columns of  $\Pi'(s)= (Z(s)  \id_g)$.
Then $\X'$ forms a family of abelian varieties which is principally polarized (i.e., the canonical map $X'_t \to (X')^\vee_t$ is an isomophism), 
and we thus have a family of isogenies $\phi_t\colon X_t \to X'_t$. 

\medskip

 Since $(Z(s+1) ~ \id_g)$ and $(Z(s)~\id_g)$ define the same polarized abelian variety, there exists an element
$M\in \mathrm{Sp}(2g,\ZZ)$ such that $M\cdot Z(s+1) = Z(s)$ for all $s$, see~\cite[Proposition~8.1.3]{Birkenhake-Lange_AV} (the action is by generalized Möbius transformations). 
The matrix $M$ encodes the monodromy action of $\pi_1(\DD^*)$ on $H^1(X_\star, \ZZ)$ which fixes $\tilde{e}_1, \cdots, \tilde{e}_g$, hence is upper block-triangular. 
By the Griffiths-Landman-Grothendieck's monodromy theorem (see, e.g., \cite[Theorem 13.7.3]{MR3727160})  $M$ is moreover quasi-unipotent. 
Up to base change we may (and shall) suppose that $M$ is unipotent. 
This means that we can find an integral-valued matrix $B\in M(\ZZ,g)$ such that
$Z(s+1)= Z(s) +B$.
\begin{lemma}\cite[Lemma 2.3]{Nakamura_Neron_models}\label{lem:6.3}
The matrix $B$ is  symmetric positive semi-definite, and there exists a holomorphic map 
$Z_0\colon \DD \to\Siegel_g$  such that 
$Z(s)= Z_0(\be(s)) +Bs$. 
\end{lemma}

Denote by $r'$ the rank of the matrix $B$. After changing the basis $e_i$, we may suppose that 
\begin{align}\label{eq: Nakamura's data of the family}
 B = \begin{pmatrix}
      0 & 0 \\ 0 & B'
     \end{pmatrix}; && 
     Z_0(t) = \begin{pmatrix}
   Z_1(t) & Z_2(t)\\ {}^{t}\!Z_2(t) & Z_3(t)
     \end{pmatrix},
\end{align}
where $B'\in \Mat(r',\ZZ)$ is a positive definite integral matrix of size $r'$, $g' := g-r'$, and $Z_1\colon \DD \to \Sym(g',\CC)$,  $Z_2\colon \DD \to \Mat(g'\times r',\CC)$,
 $Z_3\colon \DD \to \Sym(r',\CC)$,   are holomorphic. Here $ \Sym(h,\CC)$ denotes the space of complex symmetric square matrices of size $h$.
\begin{remark}
In order to be consistent with the notation of the rest of the paper, we insist on writing $r'$ for the rank of $B$. 
This is in conflict with the notation used by Nakamura who used $g'$ for this rank. 
\end{remark}
\begin{remark}
The number $r'$ equals the dimension of maximal torus in a semi-abelian component of the Néron model of $\X$ in the case when $\X$ has semi-abelian reduction. In particular,
$r'=0$ iff $X_t$ is not degenerating.
\end{remark}
Let us explain now how to recover the family $\X^*$ from~\eqref{eq: Nakamura's data of the family}. 
Observe that for a fixed $t\in \DD^*$ with $t=\be(s)$, then $X_t$ is the quotient of $V\simeq \CC^g$ by the 
lattice generated by the column of $\Pi(s)$. Taking the quotient by $\ZZ e_1\oplus \cdots \oplus \ZZ e_g$
first, we get $X_t = \GGm^g/\Gamma$ where $\Gamma = \ZZ^g$ acts on the split torus
by multiplication 
\begin{equation}\label{eq:action}
\gamma \cdot (z'_1, \cdots , z'_g) = \left(
z'_1\, \prod_{j=1}^g \be(a_{1j}(s)\gamma_j), \cdots, 
z'_g\, \prod_{j=1}^g\be(a_{gj}(s)\gamma_j)
\right). 
\end{equation}

The family $\X^*$ can be thus obtained as the quotient of $\GGm^g\times\DD^*$ by an action of $\Gamma$
which we now describe in dual terms. Recall from the previous section that $M$ is the set of characters of $\GGm^g$. 
The universal covering map $V\to \GGm^g$ is given in the coordinates $e_i$ by the map 
$v =\sum v_i {e}_i \mapsto  (\be(v_1), \cdots, \be(v_g))$. 
To exploit the block-diagonal form of $B$ as in~\eqref{eq: Nakamura's data of the family}, we write the coordinates on $\GGm^g$ as
$(z_1, \cdots, z_{g'}, w_1, \cdots , w_{r'}) = (\be(v_1), \cdots, \be(v_g))$. 
A character $m\in M$ can thus be written as $m = z^p w^q$ with $p\in\ZZ^{g'}$ (resp. $q\in\ZZ^{r'}$) a character of $\GGm^{g'}$ (resp. of $\GGm^{r'}$). 

The action of $\Gamma$ can then be described on $M\times \NN$ as follows.
For any  $\gamma= (\alpha, \beta)\in \Gamma$, where $\alpha\in \ZZ^{g'}$ and $\beta \in\ZZ^{r'}$
we have (with $t=\be(s)$):
\begin{align*}
{\gamma}^* (z^p) &= \be(\alpha \cdot Z_1(t)\cdot {}^t\!p + \beta \cdot  {}^t\!Z_2(t)\cdot  {}^t\!p) \, z^p;\\
{\gamma}^* (w^q t^l) &= \be( \alpha \cdot Z_2(t)\cdot {}^t\!q + \beta \cdot Z_3(t)\cdot {}^t\!q) \, t^{l + \beta\cdot B \cdot {}^{t}\!q} w^q;\\
\end{align*}
Note that this action induces a dual action of $\Gamma$ on the lattice $N\times \NN$
so that $(\gamma \cdot \tilde{n}) (\tilde{m}) = n(\gamma^*\tilde{m})$ for all $\tilde{n} \in N\times \NN$ and $\tilde{m} \in M\times \NN$.
In the dual basis $N= \{ (a,b) \in \ZZ^{g'}\times \ZZ^{r'}\}$ so that $\langle (a,b), (p,q)\rangle = (a\cdot {}^{t}\!p) + (b\cdot {}^{t}\!q)$, we have
$ (\alpha, \beta) \cdot (a,b,k) (p,q,l) =  (a,b,k) (p, q, l+ \beta\cdot B' \cdot {}^{t}\!q)$, 
hence
\begin{equation}\label{eq:actionN}
(\alpha, \beta) \cdot (a,b,k) =  (a, b + k\, \beta\cdot B',k)~.
\end{equation}
It follows that $\Gamma$ preserves the affine hyperplanes $N\times\{k\}$ for all $k$, 
acts trivially on  $N\times\{0\}$, 
and the action of the subgroup $\{0\} \times \ZZ^{g'}$ is free on  $N\times\{k\}$ for all $k\in \NN^*$. 

\smallskip

Since $\Gamma$ acts by multiplication on each element of $M\times \NN$, for any
$\gamma\in \Gamma$, the map $(z',t) \mapsto \gamma\cdot (z',t)$ on  $\GGm^g \times \DD^*$
extends as a holomorphic map from $X_\sigma$ to $X_{\gamma\cdot \sigma}$
for any (rational polyhedral strongly convex) cone $\sigma$ (in $\tilde{N}_\RR$). 

It follows that the action of $\Gamma$
extends  to the relative toric variety
$X(\Delta)$ over $\DD$ if and only if the fan $\Delta$ is $\Gamma$-invariant.

\begin{theorem}\label{theorem: Nakamura's main theorem}
There exists a $\Gamma$-invariant fan $\Delta$ such that 
the following holds: 
\begin{enumerate}
\item 
the set of rays of $\Delta$ is exactly given by $\RR_+\cdot (n,1)$
for all $n \in N$ such that $(n,1)$ belongs to the span of the 
$\Gamma$-orbit of $(0,1)$;
\item 
the action of $\Gamma$ on $X(\Delta)$ is free and $\pi_\Delta$-equivariant;
\item 
the restriction of the action of $\Gamma$ to any fiber of $\pi_\Delta$ is co-compact. 
\end{enumerate}
\end{theorem}

\begin{remark}
In order for the $\Gamma$-action to be co-compact it is necessary for the support of $\Delta$ to 
intersect $N\times\{1\}$ along an affine space. One can show that the Néron model is obtained
by taking the subfan of $\Delta$ having only one dimensional faces
generated by the $\Gamma$-orbit of $(0,1)$. 
In general $\Delta$ contains
more rays. This happens exactly when the central fiber of the Néron model is reducible. 
\end{remark}

\begin{proof}[Sketch of proof]
The proof is given in~\cite[Theorem 2.6]{Nakamura_Neron_models}. 
Let $\Pi$ be the real affine subspace generated by the $\Gamma$-orbit of $(0,1)$. 
Observe that $N_\Pi:= N\times\{1\}\cap \Pi$ is a co-compact lattice in $\Pi$. 
Let $\Delta_0$ be the fan whose cones are $(0)$ and the $1$-dimensional cones $\RR_+ (n,1)$ with $n \in \Pi\cap N\times\{1\}$.

Since $\Gamma$ acts by translation on $\tilde{N}_\RR$ preserving $N\times \NN$, we may find a scalar product (hence a metric $d$ on $\tilde{N}_\RR$)
which is $\Gamma$-invariant. 
Following~\cite[\S 1]{Alexeev_Nakamura}, we define a Delaunay cell $\sigma$ to be the closed convex hull of all elements $\tilde{n}\in N_\Pi$
for which there exists $\alpha \in \Pi$ satisfying $\inf_{N_\Pi} d(\alpha,\cdot) = d(\alpha,\tilde{n})$. If the metric is sufficiently general, then 
we obtain a $\Gamma$-invariant cell-decomposition of $\Pi$ by simplices whose
vertices are all contained in $N_\Pi$. 
 Observe that this fan is not necessarily regular if $g\ge 5$ (see~\cite[\S 1.14]{Alexeev_Nakamura}).

We define $\Delta$ to be the fan whose cones are generated by cells of this triangulation. The fact that $\Gamma$ acts co-compactly on $X(\Delta)$ is proved in~\cite{Nakamura_Neron_models}
for any dimension $g$. 
\end{proof}

A crucial observation is the following. 
\begin{proposition}\label{prop:holo-sec}
Let $\pi\colon\X\to\DD$ be any principally polarized family of abelian varieties. 
Then any holomorphic section over $\DD^*$ which is meromorphic at $0$
 extends holomorphically to $\DD$.
 
 If moreover $\X = X(\Delta)/\Gamma$ for some discrete group $\Gamma\simeq \ZZ^{g}$
 acting on $\GGm^g\times\DD$ as in ~\eqref{eq:action}, then 
 any meromorphic section $\psi\colon \DD^*\to \X$ lifts to a holomorphic
 section $\tilde\psi\colon \DD\to X(\Delta)$ having the following property.
 
 Since $X(\Delta) \supset \GGm^g \times \DD^*$,  post-composing $\tilde\psi|_{\DD^*}$ with the first projection yields a 
 meromorphic map $\phi\colon \DD^* \to \GGm^g$. 
Then $(n_\phi,1)$ as defined in Lemma~\ref{lem:one-parameter} belongs to the linear span of $\Gamma\cdot (0,1) \subset \tilde{N}_\RR$.
 \end{proposition} 
Note that $\Gamma\cdot (0,1)$ is the translate  by $(0,1)$ of a discrete group (i.e., $\{(0,\beta\cdot B'), \beta \in \ZZ¨{r'}\}$) whose rank equals the one of $B$ (i.e., $r'$), see~\eqref{eq:actionN}.
\begin{proof}
We first claim that $\psi$ extends holomorphically through $0$. 
Since $\psi$ is meromorphic, there exists a closed analytic subset $C$
of dimension $1$ in $\DD\times \X$ that contains $\{(t,\psi(t)), t\in \DD^*\}$. 
We observe that $C$ is necessarily analytically irreducible at the point $p = C\cap X_0$, 
hence $\psi$ extends continuously (hence holomorphically) at $0$. 

Choose coordinates $z=(z_1, \cdots, z_{g})$ centered at $p$, and write $\pi(z) = \sum_I a_I z^I$. 
Since $\pi(\psi(t)) = t$,  we infer that at least one $a_i \neq 0$. It follows that $\pi^{-1}(0)$ is smooth and reduced at 
$p$.

Suppose now that  $\X=X(\Delta)/\Gamma$ as given by the previous theorem. 
Then pick a point $q\in X(\Delta	)$ which is mapped to $p$ in $\X$. Since $\Gamma$ acts freely, 
the canonical map $X(\Delta) \to \X$ is a covering map, hence we can lift $\psi$ to a holomorphic map
$\tilde{\psi} \colon \DD\to \X(\Delta)$ such that $\tilde{\phi}(0) = q$. 
The projection of $\tilde{\psi}$ to $\GGm^g$ induces a holomorphic map $\phi\colon\DD^* \to \GGm^g$, which is meromorphic at $0$, 
hence by Lemma~\ref{lem:one-parameter} we deduce that $\RR_+\cdot(n_\phi,1)$ forms a one-dimensional 
cone of $\Delta$. This condition is equivalent to say that $(n_\phi,1)$ belongs to the linear span of $\Gamma\cdot (0,1) \subset \tilde{N}_\RR$.
\end{proof}
 \begin{remark}
 The proof shows a weak version of the Néron property for $X(\Delta)/\Gamma$. In fact the Néron model is given by the smooth locus of $\pi \colon X(\Delta)/\Gamma \to \DD$
 whose complement has codimension $2$ in $X(\Delta)/\Gamma$.
 \end{remark}

\subsection{Proof of Theorem \ref{Theorem: translations are regularizable}}\label{subsection: proof of Theorem C}
Let $f \colon \X \dashrightarrow \X$ be a meromorphic map of a polarized family $\pi\colon \X\to \DD$ of $g$-dimensional abelian varieties satisfying $\pi \circ f = \pi$ such that for any $t\in \DD^*$, the map $f^*_t\colon H^1(X_t, \ZZ) \to  H^1(X_t, \ZZ)$ has finite order. 

Replacing $f$ by a suitable iterate and applying Proposition~\ref{prop:iterate-regular}, we may suppose that 
$f^*_t=\id$ for all $t$ so that 
the induced map
$f_t\colon \X_t\to \X_t$
is a translation of the abelian variety $\X_t$. 

By Proposition~\ref{prop:regul-isogeny} and the discussion before Lemma~\ref{lem:6.3}, we may assume that $X_t$ is principally polarized. 
We may also replace the family $\X$ by any base change by Proposition~\ref{prop: if for thm for families over disc}, 
and by the discussion of the previous section, we may suppose that $\X = X(\Delta)/\Gamma$ 
where $\Gamma\simeq \ZZ^g$ acts on $X(\Delta) \subset \GGm^g \times \DD$ as in~\eqref{eq:action}, and $\Delta$
is a $\Gamma$-invariant fan containing $(0,1)$.

Let $\phi(t) = f_t(0)$. This is a holomorphic section over $\DD^*$ which is meromorphic at $0$ since $f$ is meromorphic. 
It follows from Theorem~\ref{theorem: Nakamura's main theorem} (1) and Proposition~\ref{prop:holo-sec} that $\phi$ lifts to a section $\tilde{\phi} \colon\DD \to X(\Delta)$, hence
 $(n_\phi, 1)$ lies in the linear span of the $\Gamma$-orbit of $(0,1)$ and we can find an integer $N$ such that 
$(N n_\phi, 1)= \gamma\cdot (0,1)$ for some $\gamma \in \Gamma$.

The map $F(z,t)= (\phi(t)\cdot z, t)$ defines a meromorphic self-map of $\GGm^g\times \DD$ lifting $f$. 
Let $\psi(t) = \widetilde{\phi}(t)^N/ \exp(\gamma)$ 
Then the vector $n_{\psi}$ defined in Lemma~\ref{lem:one-parameter} equals $0$. 
By Lemma~\ref{lem:one-parameter}, $\psi$ extends as a holomorphic map 
$\psi\colon\DD \to\GGm^g$, and $g$ defines an automorphism of $X(\Delta)$ since
any (relative) toric manifold carries an action by multiplication by $\GGm^g$.

Since $F^N$ is the composition of $g$ and the action of $\gamma$ it also induces a biholomorphism on $X(\Delta)$ which commutes with the action of 
$\Gamma$. Therefore $f^N$ is a biholomorphism on $X(\Delta)/\Gamma$, hence $f$ is regularizable by Propositions~\ref{prop:smooth-regular} and~\ref{prop:iterate-regular}.

%%%%%%%%%%%%%%%

\subsection{Orbits of translations on families of abelian varieties}
We now explore the closure of the orbit of a general point under a family of translations on a 
family of abelian varieties. The setting is as follows: 
$f\colon \X \to \X$ is an automorphism of a polarized family of abelian varieties $\L\to \X$, $\pi \colon \X\to \DD$
given by the translation along a section $\alpha\colon \DD \to \X$ of $\pi$ so that 
$f_t(z) = z + \alpha(t)$ for all $z\in X_t$ and all $t\in\DD^*$. 

For each $t$, we denote by $\tilde{P}(t):= \overline{\{n\cdot \alpha(t)\}}_{n\in\ZZ}\subset X_t$ the closure (for the euclidean topology) of the orbit of $0$, and 
by $P(t)$ its connected component containing $0$. Then $P(t)$ is a closed connected abelian real-subgroup of $X_t$
isomorphic to $(\RR/\ZZ)^{h(t)}$ for some $h(t) \in\{0,\cdots, 2g\}$, and $\tilde{P}(t)$ is a finite union of translates of $P(t)$. 
The group $P(t)$ also contains a maximal complex Lie group $A(t)$
of (complex) dimension $s(t)$. Write $r(t):=h(t) - 2s(t)$. 

\smallskip

Recall that a $F_\sigma$-set is  a countable union of
closed subset, and that any countable union of proper analytic subsets
is an $F_\sigma$-set having zero Lebesgue measure (hence has empty interior).

\smallskip

We start with the following observation. 
\begin{proposition}\label{prop:semicontinuity}
Let $f\colon \X \dto \X$ be a family of translations. 
Set $h(f) = \max_{t\in \DD^*} h(t)$.
\begin{enumerate}
\item
The set $\E_0:= \{t\in \DD^*, h(t) < h(f)\}$ is included in a countable union of proper analytic subsets. 
\item
There exists a real analytic family $\{P'_t\}_{t\in\DD^*}$ of real subtori of $X_t$, 
such that $P'_t\supset P(t)$ for all $t$ and 
$P'_t= P(t)$
for all $t\notin \E_0$.
\item 
There exists a real analytic family $\{\tilde{P}'_t\}_{t\in\DD^*}$ of real closed subgroups  of $X_t$, 
such that $\tilde{P}'_t\supset \tilde{P}(t)$ for all $t$ and 
$\tilde{P}'_t= \tilde{P}(t)$
for all $t\notin \E_0$.
\item 
There exists
a countable union of proper analytic subsets $\E\supset \E_0$ 
such that 
$s(t) = s(f)\in \{0, \cdots, g\}$ for all $t\notin \E$. 
\item 
There exists a subfamily $A'_t$ of abelian subvarieties of $X_t$, 
such that $A'_t= A(t)$ for all $t\notin \E$.
\end{enumerate}
\end{proposition}
\begin{proof}
We may argue locally near $t_0\in \DD^*$, and suppose that
$X_t = V/\Lambda_t$ where $V\simeq \CC^g$, and $\Lambda_t$ is a family of rank $2g$ sublattices of $V$ moving holomorphically with $t$.  
Denote by $\pi_t\colon V \to X_t$ the canonical projection. 

Fix a lattice basis $e_1, \cdots, e_{2g}$ of $\Lambda_{t_0}$ and 
denote by $e_i \colon \DD \to V$ the holomorphic map such that $e_i(t) \in \Lambda_t$ for all $t$, and $e_i(t_0) = e_i$. 
Let $e_{i,t}^\vee$ be the dual $\RR$-basis so that
$e_{i,t}^\vee (e_j(t)) = \delta_{ij}$. Observe that $e_{i,t}^\vee$ is a real linear form on $V$ which varies real analytically in $t$ (it does not vary holomorphically). 

Consider the $\QQ$-vector space
$L_t= \{q \in \QQ^{2g}, \, \sum_{j=1}^{2g} q_j e^\vee_{j,t} (\alpha(t)) \in \QQ\}$ so that
\[
P(t) := \pi_t \left(\bigcap_{q\in L_t} \left\{v\in V,  \sum_{j=1}^{2g} q_j e^\vee_{j,t}(v)= 0\right\}\right)
\] and $h(t) =2g - \dim_\QQ (L_t)$. For fixed $q, q'\in\QQ^{2g}\times \QQ$, the set $\{t\in\DD^*, \sum_{j=1}^{2g} q_j e^\vee_{j,t} (\alpha(t)) = q'\}$
is real-analytic. It follows that $L_t$ is fixed, say equal to  $L_\star$, for all $t$ outside a countable union of proper real-analytic subsets $\E_0$. 
And  $L_\star\subset L_t$ for all $t\in\DD^*$ which proves $h(t) \le h(f)$ for all $t$. This proves (1).
Observe also that $P'_t = \pi_t (\bigcap_{q\in L_\star} \{ \sum_{j=1}^{2g} q_j e^\vee_{j,t}= 0\})$
is a real-analytic family of real subtori for which (2) holds. In the same way,
$\tilde{L}_t= \{(q,q') \in \QQ^{2g+1}, \, \sum_{j=1}^{2g} q_j e^\vee_{j,t} (\alpha(t)) =q'\}$
is constant equal to $\tilde{L}_\star$ off $\E_0$, hence $\tilde{P}'_t = \pi_t (\bigcap_{(q,q')\in \tilde{L}_\star} \{ \sum_{j=1}^{2g} q_j e^\vee_{j,t}= q'\})$
is a real-analytic family of closed subgroups of $X_t$ equal to $\tilde{P}(t)$ for all $t\notin\E_0$, proving (3).

\smallskip

To any $\RR$-linear form $\ell_{q,t} = \sum_{j=1}^{2g} q_j e^\vee_{j,t}$, we attach the $\CC$-linear form 
$u_{q,t} (v) = \ell_{q,t}(v) - i \ell_{q,t} (iv) \in V^*$ so that $\ker (u_{q,t})$ is the complex linear hypersurface included in $\ker(\ell_{q,t})$,
and
\[
A(t):= 
\pi_t \left(\bigcap_{q\in L_t} \ker (u_{q,t}) \right)~.
\]
As $u_{q,t}$ varies analytically in $t$, it follows that $t\mapsto \dim_\CC\left(\bigcap_{q\in L_\star} \ker (u_{q,t})\right)$ is again constant 
off a countable union of proper analytic subsets $\E_1$ (and 
is upper-semicontinuous). It follows that $s(t) := \dim_\CC A(t)$
is constant on $\E:= \E_0\cup\E_1$ equal to some constant $s_\star$, hence (4) holds. 

\smallskip

We now prove that $\{A(t)\}_{t\notin \E}$ glue to a holomorphic family over $\DD$.  Suppose that $t_0\notin \E$, 
and pick a basis of lattice points $f_1, \cdots, f_{2s(f)}$  in $\bigcap_{q\in L_\star} \ker (u_{q,t})$. 
Then $A(t)$ is the image under $\pi_t$ of the complex vector space generated by the holomorphically
varying points $f_{1,t}, \cdots, f_{2s(f),t}$. 
It follows that $A'_t= \pi_t(\mathrm{Span}\{f_{1,t}, \cdots, f_{2s(f),t}\})$ forms a family of subabelian varieties over $\DD^*$.
Since the family $X_t$ is polarized, $A'_t$ is also polarized and extend through $0$. This proves (5). 
\end{proof}

In the remaining of the section, $\E$ will always denote the $F_\sigma$-set defined by the previous lemma. 
Our next objective is the following decomposition result. 
\begin{proposition}\label{prop:break in family}
  Let  $f\in\Aut(\X)$ be a family of translations on a family  $\pi\colon \X\to \DD$ of polarized abelian varieties.
  Then there exist two families of sub-abelian varieties $\pi_\A \colon \A \to \DD$,  $\pi_{\cB} \colon \cB \to \DD$ of $\X$, 
  two families of translations
 $f_\A \in \Aut(\A)$, and $f_{\cB} \in \Aut(\cB)$, and a meromorphic map
 $\Phi \colon \A \times \cB \dto \X$ such that: 
 \begin{enumerate}
 \item
 $\Phi$ induces an isogeny $A_t\times B_t \to X_t$ for all $t\in \DD^*$;
 \item 
 the orbit of $0$ under $f_{\A,t}$ is dense in $A_t$ for all $t\notin \E$; 
 \item 
 the closure of the orbit of $0$ under $f_{\cB,t}$ is totally real for all $t\notin \E$.
 \end{enumerate}
 \end{proposition}

\begin{proof}
We let $\A$ be the family of subabelian varieties given by Proposition~\ref{prop:semicontinuity}. 
Recall the definition of the dual family $\X^\vee$ and of the canonical meromorphic map 
$\phi\colon \X \dto \X^\vee$ from the proof of Proposition~\ref{prop:regul-isogeny}.

Define a family of subabelian varieties $B_t$
as the connected components containing $0$ of the kernel of the composition map
$X_t \mathop{\longrightarrow} X^\vee_t \to A^\vee_t \mathop{\longrightarrow} A_t \to X_t$. 
We obtain a family of polarized subvarieties over $\DD^*$ hence over $\DD$ and the first property is clear, see~\cite[Chapter 5, \S3]{Birkenhake-Lange_AV}. 

Let us fix any parameter $t\in \DD^*$. We have a canonical splitting $V= E_t \oplus F_t$ so that $\pi_t(E_t) = A_t$ and $\pi_t(F_t) = B_t$
and since $A_t, B_t$ are varying holomorphically, the same holds for the vector spaces $E_t$ and $F_t$. 
We may thus write $\alpha(t) = a(t) \oplus b(t)$ with $a, b$ holomorphic.  
Denote by $f_{\A,t}$ the translation by $a(t)$ on $A_t$ and by $f_{\cB,t}$ the translation by $b(t)$ on $B_t$.

Suppose that $t\notin \E$. Replacing $\alpha(t)$ by a suitable multiple, we may suppose
that the closure of the orbit of $0$ is a real torus so that $\widetilde{P}(t)= P(t)$. Denote by $\Pi_t$
its lift to $V$, and pick $\widetilde{\alpha}(t)$ (resp. $\widetilde{a}(t)$) a lift of $\alpha(t)$ (resp. of $a(t)$) to $V$. 
Observe that  $\Pi_t$ is the closure of $\ZZ \widetilde{\alpha}(t) + \Lambda_t$, and that $\Pi_t \cap \Lambda_t$ is a co-compact lattice in $\Pi_t$. 
Since $A_t$ is the largest subabelian variety in $P(t)$, we have 
$E_t = \Pi_t \cap i \Pi_t$, and $F_t \cap \Pi_t$ cannot contain any complex linear subspace, hence is 
totally real. 

The linear projection $p_{E,t} \colon V \to E_t$ parallel to $F_t$ 
semi-conjugates the translation by $\alpha(t)$ and the translation by $a(t)$. 
In particular, the image of $\ZZ \widetilde{\alpha}(t) + \Lambda_t$ under $p_{E,t}$
equals $\ZZ a(t) + p_{E,t}(\Lambda_t)$, so that the latter is dense in $E_t$. 
Since the addition map $A_t \times B_t \to X_t$ is an isogeny, 
the co-compact lattice $\Lambda_t \cap E_t$ has finite index in 
$p_{E,t}(\Lambda_t)$, and we conclude that $\ZZ \widetilde{a}(t) + (\Lambda_t \cap E_t)$
is dense in $E_t$. This proves (2). 

The proof of (3) is completely analogous, once one
observes that the image of $\Pi_t$ under the projection $p_{F,t} \colon V \to F_t$ parallel to $E_t$
is totally real. 
\end{proof}

The proof of Theorem~\ref{thm:orbit-closure} will be complete after we prove: 

\begin{proposition}\label{prop:totally real in family}
Let  $\pi\colon \X\to \DD$ be a family of polarized abelian varieties. 
Let  $f\in\Aut(\X)$ be a family of translations such that the closure of the orbit of $0$ under $f_{\X,t}$ is totally real for Lebesgue-almost every $t$.

Then the closure of the orbit of $0$ under $f_{\X,t}$ is totally real for all $t$, and 
there exist a proper model $\X'$ and  sequence of automorphisms $f_n\in \Aut(\X')$ of finite order such that 
 $f_n \to f$ locally uniformly on $\X'$. 
 \end{proposition}

\begin{proof}
Replacing $\alpha$ by a suitable multiple, we may (and shall) assume that $\tilde{P}(t) = P(t)$ is connected. 

We shall first argue locally near a fixed point $t_0\notin \E$, where $\E$ is the set defined in Proposition~\ref{prop:semicontinuity}.
We use the same notation as in the proof of the previous proposition. 
Fix a basis $\tilde{e}_1, \cdots, \tilde{e}_k$ for the lattice $\Pi_{t_0} \cap \Lambda_{t_0}$, and 
complete it as a maximal set of $\ZZ$-independent elements $\tilde{e}_1, \cdots, \tilde{e}_{2g}$ in $\Lambda_{t_0}$. 
 Observe that 
it may happen that $\tilde{e}_1, \cdots, \tilde{e}_{2g}$ generate only a finite index subgroup of $\Lambda_{t_0}$. 
Observe also, that the holomorphically varying vectors $\tilde{e}_{1,t}, \cdots, \tilde{e}_{k,t}$ 
are $\CC$-linearly independent because $\Pi_t$ is totally real. We may thus suppose that $\tilde{e}_1, \cdots,  \tilde{e}_{g}$
are $\CC$-linearly independent.

Since the closure of the orbit of $0$ is connected, $\widetilde{\alpha}(t) \in \Pi_t$ for all $t$, 
and we may find $\alpha_j(t)\in \RR$ such that 
$\widetilde{\alpha}(t) = \sum_{j=1}^{k} \alpha_j(t)  \tilde{e}_{j,t}$.
The map $t\mapsto \widetilde{\alpha}(t)$ is holomorphic map from $\DD^*$ to the dual space $V^*$
hence $\alpha_1, \cdots, \alpha_k$ are all constant. 

We now replace $\X$ by an isogenous family which is principally polarized, and choose the model $\X' = X(\Delta)/\Gamma$ 
described in \S\ref{sec:naka-vs-mum}. Pick any $\beta\in\RR^k$, and set 
$\beta(t)= \sum_{j=1}^{k} \beta_j \tilde{e}_{j,t}$ for $t$ close to $t_0$. 
We claim that $\beta(t)$ extends to a holomorphic section of $\X' \to \DD$
over $\DD$.

To see this, write $\pi\colon X(\Delta) \to \X'$ for the canonical projection. By analytic continuation, we may  find a lift
$s \mapsto \hat{\beta}(s)$ defined over $s\in \HH$ such that 
$\pi(\hat{\beta}(s)) = \beta(e^{2i\pi s})$. We now write $\hat{\beta}$ in terms of a holomorphic varying symplectic basis
$e_1, \cdots , e_{2g}$ as in \S\ref{sec:naka-vs-mum}: there exists some real numbers $\theta_1, \cdots, \theta_g$ and $b_1, \cdots , b_g$
such that 
\[\hat{\beta}(s)= \left(e^{i\pi \theta_1} \, \prod_{j=1}^g \be(a_{1j}(s)b_j), \cdots, 
e^{i \pi \theta_g}\, \prod_{j=1}^g\be(a_{gj}(s)b_j)\right)\]
(see~\eqref{eq:action}).
Doing the same for $\alpha$, we get some $\vartheta_j, a_j \in \RR$ such that 
\[\hat{\alpha}(s)= \left(e^{i\pi \vartheta_1} \, \prod_{j=1}^g \be(a_{1j}(s)a_j), \cdots, 
e^{i \pi \vartheta_g}\, \prod_{j=1}^g\be(a_{gj}(s)a_j)\right)\]
But $\alpha$ is a section over $\DD$ so the vector $B a$ belongs to $\ZZ^g$. 

Observe that  the function $ v \in \RR^g/\ZZ^g \mapsto B v \in  \RR^g/\ZZ^g$ is continuous, and
vanishes at all vectors $\{na\}_{n\in \NN}$. Since $P(t) = \pi_t(\Pi_t)$ is the closure of the orbit of $0$ under the translation by $\alpha(t)$
we may find a sequence $q^{(n)}\in \NN$ such that $q^{(n)}\alpha(t) \to \beta(t)$
in a neighborhood of $t_0$. This implies $q^{(n)}a \to b$ in $ \RR^g/\ZZ^g$, hence $B b = 0$ in $ \RR^g/\ZZ^g$, and $\beta$ defines
a holomorphic section over $\DD^*$. By Lemma~\ref{lem:6.3} this section is meromorphic at $0$.
By Proposition~\ref{prop:holo-sec} it extends to a holomorphic section over $\DD$.

Now choose any sequence $\beta^{(n)}\in\QQ^k$ converging to $(\alpha_1,\cdots, \alpha_k)$. 
Then $\beta^{(n)}(t)= \sum_{j=1}^{k} \beta^{(n)}_j e_{j,t}$ converges locally uniformly to $\alpha(t)$
on $\DD$, and the translation $f_n$ by $\beta^{(n)}(t)$  has finite order, which concludes the proof.
\end{proof}

 \begin{remark}
The previous proof implies that the family of real subtori $P'_t$ actually extends to a family over $\DD$. 
Recall that the smooth locus of the central fiber $\X'_0$
is a finite union of semi-abelian varieties which are principal $\GGm^{r'}$-bundles over an abelian variety (where $r'$
is the rank of the monodromy, see the discussion after Lemma~\ref{lem:6.3}).  

One can also prove that $P'_0$ is transversal to the fibers of the $\GGm^{r'}$-fibration, 
hence $s(f) \le r'$. 
\end{remark}

\bibliographystyle{alpha}
\bibliography{references.bib}

\begin{thebibliography}{BHPVdV04}

\bibitem[Abb23]{abboud}
Marc Abboud.
\newblock Sur la dynamique des endomorphismes des surfaces affines maps.
\newblock {\em Th{\`e}se de l'universit{\'e} de Rennes}, 2023.

\bibitem[AN99]{Alexeev_Nakamura}
Valery Alexeev and Iku Nakamura.
\newblock On {M}umford's construction of degenerating abelian varieties.
\newblock {\em Tohoku Math. J. (2)}, 51(3):399--420, 1999.

\bibitem[AV21]{arXiv:2112.01951}
Ekaterina Amerik and Misha Verbitsky.
\newblock Parabolic automorphisms of hyperk{\"a}hler manifolds.
\newblock {\em arXiv:2112.01951}, 2021.

\bibitem[BC16]{MR3454379}
J\'{e}r\'{e}my Blanc and Serge Cantat.
\newblock Dynamical degrees of birational transformations of projective
  surfaces.
\newblock {\em J. Amer. Math. Soc.}, 29(2):415--471, 2016.

\bibitem[BCK14]{MR3277202}
Eric Bedford, Serge Cantat, and Kyounghee Kim.
\newblock Pseudo-automorphisms with no invariant foliation.
\newblock {\em J. Mod. Dyn.}, 8(2):221--250, 2014.

\bibitem[BD05]{MR2140266}
Eric Bedford and Jeffrey Diller.
\newblock Energy and invariant measures for birational surface maps.
\newblock {\em Duke Math. J.}, 128(2):331--368, 2005.

\bibitem[BDJK21]{bell-diller-jonsson}
Jason Bell, Jeffrey Diller, Mattias Jonsson, and Holly Krieger.
\newblock Birational maps with transcendental dynamical degree.
\newblock {\em arXiv:2107.04113}, 2021.

\bibitem[BDK15]{MR3587459}
Eric Bedford, Jeffrey Diller, and Kyounghee Kim.
\newblock Pseudoautomorphisms with invariant curves.
\newblock In {\em Complex geometry and dynamics}, volume~10 of {\em Abel
  Symp.}, pages 1--27. Springer, Cham, 2015.

\bibitem[BGPS14]{MR3222615}
Jos\'{e}~Ignacio Burgos~Gil, Patrice Philippon, and Mart\'{\i}n Sombra.
\newblock Arithmetic geometry of toric varieties. {M}etrics, measures and
  heights.
\newblock {\em Ast\'{e}risque}, (360):vi+222, 2014.

\bibitem[BHPVdV04]{MR2030225}
Wolf~P. Barth, Klaus Hulek, Chris A.~M. Peters, and Antonius Van~de Ven.
\newblock {\em Compact complex surfaces}, volume~4 of {\em Ergebnisse der
  Mathematik und ihrer Grenzgebiete. 3. Folge. A Series of Modern Surveys in
  Mathematics [Results in Mathematics and Related Areas. 3rd Series. A Series
  of Modern Surveys in Mathematics]}.
\newblock Springer-Verlag, Berlin, second edition, 2004.

\bibitem[Bin76]{MR0404678}
J\"{u}rgen Bingener.
\newblock Schemata \"{u}ber {S}teinschen {A}lgebren.
\newblock {\em Schr. Math. Inst. Univ. M\"{u}nster (2)}, page~52, 1976.

\bibitem[BK10]{MR2677899}
Eric Bedford and Kyounghee Kim.
\newblock Continuous families of rational surface automorphisms with positive
  entropy.
\newblock {\em Math. Ann.}, 348(3):667--688, 2010.

\bibitem[BK11]{MR2858166}
Eric Bedford and Kyounghee Kim.
\newblock Linear fractional recurrences: periodicities and integrability.
\newblock {\em Ann. Fac. Sci. Toulouse Math. (6)}, 20(Fascicule
  Sp\'{e}cial):33--56, 2011.

\bibitem[BK12]{MR2905001}
Eric Bedford and Kyounghee Kim.
\newblock Dynamics of rational surface automorphisms: rotation domains.
\newblock {\em Amer. J. Math.}, 134(2):379--405, 2012.

\bibitem[BK14]{MR3161509}
Eric Bedford and Kyounghee Kim.
\newblock Dynamics of (pseudo) automorphisms of 3-space: periodicity versus
  positive entropy.
\newblock {\em Publ. Mat.}, 58(1):65--119, 2014.

\bibitem[BL04]{Birkenhake-Lange_AV}
Christina Birkenhake and Herbert Lange.
\newblock {\em Complex abelian varieties}, volume 302 of {\em Grundlehren der
  mathematischen Wissenschaften [Fundamental Principles of Mathematical
  Sciences]}.
\newblock Springer-Verlag, Berlin, second edition, 2004.

\bibitem[Bla13]{MR3179687}
J\'{e}r\'{e}my Blanc.
\newblock Dynamical degrees of (pseudo)-automorphisms fixing cubic
  hypersurfaces.
\newblock {\em Indiana Univ. Math. J.}, 62(4):1143--1164, 2013.

\bibitem[BLR90]{BLR_Neron_models}
Siegfried Bosch, Werner L\"{u}tkebohmert, and Michel Raynaud.
\newblock {\em N\'{e}ron models}, volume~21 of {\em Ergebnisse der Mathematik
  und ihrer Grenzgebiete (3) [Results in Mathematics and Related Areas (3)]}.
\newblock Springer-Verlag, Berlin, 1990.

\bibitem[Can01]{MR1854708}
Serge Cantat.
\newblock Sur la dynamique du groupe d'automorphismes des surfaces {$K3$}.
\newblock {\em Transform. Groups}, 6(3):201--214, 2001.

\bibitem[CD12]{MR2904576}
Serge Cantat and Igor Dolgachev.
\newblock Rational surfaces with a large group of automorphisms.
\newblock {\em J. Amer. Math. Soc.}, 25(3):863--905, 2012.

\bibitem[CDX21]{MR4340488}
Serge Cantat, Julie D\'{e}serti, and Junyi Xie.
\newblock Three chapters on {C}remona groups.
\newblock {\em Indiana Univ. Math. J.}, 70(5):2011--2064, 2021.

\bibitem[Che05]{cheltsov}
Ivan~A. Chel'tsov.
\newblock Birationally rigid {F}ano varieties.
\newblock {\em Uspekhi Mat. Nauk}, 60(5(365)):71--160, 2005.

\bibitem[CMSP17]{MR3727160}
James Carlson, Stefan M\"{u}ller-Stach, and Chris Peters.
\newblock {\em Period mappings and period domains}, volume 168 of {\em
  Cambridge Studies in Advanced Mathematics}.
\newblock Cambridge University Press, Cambridge, 2017.
\newblock Second edition.

\bibitem[Dan20]{MR4133708}
Nguyen-Bac Dang.
\newblock Degrees of iterates of rational maps on normal projective varieties.
\newblock {\em Proc. Lond. Math. Soc. (3)}, 121(5):1268--1310, 2020.

\bibitem[DF01]{diller_favre}
Jeffrey Diller and Charles Favre.
\newblock Dynamics of bimeromorphic maps of surfaces.
\newblock {\em Amer. J. Math.}, 123(6):1135--1169, 2001.

\bibitem[DF21]{MR4276288}
Nguyen-Bac Dang and Charles Favre.
\newblock Spectral interpretations of dynamical degrees and applications.
\newblock {\em Ann. of Math. (2)}, 194(1):299--359, 2021.

\bibitem[DH22]{MR4363581}
Nguyen-Bac Dang and Thorsten Herrig.
\newblock Dynamical degrees of automorphisms on abelian varieties.
\newblock {\em Adv. Math.}, 395:Paper No. 108082, 43, 2022.

\bibitem[Dil11]{MR2825269}
Jeffrey Diller.
\newblock Cremona transformations, surface automorphisms, and plane cubics.
\newblock {\em Michigan Math. J.}, 60(2):409--440, 2011.
\newblock With an appendix by Igor Dolgachev.

\bibitem[DLOZ22]{MR4431123}
Tien-Cuong Dinh, Hsueh-Yung Lin, Keiji Oguiso, and De-Qi Zhang.
\newblock Zero entropy automorphisms of compact {K}\"{a}hler manifolds and
  dynamical filtrations.
\newblock {\em Geom. Funct. Anal.}, 32(3):568--594, 2022.

\bibitem[DN11]{Dinh_Nguyen}
Tien-Cuong Dinh and Viet-Anh Nguy\^{e}n.
\newblock Comparison of dynamical degrees for semi-conjugate meromorphic maps.
\newblock {\em Comment. Math. Helv.}, 86(4):817--840, 2011.

\bibitem[DS04]{MR2119243}
Tien-Cuong Dinh and Nessim Sibony.
\newblock Regularization of currents and entropy.
\newblock {\em Ann. Sci. \'{E}cole Norm. Sup. (4)}, 37(6):959--971, 2004.

\bibitem[DS05a]{MR2129771}
Tien-Cuong Dinh and Nessim Sibony.
\newblock Dynamics of regular birational maps in {$\mathbb{P}^k$}.
\newblock {\em J. Funct. Anal.}, 222(1):202--216, 2005.

\bibitem[DS05b]{Dinh_Sibony}
Tien-Cuong Dinh and Nessim Sibony.
\newblock Une borne sup\'{e}rieure pour l'entropie topologique d'une
  application rationnelle.
\newblock {\em Ann. of Math. (2)}, 161(3):1637--1644, 2005.

\bibitem[DS10]{MR2629598}
Tien-Cuong Dinh and Nessim Sibony.
\newblock Super-potentials for currents on compact {K}\"{a}hler manifolds and
  dynamics of automorphisms.
\newblock {\em J. Algebraic Geom.}, 19(3):473--529, 2010.

\bibitem[DTV10]{MR2752759}
Henry De~Th\'{e}lin and Gabriel Vigny.
\newblock Entropy of meromorphic maps and dynamics of birational maps.
\newblock {\em M\'{e}m. Soc. Math. Fr. (N.S.)}, (122):vi+98, 2010.

\bibitem[FC90]{MR1083353}
Gerd Faltings and Ching-Li Chai.
\newblock {\em Degeneration of abelian varieties}, volume~22 of {\em Ergebnisse
  der Mathematik und ihrer Grenzgebiete (3) [Results in Mathematics and Related
  Areas (3)]}.
\newblock Springer-Verlag, Berlin, 1990.
\newblock With an appendix by David Mumford.

\bibitem[FL17]{MR3592463}
Mihai Fulger and Brian Lehmann.
\newblock Positive cones of dual cycle classes.
\newblock {\em Algebr. Geom.}, 4(1):1--28, 2017.

\bibitem[FW12]{MR3043585}
Charles Favre and Elizabeth Wulcan.
\newblock Degree growth of monomial maps and {M}c{M}ullen's polytope algebra.
\newblock {\em Indiana Univ. Math. J.}, 61(2):493--524, 2012.

\bibitem[Giz80]{MR563788}
Marat~H. Gizatullin.
\newblock Rational {$G$}-surfaces.
\newblock {\em Izv. Akad. Nauk SSSR Ser. Mat.}, 44(1):110--144, 239, 1980.

\bibitem[Gri16]{MR3480704}
Julien Grivaux.
\newblock Parabolic automorphisms of projective surfaces (after {M}. {H}.
  {G}izatullin).
\newblock {\em Mosc. Math. J.}, 16(2):275--298, 2016.

\bibitem[KKMSD73]{KKMS}
Georges Kempf, Finn~Faye Knudsen, David Mumford, and Bernard Saint-Donat.
\newblock {\em Toroidal embeddings. {I}}.
\newblock Lecture Notes in Mathematics, Vol. 339. Springer-Verlag, Berlin-New
  York, 1973.

\bibitem[Kol07]{MR2289519}
J\'{a}nos Koll\'{a}r.
\newblock {\em Lectures on resolution of singularities}, volume 166 of {\em
  Annals of Mathematics Studies}.
\newblock Princeton University Press, Princeton, NJ, 2007.

\bibitem[Kuz22]{Regularization}
Alexandra Kuznetsova.
\newblock Regularizations of positive entropy pseudo-automorphisms.
\newblock {\em arXiv:2201.11689}, 2022.

\bibitem[LB19]{MR4030548}
Federico Lo~Bianco.
\newblock On the cohomological action of automorphisms of compact {K}\"{a}hler
  threefolds.
\newblock {\em Bull. Soc. Math. France}, 147(3):469--514, 2019.

\bibitem[Lin12]{MR3059849}
Jan-Li Lin.
\newblock Pulling back cohomology classes and dynamical degrees of monomial
  maps.
\newblock {\em Bull. Soc. Math. France}, 140(4):533--549 (2013), 2012.

\bibitem[LM22]{relative-minimal-model}
Shiji Lyu and Takumi Murayama.
\newblock The relative minimal model program for excellent algebraic spaces and
  analytic spaces in equal characteristic zero.
\newblock {\em arXiv:2209.08732}, 2022.

\bibitem[LU21]{MR4340723}
Anne Lonjou and Christian Urech.
\newblock Actions of {C}remona groups on {${\rm CAT}(0)$} cube complexes.
\newblock {\em Duke Math. J.}, 170(17):3703--3743, 2021.

\bibitem[McM07]{MR2354205}
Curtis~T. McMullen.
\newblock Dynamics on blowups of the projective plane.
\newblock {\em Publ. Math. Inst. Hautes \'{E}tudes Sci.}, (105):49--89, 2007.

\bibitem[Mum08]{MR2514037}
David Mumford.
\newblock {\em Abelian varieties}, volume~5 of {\em Tata Institute of
  Fundamental Research Studies in Mathematics}.
\newblock Published for the Tata Institute of Fundamental Research, Bombay; by
  Hindustan Book Agency, New Delhi, 2008.
\newblock With appendices by C. P. Ramanujam and Yuri Manin, Corrected reprint
  of the second (1974) edition.

\bibitem[Nak77]{Nakamura_Neron_models}
Iku Nakamura.
\newblock Relative compactification of the {N}\'{e}ron model and its
  application.
\newblock In {\em Complex analysis and algebraic geometry}, pages 207--225.
  1977.

\bibitem[NW14]{MR3156076}
Junjiro Noguchi and J\"{o}rg Winkelmann.
\newblock {\em Nevanlinna theory in several complex variables and {D}iophantine
  approximation}, volume 350 of {\em Grundlehren der mathematischen
  Wissenschaften [Fundamental Principles of Mathematical Sciences]}.
\newblock Springer, Tokyo, 2014.

\bibitem[Oda88]{MR922894}
Tadao Oda.
\newblock {\em Convex bodies and algebraic geometry}, volume~15 of {\em
  Ergebnisse der Mathematik und ihrer Grenzgebiete (3) [Results in Mathematics
  and Related Areas (3)]}.
\newblock Springer-Verlag, Berlin, 1988.
\newblock An introduction to the theory of toric varieties, Translated from the
  Japanese.

\bibitem[OP11]{MR2904995}
Keiji Oguiso and Fabio Perroni.
\newblock Automorphisms of rational manifolds of positive entropy with {S}iegel
  disks.
\newblock {\em Atti Accad. Naz. Lincei Rend. Lincei Mat. Appl.},
  22(4):487--504, 2011.

\bibitem[OT15]{MR3329200}
Keiji Oguiso and Tuyen~Trung Truong.
\newblock Explicit examples of rational and {C}alabi-{Y}au threefolds with
  primitive automorphisms of positive entropy.
\newblock {\em J. Math. Sci. Univ. Tokyo}, 22(1):361--385, 2015.

\bibitem[Rei03]{MR1972204}
Irving Reiner.
\newblock {\em Maximal orders}, volume~28 of {\em London Mathematical Society
  Monographs. New Series}.
\newblock The Clarendon Press, Oxford University Press, Oxford, 2003.
\newblock Corrected reprint of the 1975 original, With a foreword by M. J.
  Taylor.

\bibitem[Res17]{MR3619745}
Paul Reschke.
\newblock Salem numbers and automorphisms of abelian surfaces.
\newblock {\em Osaka J. Math.}, 54(1):1--15, 2017.

\bibitem[RR17]{MR3748233}
Paul Reschke and Bar Roytman.
\newblock Lower semi-continuity of entropy in a family of {K}3 surface
  automorphisms.
\newblock {\em Rocky Mountain J. Math.}, 47(7):2323--2349, 2017.

\bibitem[RT18]{MR3862056}
Erwan Rousseau and Fr\'{e}d\'{e}ric Touzet.
\newblock Curves in {H}ilbert modular varieties.
\newblock {\em Asian J. Math.}, 22(4):673--689, 2018.

\bibitem[Shi63]{MR156001}
Goro Shimura.
\newblock On analytic families of polarized abelian varieties and automorphic
  functions.
\newblock {\em Ann. of Math. (2)}, 78:149--192, 1963.

\bibitem[Sil94]{MR1312368}
Joseph~H. Silverman.
\newblock {\em Advanced topics in the arithmetic of elliptic curves}, volume
  151 of {\em Graduate Texts in Mathematics}.
\newblock Springer-Verlag, New York, 1994.

\bibitem[Sug23]{sugimoto}
Yutaro Sugimoto.
\newblock Dynamical degrees of automorphisms of complex simple abelian
  varieties and salem numbers.
\newblock {\em arXiv:2302.02271}, 2023.

\bibitem[Tru14]{MR3255693}
Tuyen~Trung Truong.
\newblock The simplicity of the first spectral radius of a meromorphic map.
\newblock {\em Michigan Math. J.}, 63(3):623--633, 2014.

\bibitem[Tru20]{MR4048444}
Tuyen~Trung Truong.
\newblock Relative dynamical degrees of correspondences over a field of
  arbitrary characteristic.
\newblock {\em J. Reine Angew. Math.}, 758:139--182, 2020.

\bibitem[Ueh16]{MR3498924}
Takato Uehara.
\newblock Rational surface automorphisms preserving cuspidal anticanonical
  curves.
\newblock {\em Math. Ann.}, 365(1-2):635--659, 2016.

\bibitem[vdG88]{MR930101}
Gerard van~der Geer.
\newblock {\em Hilbert modular surfaces}, volume~16 of {\em Ergebnisse der
  Mathematik und ihrer Grenzgebiete (3) [Results in Mathematics and Related
  Areas (3)]}.
\newblock Springer-Verlag, Berlin, 1988.

\bibitem[Vig14]{MR3330918}
Gabriel Vigny.
\newblock Hyperbolic measure of maximal entropy for generic rational maps of
  {$\mathbb{P}^k$}.
\newblock {\em Ann. Inst. Fourier (Grenoble)}, 64(2):645--680, 2014.

\bibitem[Voi21]{MR4279905}
John Voight.
\newblock {\em Quaternion algebras}, volume 288 of {\em Graduate Texts in
  Mathematics}.
\newblock Springer, Cham, [2021] \copyright 2021.

\bibitem[W{\l}o09]{MR2500573}
Jaros{\l}aw W{\l}odarczyk.
\newblock Resolution of singularities of analytic spaces.
\newblock In {\em Proceedings of {G}\"{o}kova {G}eometry-{T}opology
  {C}onference 2008}, pages 31--63. G\"{o}kova Geometry/Topology Conference
  (GGT), G\"{o}kova, 2009.

\bibitem[Yos80]{MR623443}
Hisao Yoshihara.
\newblock Structure of complex tori with the automorphisms of maximal degree.
\newblock {\em Tsukuba J. Math.}, 4(2):303--311, 1980.

\bibitem[Yos95]{MR1317527}
Hisao Yoshihara.
\newblock Quotients of abelian surfaces.
\newblock {\em Publ. Res. Inst. Math. Sci.}, 31(1):135--143, 1995.

\end{thebibliography}
\end{document}